%% file: FMFinsler.tex
\documentclass[11pt]{article}
\usepackage{epsf}
\usepackage{epsfig}
\usepackage{enumerate}
\usepackage{url}

\usepackage{amsmath,amsfonts,amssymb,amsthm}
\usepackage{wasysym}
\usepackage[retainorgcmds]{IEEEtrantools}

\def\pathPic{.}
\input{StdCommands.tex}

\def\un{{\mathbf 1}}

\DeclareMathOperator\length{length}

\DeclareMathOperator\dist{d}
\DeclareMathOperator\distC{D}
\DeclareMathOperator\cotan{cotan}
\DeclareMathOperator\arcsinh{arcsinh}

\def\htheta{{\hspace{0.01cm}\theta}}
\def\gF{{\mathfrak F}}

\def\offset{u}
\def\trial{trial}
\def\accepted{accepted}

\usepackage{algpseudocode}
\usepackage{algorithm}

\begin{document}
\title{
Efficient Fast Marching with Finsler metrics.
%Anisotropic Fast Marching on cartesian grids, \\using Adaptive Stencil Refinement.
%Efficient and Accurate \\Anisotropic Fast Marching on cartesian grids.
%Fast Marching using Anisotropic Stencil Refinement
} 
\author{Jean-Marie Mirebeau\footnote{CNRS, University Paris Dauphine, UMR 7534, Laboratory CEREMADE, Paris, France.}}
\maketitle
\date{}
\begin{abstract}

We study the discretization of the Escape Time problem: find the length of the shortest path joining an arbitrary point $z$ of a domain $\Omega$, to the boundary $\partial \Omega$. Path length is measured locally via a Finsler metric, potentially asymmetric and strongly anisotropic. 
This optimal control problem can be reformulated as a static Hamilton-Jacobi partial differential equation, or as a front propagation model. It has numerous applications, ranging from motion planning to image segmentation.

We introduce a new algorithm, Fast Marching using Anisotropic Stencil Refinement (FM-ASR), which addresses this problem on a two dimensional domain discretized on a cartesian grid. The local stencils used in our discretization are produced by arithmetic means, like in the FM-LBR \cite{M12}, a method previously introduced by the author in the special case of Riemannian metrics.
The complexity of the FM-ASR, in an average sense over all grid orientations, only depends (poly-)logarithmically on the anisotropy ratio of the metric, while most alternative approaches have a polynomial dependence.
Numerical experiments show, in several occasions, that the accuracy/complexity compromise is improved by an order of magnitude or more.
%a complexity reduction by an order or magnitude or more to meet a prescribed error bound.
%ratio by an order of magnitude.
\end{abstract}

%\tableofcontents

\section*{Introduction}

The Escape Time $\distC(z)$, from a point $z$ of the domain $\Omega$, is the length of the shortest path joining this point to the boundary $\partial \Omega$. Computing the escape time, and extracting an associated minimal path, is a task of obvious interest in motion planning control problems \cite{AltonMitchell12}. Yet this versatile problem has numerous other applications \cite{SethBook96}, including image classification \cite{PPKC10}, seismic imaging \cite{SV03} or the modeling of bio-physical phenomena \cite{SKD07}.
We are motivated by medical image segmentation problems, which often involve a strongly anisotropic \cite{BC10}, and potentially asymmetric \cite{MPAT08,ZSN09}, local measure of path length. 

From a theoretical point of view, the Escape Time problem can be reformulated as a static Hamilton-Jacobi, or Anisotropic Eikonal, Partial Differential Equation (PDE) \cite{SethBook96}. Its numerical discretization has attracted an important research effort, and includes the Fast Marching algorithm \cite{T95}, the Fast Sweeping method \cite{TsaiChengOsherZhao03}, and their numerous variants \cite{M12,AltonMitchell12,SV03,BR06}. As the ``Fast'' adjective indicates, performance is a crucial concern: in image processing applications, the discretization domain may contain millions of points (as many as image pixels), and CPU time should remain compatible with user interaction. Last but not least, as mentioned above, state of the art image processing applications involve strongly non-uniform, anisotropic and/or asymmetric measures of path length, which challenges available algorithms  \cite{BC10} and limits the parallelization potential \cite{RS09}.

This paper is devoted the introduction and study of a new algorithm, Fast Marching using Anisotropic Stencil Refinement (FM-ASR), a numerical solver for the two dimensional Escape Time problem discretized on a cartesian grid. Path length is measured locally through a given arbitrary Finsler metric $\cF$: a continuous map associating to each point $z\in \Omega$ an asymmetric norm $\cF_z$. The FM-ASR regards the discretization grid as a subset of the Lattice $\Z^2$, and uses arithmetic tools to produce the local stencils involved in the discretization of the associated Partial Differential Equation (PDE), which results in a huge complexity reduction in comparison with more classical approaches.
Note that the FM-LBR \cite{M12}, previously introduced by the author, shares this approach but is limited to metrics of Riemannian type (elliptic anisotropy). Non-Riemannian metrics arise in applications which take advantage of their potential asymmetry \cite{ZSN09,MPAT08}, or as the result of the homogenization of smaller scale Riemannian metrics \cite{OTV09}. 
The anisotropy ratios of an asymmetric norm $F : \R^2 \to \R_+$, and of a Finsler metric $\cF : \overline \Omega \times \R^2 \to \R_+$, are defined by 
\be
\label{defKappa}
\kappa(F) := \max_{|u|=|v|=1} \frac {F(u)}{F(v)}, \qquad \kappa(\cF) := \sup_{z\in \overline \Omega} \kappa(\cF_z).
\ee
The average complexity of the FM-ASR only depends (poly-)logarithmically on the anisotropy ratio of the given metric $\cF$, and is quasi-linear in the number $N$ of discretization points. In contrast, alternative approaches show a polynomial dependence either on $\kappa(\cF)$ \cite{AltonMitchell12,SV03}, or on $N$ \cite{BR06}, a difference clearly apparent in the numerical experiments presented in \S3.
%while many alternative approaches show a polynomial dependence \cite{AltonMitchell12,SV03}; a .
%Denoting by $\kappa$ the anisotropy ratio of the metric $\cF$ (see \iref{defKappa} below), and by $N$ the cardinality of the discrete domain, the complexity of the FM-ASR is only $\cO(N \ln^3 \kappa(\cF) + N \ln N)$.
In average over all grid orientations, and denoting $\ln^\alpha x := (\ln x)^\alpha$, the complexity of the FM-ASR is only $\cO(N \ln^3 \kappa(\cF) + N \ln N)$. 
%(We denote $\ln^\alpha (x) := (\ln x)^\alpha$ for $\alpha\geq 0$, $x \geq 1$.)

\section{Description of the problem, algorithm, and main results}

The Escape Time problem is posed on a two dimensional bounded domain $\Omega \subset \R^2$, equipped with a Finsler metric $\cF$. This metric is a continuous map $\cF: \overline \Omega \times \R^2 \to \R_+$, $(z,u) \mapsto \cF_z(u)$, such that for each fixed $z\in \overline \Omega$, the restriction $u \mapsto \cF_z(u)$ is an asymmetric norm (i.e.\ a proper $1$-homogeneous convex function\footnote{%
Finsler metrics are often assumed to be smooth, and the local asymmetric norms to be strictly convex. 
These assumptions, tailored for the study of minimal paths, are not required in our analysis of the Escape Time problem.
%These asymmetric norms are often assumed to be strictly convex, which guarantees the uniqueness  While this is required to guarantee the
}).
%The Anisotropic Eikonal Equation is a Partial Differential Equation (PDE), which characterizes the solution of an optimal control problem: the escape time from a given domain $\Omega$, a compact subset of the plane in this paper, equipped with a Finsler metric $\cF$. This metric is an arbitrary continuous function $\cF : \overline \Omega \times \R^2 \to \R_+$, $(z,u) \mapsto \cF_z(u)$, such that for each fixed $z\in \overline \Omega$, the restriction $u \mapsto \cF_z(u)$ is an asymmetric norm (i.e. a proper homogeneous convex function). 
The length of a path $\gamma\in C^1([0,1], \overline \Omega)$ is measured through the metric $\cF$:
$$
\length(\gamma) := \int_0^1 \cF_{\gamma(t)}( \gamma'(t)) \, dt.
$$
Notable special cases include Isotropic metrics: $\cF_z(u) = n(z)\|u\|$, where the parameter $n(z)>0$ corresponds to the local index in geometrical optics.  Riemannian metrics have the form: $\cF_z(u) := \sqrt{\<u, \cM(z) u\>}$, where $\cM(z)$ is a symmetric positive definite matrix. 
Symmetric Finsler metrics are subject to the condition $\cF_z(-u) = \cF_z(u)$, for all $z\in \Omega$, $u \in \R^2$. See Figure \ref{fig:MetricTypes}. Here and below we denote by $\| \cdot \|$ and $\<\cdot,\cdot\>$ the canonical euclidean norm and scalar product on $\R^2$.

The length of a path $\gamma \in C^1([0,1], \overline \Omega)$, and of the reversed path $\hat \gamma : t \mapsto \gamma(1-t)$ may be different in the case of a  general asymmetric Finsler metric. %In the case of a general asymmetric Finsler metric $\cF$, the reversed path $t \mapsto \gamma(1-t)$ may have a different length. 
This apparent oddity is entirely relevant in the study of motion planning under the influence of wind \cite{AltonMitchell12}. It is also essential in minimal path based image segmentation methods \cite{ZSN09,MPAT08}, where the right and left of the path should have different prescribed characteristics, since they respectively correspond to the foreground and background of the segmented object.
We introduce an asymmetric distance $\distC(\cdot, \cdot)$ on $\overline \Omega$
$$
\distC(x,y) := \inf \{ \length(\gamma); \, \gamma\in C^1([0,1], \overline \Omega), \, \gamma(0)=x,\, \gamma(1)=y\}.
$$
The solution of the Escape Time optimal control problem is the distance $\distC(\cdot)$ to the boundary:
%The seeked solution of the 
%The Escape Time optimal control problem, and the Anisotropic Eikonal Equation PDE, characterize the distance $\distC(\cdot)$ to the boundary: 
for all $x\in \overline \Omega$ % $\partial \Omega$
%This paper's objective is to provide efficient numerical methods for computing the length $\distC(x)$ of the shortest path joining an arbitrary point $x\in \overline \Omega$ to the domain's boundary $\partial \Omega$. 
\be
\label{EscapeTime}
\distC(x) := \min\{\distC(x,y); \, y\in \partial \Omega\}. %\inf_{y \in \partial \Omega} \distC(x,y).
\ee
The function $D$ is also characterized as the unique viscosity solution \cite{L82} of the static Hamilton-Jacobi, or Anisotropic Eikonal, PDE (see e.g.\ \cite{OTV09} for a discussion on this reformulation)
\be
\label{eikonal}
\left\{
\begin{array}{rl}
\cF^*_z(-\nabla \distC(z)) = 1 & \text{for all } z \in \Omega,\\
\distC(z) = 0 & \text{for all } z\in \partial \Omega.
\end{array}
\right.
\ee
%(*sign*)
In the above equation, we denoted by $F^*$ the dual asymmetric norm of an asymmetric norm $F$ on $\R^2$, which is defined for all $u\in \R^2$ by
\be
\label{def:dual}
F^*(u) := \max_{v\neq 0} \frac{\<u,v\>}{F(v)}.
\ee
Consider the front defined by $\cE_t := D^{-1}(\{t\})$, $t \geq 0$, thus $\cE_0 = \partial \Omega$. The normal to this front, at a point $z \in \cE_t$ where $\distC$ is differentiable, is positively collinear to $\nabla \distC(z)$. The speed of the front along in this normal direction is inversely proportional to the gradient euclidean norm, $1/\|\nabla \distC(z)\|$, and is thus determined by the identity $\cF_z^*(-\nabla \distC(z)) = 1$. % , and propagates at the speed $v = \nabla D(z)$ (where this function is differentiable) which has  a unit dual norm $\cF^*_z(-v)=\cF^*_z(-\nabla \distC(z))=1$. 
Note that the front may only go forward, and that the front speed cannot depend on global or high order properties of the front, such as its curvature. See \cite{SethBook96} for the applications, and limits, of this elementary front propagation model.\\

Since $\distC(\cdot, \cdot)$ is a path length (asymmetric) distance, one has for any point $x$ and neighborhood $V$, $x\in V \subset \Omega$, the identity 
\be
\label{pathLength}
\distC(x) = \min_{y \in \partial V} \distC(x,y)+\distC(y).
\ee
Indeed, any path $\gamma$ joining $x$ to $\partial \Omega$ must cross $\partial V$ at least once, at some point $y$.
The discretization of the Escape time problem is based on an approximation of the right hand side of \iref{pathLength}, the so-called Hopf-Lax update operator introduced in \cite{KushnerDupuis92}, see also \cite{SV03,BR06,M12}, and on a reinterpretation of this equation as a fixed point problem.

%More precise ?

\begin{figure}
\begin{center}
\begin{tabular}{ccc}
\hspace{-0.6cm}
{\raise 0.1cm \hbox{
\includegraphics[width=5cm]{\pathPic/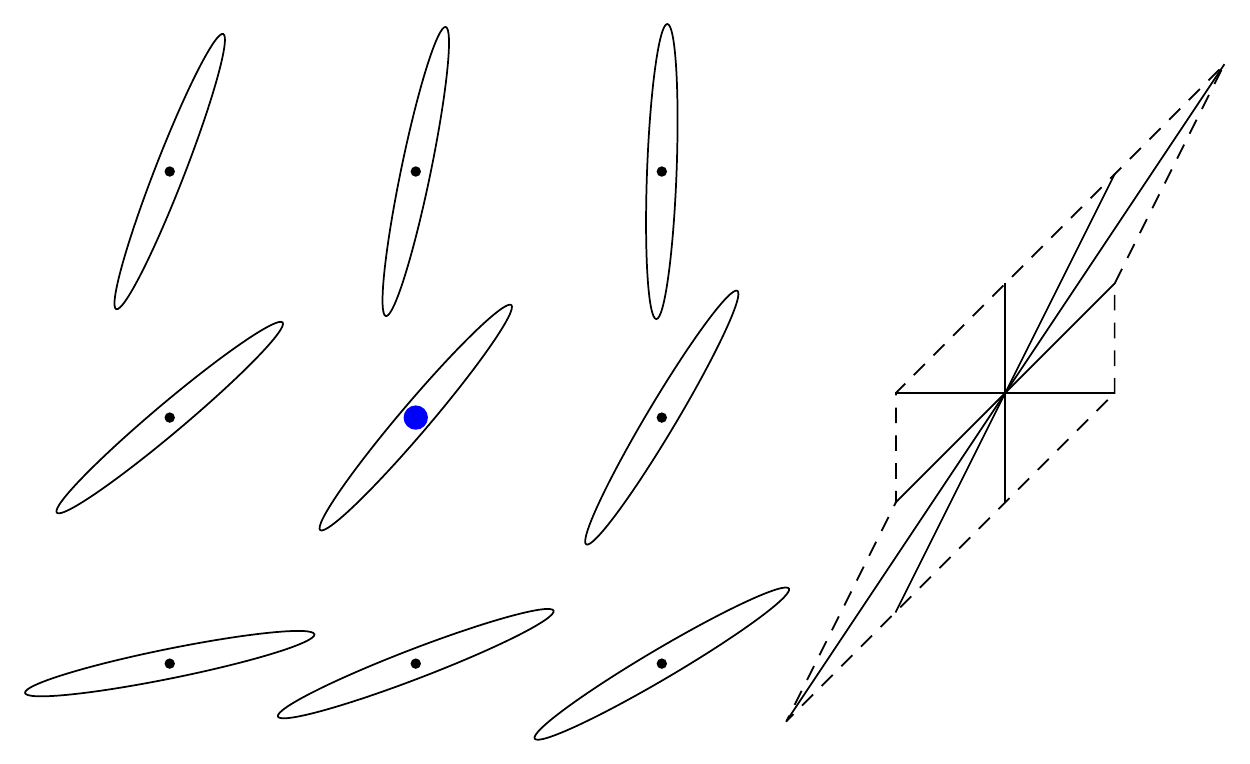}
}}
 &
\includegraphics[width=5cm]{\pathPic/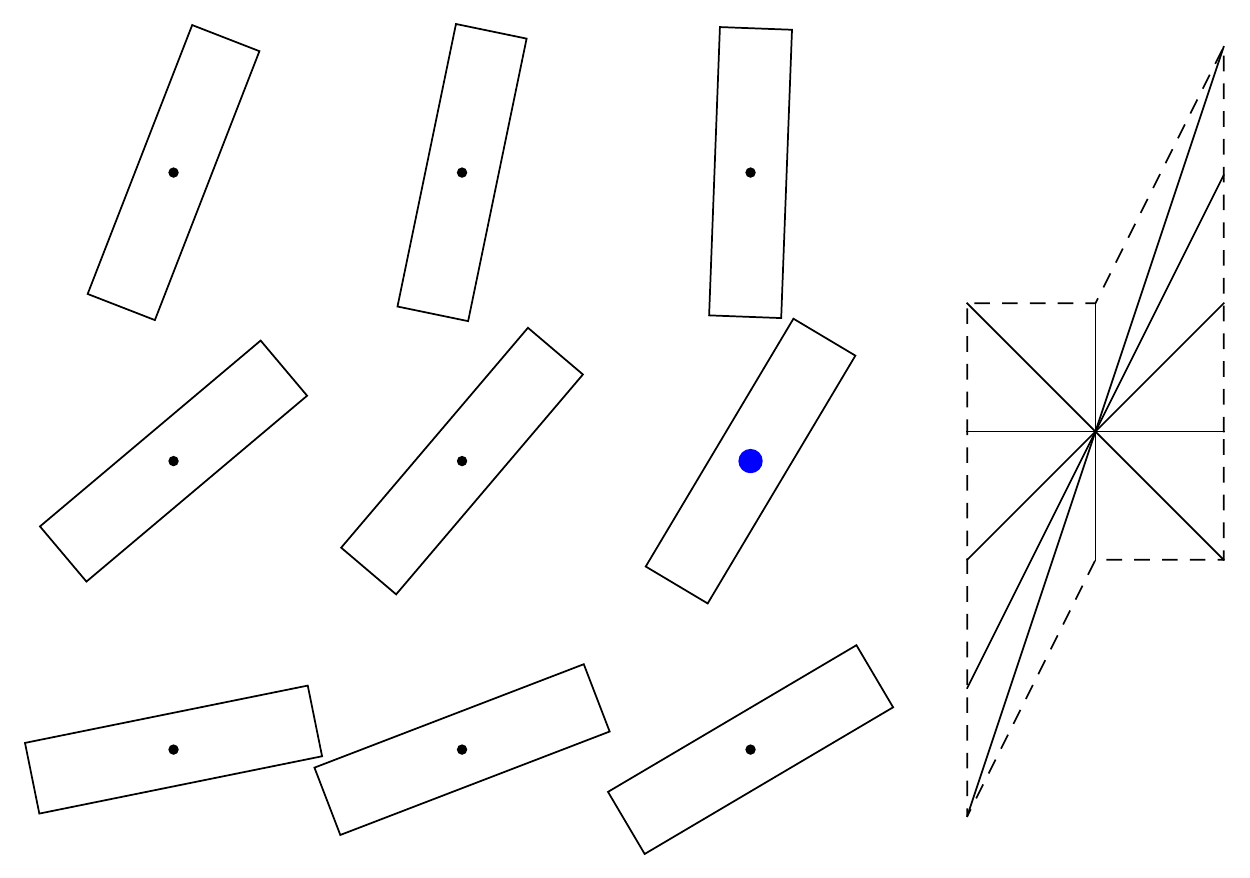} &
\hspace{0.2cm}
{\raise 0.3cm \hbox{
\includegraphics[width=5cm]{\pathPic/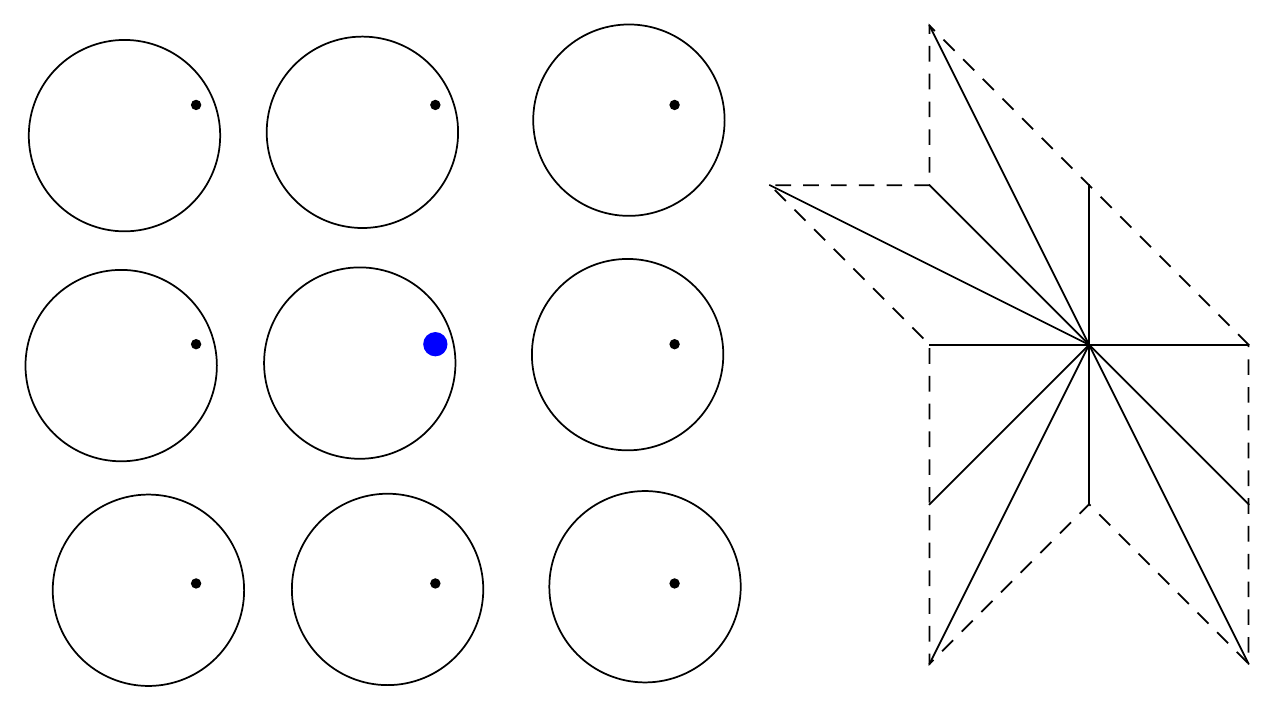}
}}
\end{tabular}
\end{center}
\caption{
\label{fig:MetricTypes}
A Finsler metric $\cF$, on a domain $\Omega\subset \R^2$, is the data of a continuously varying asymmetric norm $\cF_z$, at each point $z\in \overline \Omega$. 
The convex sets $\{u ; \, \cF_z(u) \leq 1\}$, at several points $z\in \Omega$, %, for each point $z$ of the discretization grid, 
are used to visualize the metric $\cF$.
Finsler metrics can be of Riemannian type (left), symmetric (center), or asymmetric (right). 
The discretization of the Escape Time problem involves the construction of local stencils, three those produced by the FM-ASR are illustrated. % (at a smaller scale).
%We describe in this paper a discretization of the eikonal equation associated to an arbitrary Finsler metric, which involves the construction of local meshes.
}
\end{figure} %Riemannian metrics (left) are a parti

For that purpose we introduce discrete sets $\Omega_*$ and $\partial \Omega_*$, devoted to the sampling of the continuous domain $\Omega$ and of its boundary $\partial \Omega$ respectively. In the FM-ASR, $\Omega_*$ and $\partial \Omega_*$ need to be subsets of the grid $\Z^2$, or of another orthogonal grid obtained by rescaling, rotating and offsetting $\Z^2$.
%We associate to each $z\in \Omega_*$ a small neighborhood $V(z)$, the stencil, \emph{under the form of a triangulation with its vertices on $\Omega_* \cup \partial \Omega_*$}.
A small neighborhood $V_*(z)$ of each $z\in \Omega_*$, the stencil, is constructed \emph{under the form of a triangulation, of vertices in $\Omega_* \cup \partial \Omega_*$}. See Figures \ref{fig:MetricTypes} and \ref{fig:KappaTheta} for some stencils used in the FM-ASR, and Figure \ref{fig:classical} for more classical examples\footnote{The stencil construction of the AGSI and of the MAOUM requires a mesh of the underlying discrete domain $\Omega_*$, here a subset of $h\Z^2$ for some $h>0$. We triangulated this grid with rescaled translates of the triangle of vertices $(0,0),(1,0),(0,1)$, and of its symmetric with respect to the origin.}.
For any discretization point $x\in \Omega_*$, and any discrete map $\dist : \Omega_* \cup \partial \Omega_* \to \R_+$, we define the Hopf-Lax update
%The Hopf-Lax update operator is defined by 
\be
\label{defLambda}
\Lambda(\dist,x) := \min_{y \in \partial V} \cF_x(y-x)+\interp_{V} \dist(y).
\ee
We denoted by $V$ the stencil $V_*(x)$, and by $\interp_V$ the piecewise linear interpolation operator on this triangulation. Note that $\Lambda(\dist,x)$ does not depend on the value of $\dist(x)$, but only on $\dist(y)$ for points $y$ of the discrete domain $\Omega_* \cup \partial \Omega_*$ which lie on the boundary of the stencil $V=V_*(x)$. In the following we set $\overline \R_+ := \R_+ \cup \{+\infty\} = [0, + \infty]$, adopt the convention $0 \times \infty =0$, and allow discrete maps to take the value $+\infty$. 

Numerical methods for the Escape Time problem construct a discrete approximation $\dist : \Omega_* \cup \partial \Omega_* \to \overline \R_+$ of the continuous solution $\distC$ of \iref{eikonal}, characterized by the following discrete fixed point problem:
\be
\label{eikonal_disc}
\left\{
\begin{array}{ll}
\dist(z) = \Lambda(\dist,z) & \text{for all } z\in \Omega_*,\\
\dist(z) = 0 & \text{for all } z\in \partial \Omega_*.
\end{array}
\right.
\ee

\begin{figure}
\hspace{1.25cm}
\includegraphics[width=2.5cm]{\pathPic/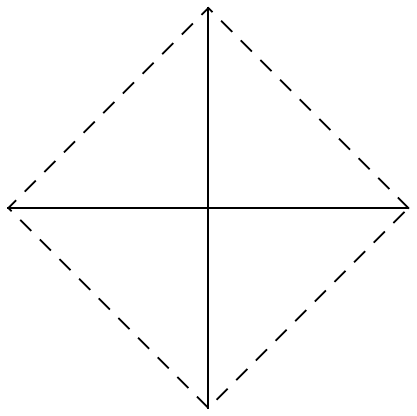}
\includegraphics[width=2.5cm]{\pathPic/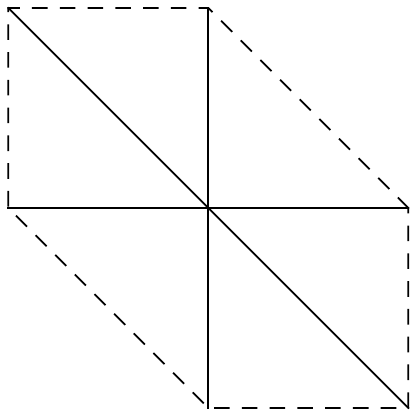}
\includegraphics[width=2.5cm]{\pathPic/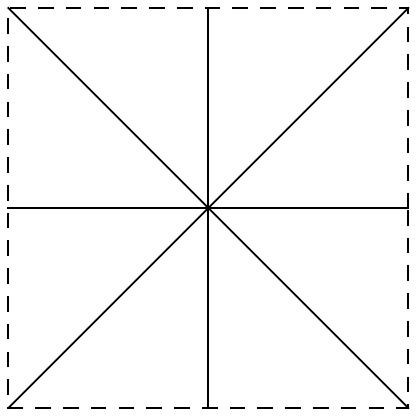}
\includegraphics[width=2.5cm]{\pathPic/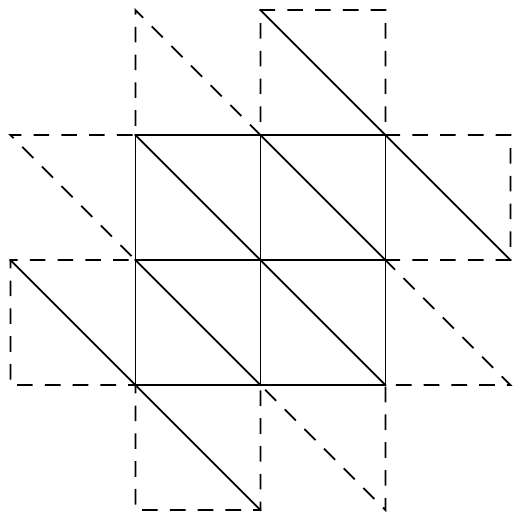}
\includegraphics[width=2.5cm]{\pathPic/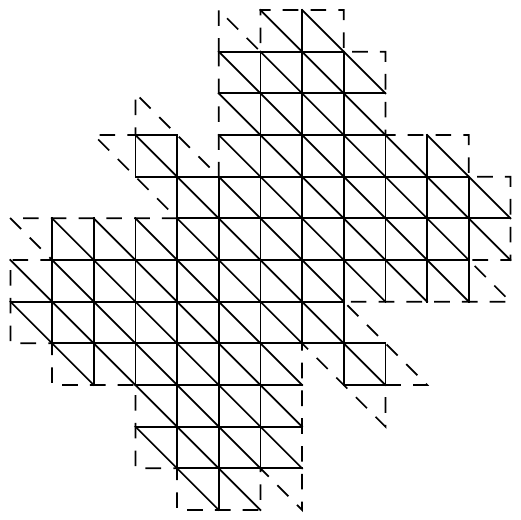}
\caption{
\label{fig:classical}
Stencil used in the classical Fast Marching algorithm (first), the AGSI \cite{BR06} (second), the FM-8 (third), and the MAOUM \cite{AltonMitchell12} at a grid point $z\in \Omega_*$ when the local anisotropy ratio $\kappa(\cF_z)$ is $1.5$ or $6$ (fourth and fifth, respectively). Algorithms compared in \S \ref{sec:num} to the FM-ASR.
%4 neigh, 6 neigh, 8 neigh. MAOUM
}
\end{figure}
 
Note that the distance on a weighted graph obeys a system of equations of similar nature, except that the neighborhoods $V_*(z)$, of each vertex $z$, are given by the graph structure and are not dependent on the numerical method. Two well known algorithms can be used to solve this system and evaluate graph distances: the fast, single pass, Dijkstra algorithm, and the slower but more flexible (in that negative edge weights are allowed) Bellman-Ford algorithm. The algorithms used to solve the system \iref{eikonal_disc}, associated to the Escape Time problem, are inspired by these two methods, and the lack of negative edge weights in the graph setting is translated into a geometrical property of the stencils, named the Causality Property, see below.

Bellman-Ford inspired algorithms solve the system \iref{eikonal_disc} via Gauss-Siedel iteration: the replacement rule $\dist(z_k) ::= \Lambda(\dist,z_k)$, $k\geq 0$, is applied repeatedly to a mutable map $\dist: \Omega_*\cup\partial\Omega_* \to \overline \R_+$, until a prescribed convergence criterion is met. The map $\dist$ is initialized to $+\infty$ on $\Omega_*$, and $0$ on $\partial \Omega_*$. The choice of the sequence of points $z_k \in \Omega_*$, $k \geq 0$, depends on the method. This sequence enumerates the lines and columns of $\Omega_*$ in the fast sweeping methods \cite{TsaiChengOsherZhao03}, and is obtained via a priority queue in the Adaptive Gauss Siedel Iteration (AGSI) \cite{BR06}. The stencils are usually extremely simple, see Figure \ref{fig:classical}, left and center left. 
The complexity of these methods is linear in $N := \#(\Omega_*)$ in the special case of an \emph{Isotropic} metric, $\cO(\lambda(\cF) N)$ for the Fast Sweeping \cite{Zhao05}, but is polynomial in general, $\cO(\mu(\cF)N^{3/2})$ for the AGSI \cite{BR06}. The constants $\lambda(\cF)$ and $\mu(\cF)$ depend on global geometrical features of the metric.
The AGSI is popular, simple and quite efficient; it appears for reference in our numerical experiments.

We next introduce some geometrical concepts, and the Causality Property which is at the foundation of Dijkstra inspired solvers of the Escape Time problem: the Fast-Marching algorithm \cite{T95}, and its variants \cite{SV03,AltonMitchell12,M12}. When satisfied, this property allows to ``decouple'' and solve the discrete system \iref{eikonal_disc} in a non-iterative, single pass fashion, resulting in a complexity  independent of global features of the metric, and quasi-linear in the number $N$ of unknowns.

\begin{definition}
\label{def:acute}
Let $F$ be an asymmetric norm on $\R^2$. We say that two vectors $u,v\in \R^2\sm\{0\}$ form an $F$-acute angle if 
\be
\label{eq:acute}
F(u+\delta v) \geq F(u) \stext{ and } F(v+\delta u) \geq F(v) \ \text{ for all } \delta \geq 0.
\ee
\label{def:acuteMesh}
%Let $F$ be an asymmetric norm on $\R^2$ and let $\cT$ be a finite conforming triangulation. 
We say that a finite conforming triangulation $\cT$ is $F$-acute if
\begin{enumerate}[(i)]
\item The union of the triangles $T \in \cT$ is a neighborhood of the origin.
\item The vertices of each $T\in \cT$ lie on $\Z^2$, one of them is the origin $0$, and $T$ has area $1/2$.
\item The non-zero vertices of each triangle $T \in \cT$ form an $F$-acute angle.
\end{enumerate}
%All the elements of $\cT$ are $F$-acute simplices. $u,v$ 
\end{definition}
In other words, two vectors form an $F$-acute angle if adding a positive multiple of one to the other increases its $F$ norm.  
The stencils $V_*(z)$ of the FM-ASR, at a point $z\in \Omega_*$, are built from (translated, rescaled, rotated) $F$-acute triangulations \eqref{defVz}. Condition (i) heuristically ensures that information is propagated in all directions in \iref{eikonal_disc}. Condition (ii) ensures that this information stays on the grid $\Omega_*$. In addition, this condition implies that a triangle $T\in \cT$ does not contain any point of $\Z^2$ except its vertices, which heuristically ensures that information does not ``jump over'' a subset of $\Omega_*$.

The core of this paper is devoted to the construction and study of an $F$-acute mesh $\cT(F)$, defined for each asymmetric norm $F$, see Figure \ref{fig:MetricTypes}, \ref{fig:KappaTheta}, and used to assemble the stencils of the FM-ASR. This mesh is produced by the following algorithm.
\begin{algorithm}
\label{algo:TF}
\textbf{Construction of the mesh $\cT(F)$}, associated to a given asymmetric norm $F$.\\ This mesh is star shaped with respect to the origin, see Figure \ref{fig:MetricTypes}. 
The sequence $L$ of its consecutive boundary vertices is generated as follows, using only two lists $L$ and $M$.\\
\begin{tabular}{llll}
 & \multicolumn{3}{l}{\textbf{Set} $L::=[(1,0)]$, $M::=[(1,0),(0,-1),(-1,0),(0,1)]$.}\\
 & \multicolumn{3}{l}{\textbf{While} $M$ is non-empty \textbf{do}}\\
 & & \multicolumn{2}{l}{Denote by $u,v$ the last element respectively of $L$ and $M$.} \\
 & & \multicolumn{2}{l}{\textbf{If} $u,v$ form an $F$-acute angle}\\
 & & & \textbf{then} remove $v$ from $M$ and append it to $L$\\
 & & & \textbf{else} append $u+v$ to $M$. \\
 & & \multicolumn{2}{l}{\textbf{EndIf}}\\
 & \multicolumn{3}{l}{\textbf{EndWhile}}
\end{tabular}
\end{algorithm}

We assume in the following that the discrete domain $\Omega_*$ is defined as the intersection $\Omega_* := \Omega\cap \cZ_*$ of the continuous domain $\Omega$ with a grid $\cZ_*$ of the form
\begin{equation}
\label{def:Z}
\cZ_* := h R_\theta (\offset+\Z^2) = \{ h R_\theta (\offset+x); \, x \in \Z^2\}.
\end{equation}
This grid is defined through a scale parameter $h>0$, a rotation $R_\theta$ of angle $\theta \in \R$, and an offset $\offset \in \R^2$. In practical applications, one generally chooses for simplicity $\theta=0$ and $\offset=0$ (this is the case of all illustrations of this paper). 
The complexity of the FM-ASR may however show, for some untypical Finsler metrics, a strong dependence on the parameters $\theta$ and $\offset$. Hence there is a significant difference between the worst case complexity of the FM-ASR, and the  average case complexity over randomized grid orientations $\theta \in [0,2 \pi]$ and offsets $\offset \in [0,1]^2$, see below.
%In some situations, the complexity
%For some special choices of $\theta$ and $\offset$, the complexity 
%However we will see below that the complexity of the FM-ASR might 
%We denoted by $h>0$ the scale parameter, by $\theta \in \R$ an angle and by $R_\theta$ the associated rotation, and by $o \in [0,1[^2$ an 
%The grid $\cZ$ is obtained 
%Assume that the discrete domain $\Omega_*$ is a portion of the grid $\cZ := z_0+h R\, \Z^2$, obtained by rotating the grid $\Z^2$ by a rotation $R$, rescaling it by a factor $h >0$, and offsetting it by $z_0 \in \R^2$. 

The stencil $V_*(z)$, $z\in \Omega_*$, assembled in the Preprocessing of the FM-ASR% (full description of the algorithm page \pageref{algo:prepross})
, and involved in \iref{eikonal_disc}, is defined by rotating, rescaling and offsetting the mesh $\cT(\cF_z \circ R_\theta)$: with obvious notations
\be
\label{defVz}
V_*(z) := z+h R_\theta \,\cT(\cF_z \circ R_\theta). 
\ee
These stencils have a fine angular resolution in the direction of anisotropy, and a coarser one in other directions, see Figure \ref{fig:MetricTypes}. This distinctive property justifies the name of our algorithm: Fast Marching using Anisotropic Stencil Refinement (FM-ASR).

\begin{figure}
\includegraphics[width=16cm]{\pathPic/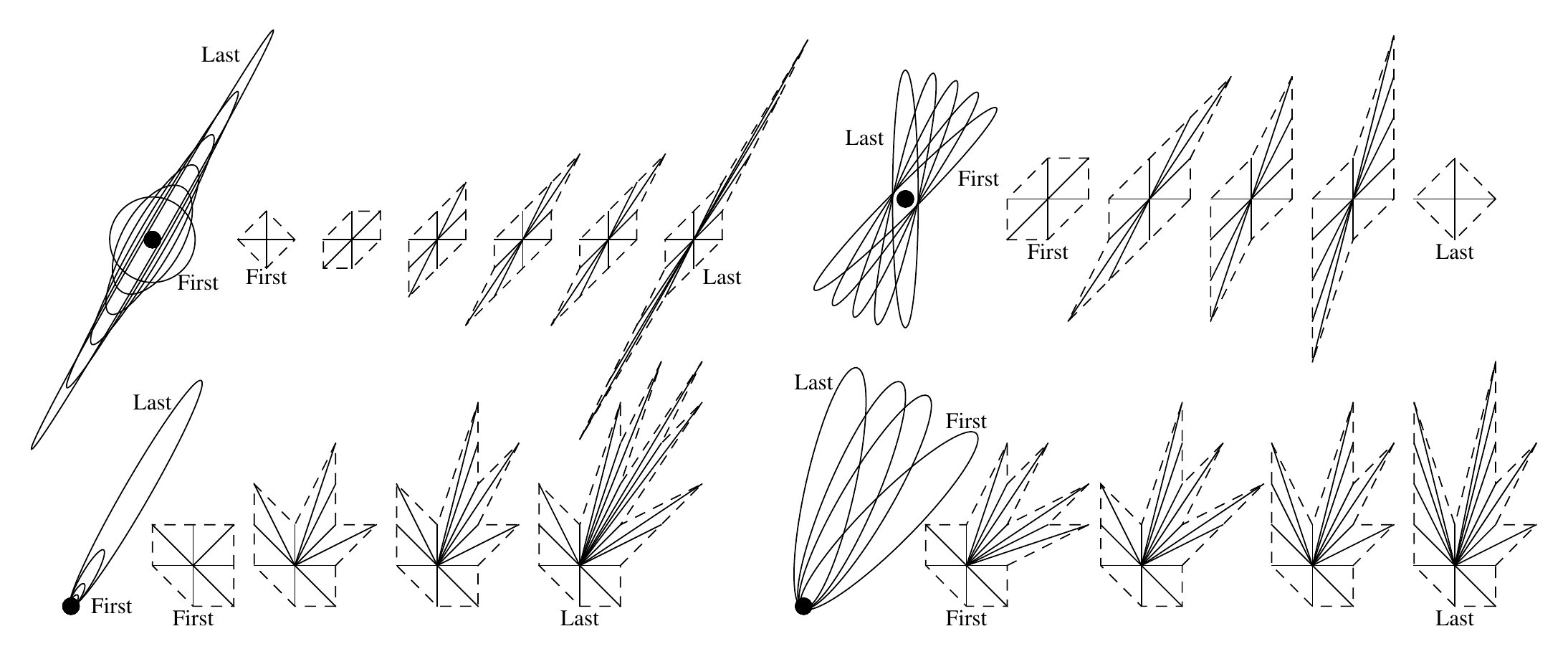}
\caption{
\label{fig:KappaTheta}
Top: the mesh $\cT(F)$ constructed for norms $F$ of anisotropic euclidean type (i.e.\ $F(u):=\sqrt{u^\trans M u}$ for some matrix $M\in S_2^+$), of anisotropy ratio $\kappa(F)$ ranging from $1$ to $32$ and orientation $\pi/3$ (top left), or of anisotropy ratio $\kappa(F)=10$ and orientation ranging from $\pi/4$ to $\pi/2$ (right). Bottom: likewise for asymmetric norms of the form $F(u) := \|u\|-\<\omega,u\>$, $u\in \R^2$, of anisotropy ratio ranging from $4$ to $400$ and orientation $\pi/3$ (bottom left), or of anisotropy ratio $100$ and varying orientations (bottom right).
%Varying orientation, anisotropy.
}
\end{figure}

\begin{algorithm}
\label{algo:prepross}
FM-ASR: Preprocessing.\\
\begin{tabular}{ll}
\textbf{Input:} & A bounded domain $\Omega \subset \R^2$, equipped with a Finsler Metric $\cF \in C^0(\overline \Omega \times \R^2, \R_+)$. \\
& A grid $\cZ_*$, obtained by rotating, rescaling and offsetting (if needed) the grid $\Z^2$.
\end{tabular}
\begin{tabular}{l}
\textbf{Set} $\Omega_* := \Omega \cap \cZ_*$. \\
\textbf{Assemble} the stencils $V_*(z)$, $z\in \Omega_*$, as in \iref{defVz}.\\
\textbf{Assemble} the ``reversed stencils'', defined by $V^*(y):=\{x\in \Omega_*\sm\{y\}; \, y \text{ is a vertex of } V_*(x)\}$.\\
\textbf{Set} $\partial \Omega_* := \{y\in \cZ_* \sm \Omega_*; \, V^*(y)\neq \emptyset\}$.\\
\end{tabular}
\end{algorithm}

The vertices $v$ of the mesh $\cT(F)$, associated to an asymmetric norm $F$, are bounded in terms of the anisotropy ratio: $\|v\| \leq 2 \kappa(F)$ (see Proposition \ref{prop:defMesh} below). Hence the discrete boundary $\partial \Omega_*$, produced by the FM-ASR initialization, may contain grid points at distance $2 h\, \kappa(\cF)$ from the domain $\Omega$.
%Note that the stencil $V_*(z)$, $z\in \Omega_*$, can be rather large, see Figure \ref{fig:MetricTypes}, . %, and may thus contain vertices outside of $\Omega_* \cup \partial \Omega_*$. 
This is not an issue in the case of the null boundary condition \iref{eikonal}, \iref{eikonal_disc}, chosen in our presentation, or of a point source problem (the most common case in applications, see \S \ref{sec:num}). However, if the boundary condition is non-trivial, then its extension from the boundary $\partial \Omega$ to the wider discrete set $\partial \Omega_*$ is required by the FM-ASR.

The next lemma gives a simple characterization of $F$-acuteness when the asymmetric norm $F$ is differentiable or of anisotropic euclidean type (i.e.\ defined by a symmetric positive definite matrix). The characterization \iref{eq:acuteGrad}, for smooth norms, was introduced in \cite{Vlad08} in the same context.
We denote by $S_2^+$ the collection of $2\times 2$ symmetric positive definite matrices. %(*illustration ?*)
\begin{lemma}
\label{lem:acuteCriterion}
Let $F$ be an asymmetric norm on $\R^2$, and let $u,v\in \R^2 \sm\{0\}$.
\begin{enumerate}
\item
If $F$ is differentiable at $u,v$, then these vectors form an $F$-acute angle if and only if 
%Let $F$ be an asymmetric norm, differentiable on $\R^2\sm \{0\}$. Two vectors $u,v\in \R^2 \sm\{0\}$ form an $F$ acute angle if and only if 
\be
\label{eq:acuteGrad}
\<u,\nabla F(v)\> \geq 0 \stext{ and } \<v,\nabla F(u)\> \geq 0.
\ee
\item 
If these exists $M\in S_2^+$ such that $F(w) = \sqrt{\<w,M w\>}$, for all $w\in \R^2$, then $u,v$ form an $F$-acute angle if and only if 
\be
\label{eq:acuteM}
\<u,M v\> \geq 0.
\ee
%Let $M\in S_2^+$ and consider the norm defined by $F(u) := \sqrt{u^\trans M u}$, $u\in \R^2$. Two vectors $u,v$ form an $F$-acute angle if an only if $u^\trans M v \geq 0$. (The positivity of this scalar product justifies the ``acute'' terminology.)
\end{enumerate}
\end{lemma}

\begin{proof}
We first establish Point 1. We have the Taylor development $F(u+\delta v) = F(u)+ \delta \<v, \nabla F(u)\>+ o(\delta)$ as $\delta\to 0$, and likewise exchanging the roles of $u$ and $v$. Thus \iref{eq:acute} clearly implies \iref{eq:acuteGrad}. Conversely the function $F$, being convex, is above its tangent maps, hence $F(u+\delta v) \geq F(u) + \delta \<v, \nabla F(u)\>$ for all $\delta\in \R$, and likewise exchanging the roles of $u$ and $v$. Thus \iref{eq:acuteGrad} implies \iref{eq:acute}, which concludes the proof of Point 1.
%If $u,v$ form an $F$ acute angle, then a taylor development of , and for that purpose we remark that $F$, being convex, is above its tangent maps. Hence for any $u,v\in \R^2\sm\{0\}$ and $\delta\geq 0$ one has 
%$$
%F(u+\delta v) \geq F(u)+\delta v \nabla F(u).
%$$

Point 2 immediately follows from the following expansion: for any $u,v \in \R^2$, $\delta\in \R$, one has 
\[
F(u+\delta v)^2 = F(u)^2+2\delta \<u,M v\> + \delta^2 F(v)^2.
\qedhere
\]
\end{proof} %put this square where it should be.

If $F$ is the canonical euclidean norm, then $F$-acuteness coincides with the standard notion of acuteness (apply \iref{eq:acuteM} to $M:=\Id$). The following proposition, or a close variant \cite{T95,SV03}, is at the foundation of all Dijkstra inspired methods. The positivity of the differences $d_w-d_u$, $d_w-d_v$, is a substitute for the positivity of the edge weights in the classical Dijkstra algorithm.

% methods are all based on the following proposition, or a close variant \cite{T95,SV03}.
%The following proposition, is now fairly classical and has a number of close variants \cite{T95,SV03}, is at the foundation of fast marching methods.
%The following proposition shows the application of $F$-acuteness in connection with \iref{defLambda}.
%We recover the standard notion of acuteness if $F$ is the 

\begin{prop}[Causality Property]
\label{prop:upwind}
Let $F$ be an asymmetric norm on $\R^2$, let $u,v \in \R^2$ be linearly independent, and let $d_u, d_v \in \R$. Assume that $u$ and $v$ form an $F$-acute angle. 
Define 
\be
\label{mint}
d_w := \min_{t \in [0,1]} t d_u+(1-t) d_v + F(t u+(1-t) v),
\ee
and assume that this minimum is not attained for $t\in \{0,1\}$.
%and denote by $t \in [0,1]$ a minimizer of the above expression. 
Then $d_u<d_w$ and $d_v<d_w$. %, and $F^*(d_w-d_u,d_w-d_v) = 1$.
%Then the following holds : 
%\begin{center}
%if $t>0$ then $d_u < d_w$, and if $t<1$ then $d_v < d_w$.
%\end{center}
%Furthermore $F^*(d_w-du,d_w-d_v) = 1$ if $0<t<1$.
\end{prop}

\begin{proof}
See appendix.
\end{proof}

In order to describe the Execution of the FM-ASR, see algorithm page \pageref{algo:execution}, we introduce a variant of the Hopf-Lax update \iref{defLambda}, which uses two additional variables: a boolean map $b:\Omega_* \to \{\trial, \accepted\}$, and a grid point $y\in \Omega_*$.
\be
\label{defLambdaby}
\Lambda(\dist, x; \, b,y) := \min_{x' \in \Gamma} \cF(x'-x) + \interp_V \dist(x'),
\ee
where $V:=V_*(x)$, and $\Gamma$ denotes union of the vertex $y$, and the (at most two) segments $[y,z]$ of $\partial V$ containing $y$ and another vertex $z$ of $V$ such that $b(z) = \accepted$. The second part of the FM-ASR, Execution, is common to the original Fast-Marching algorithm \cite{T95} and its variants \cite{AltonMitchell12,M12}. The fact that it solves the discrete fixed point problem \iref{eikonal_disc} follows from the Causality Property, see \cite{T95,M12} for a proof. %, and the $F$-acuteness of consecutive boundary vertices of the $F$-acute mesh $\cT(F)$  

\begin{algorithm}
\label{algo:execution}
%Fast Marching using Adaptive Stencil Refinement (FM-ASR). \emph{Execution}.\\
FM-ASR: Execution (common to other variants of the Fast Marching algorithm \cite{T95}).\\
Variables: a boolean table $b: \Omega_* \cup \partial \Omega_* \to \{\trial, \accepted\}$, and a map $\dist : \Omega_* \cup \partial \Omega_* \to \overline \R_+$.\\
\begin{tabular}{lll}
\multicolumn{3}{l}{
\textbf{Initialize}  $\dist$ to $+\infty$ on $\Omega_*$, and to $0$ on $\partial \Omega_*$. Initialize $b$ identically to $\trial$.} \\
\multicolumn{3}{l}{
\textbf{While} $b$ is not identically $\accepted$ \textbf{do}}\\
& 
\multicolumn{2}{l}{
Denote by $y \in \Omega_* \cup \partial \Omega_*$ a minimizer of $\dist$ among those points such that $b(y)=\trial$.} \\
& \multicolumn{2}{l}{\textbf{Set} $b(y)::=\accepted$.}\\
& \multicolumn{2}{l}{\textbf{For all} $x\in V^*(y)$ such that $b(y)=\trial$ \textbf{do}}\\
& & \textbf{Set} $\dist(x) ::= \min\{\dist(x),\, \Lambda(\dist,x; \, b,y)\}$.\\
& \multicolumn{2}{l}{\textbf{EndFor}}\\
\multicolumn{3}{l}{\textbf{EndWhile}}\\
\multicolumn{3}{l}{\textbf{Output:} the distance map $\dist$.}
\end{tabular}
\end{algorithm}

For each step size $h>0$, consider the discrete domain $\Omega_h := \Omega \cap (h \Z^2)$, and the associated solution $\dist_h$ of the system \iref{eikonal_disc} produced by the FM-ASR. A proof of uniform convergence of the discrete maps $(\dist_h)_{h>0}$ towards the solution $\distC$ of the continuous Escape Time problem,
$$
\lim_{h \to 0} \left(\max_{z\in \Omega_h} | \distC(z) - \dist_h(z)|\right),
$$
is presented in \cite{M12} for the FM-LBR, a closely related algorithm, in the special case where 
$\Omega = [-1/2,1/2]^2\sm \{(0,0)\}$, and where periodic boundary conditions are applied to the external boundary of $\Omega$
(Equivalently $\Omega = \R^2\sm \Z^2$ and the metric is periodic: $\cF_z = \cF_{z+u}$ for all $u \in \Z^2$).
%$\Omega$ is the unit square equipped with periodic boundary conditions, and $\partial \Omega$ is reduced to a point. 
%$\Omega$ is the unit square $[-1/2,1/2]^2\sm \{(0,0)\}$, equipped with periodic boundary conditions. Equivalently
The adaptation of this proof to the FM-ASR is straightforward\footnote{The proof can in fact be simplified in the case of the FM-ASR,  since the stencil $V_*(z)$ of a grid point $z\in \Omega_*$ contains the four immediate grid neighbors of $z$. This makes Lemmas 2.6 (Consistency) and 2.7 trivial in \cite{M12}. }, and is not reproduced here.\\

The rest of this introduction, and of this paper, is devoted to estimating the complexity of the FM-ASR. Unsurprisingly, this complexity is tied to the cardinality of the FM-ASR stencils, and thus to the cardinality of the $F$-acute meshes $\cT(F)$ used to define them.
%defined on page \pageref{algo:TF} for each asymmetric norm $F$.
%The next proposition discusses worst case performance. 
The next proposition provides an uniform upper bound on $\#(\cT(F))$, in terms of the anisotropy ratio $\kappa(F)$ of the given asymmetric norm $F$. 
%the cardinality of the stencil
%Unfortunately, this first estimate does not reflect the strong performance improvemen
% the impact of anisotropy on mesh cardinality can be quasi linear, instead of logarithmic, 
This first, coarse estimate is however not much satisfying: mesh cardinality grows (quasi-)linearly with the anisotropy ratio,
 and our anisotropic construction of $F$-acute meshes $\cT(F)$ has little advantage over an isotropic one $\cT_{\kappa(F)}$, depending only on the anisotropy ratio.

\begin{prop}
\label{prop:WorstCase}
There exists a constant $C$, such that the following holds. For any asymmetric norm $F$ on $\R^2$ one has:
\be
\label{CardAsym}
\#(\cT(F)) \leq C\kappa(F) (1+\ln \kappa(F)).
\ee
A slightly sharper estimate holds if $F$ is \emph{symmetric}:
\be
\label{CardSym}
\#(\cT(F)) \leq C\kappa(F).
\ee
For any $\kappa \geq 1$,  there exists a mesh $\cT_\kappa$ which is $F$-acute for any asymmetric norm such that $\kappa(F)\leq \kappa$, and has cardinality $\#(\cT_\kappa) \leq C \kappa (1+\ln \kappa)$.
%Let $\kappa \geq 1$ be arbitrary. Then 
There also exists an anisotropic euclidean norm $F_\kappa$ such that $\kappa(F) \leq \kappa$ and  $\#(\cT(F_\kappa)) \geq \kappa/C$. 
%Furthermore (i) there exists a for each $\kappa \geq 1$ an anisotropic euclidean norm $F_\kappa$,
%There exists a constant $C$ such that the following holds.
%\begin{enumerate}[(i)]
%\item For any \emph{symmetric} norm $F$ one has $\#(\cT(F)) \leq C \kappa(F)$. Furthermore for each $\tau \geq 1$ the norm $F_\tau$ defined by the symmetric positive definite matrix 
%$$
%M_\tau := 
%\left(
%\begin{array}{cc}
%1 & \tau \\
%\tau & 2 \tau^2
%\end{array}
%\right)
%$$
%satisfies $ \tau/C \leq \kappa(F_\tau) \leq C \tau$ and $\#(\cT(F_\tau)) \geq \tau/C$ .
%\item For any asymmetric norm $F$ one has $\#(\cT(F)) \leq C \kappa(F) (1+\ln \kappa(F))$.
%\item For each $\kappa\geq 1$ there exists a mesh $\cT_\kappa$, of cardinality $\#(\cT_\kappa) \leq C \kappa (1+\ln \kappa)$, which is $F$-acute for all asymmetric norms $F$ such that $\kappa(F) \leq \kappa$.
%\end{enumerate}
\end{prop}

%The following theorem is our main result: it establishes that the cardinality of the $F$-acute mesh $\#\cT(F)$, used to define the stencils of the FM-ASR, grows only logarithmically with the anisotropy of the given asymmetric norm $F$, in an average sense over all orientations.
The following theorem is our main result: it establishes that the cardinality of $\cT(F)$ grows only (poly-)logarithmically with the anisotropy of $F$, in an average sense over all orientations.
The difference between the uniform and the average cardinality bounds, Proposition \ref{prop:WorstCase} and Theorem \ref{th:TCard} respecticely, reflects the fact, illustrated on Figure \ref{fig:rotate}, that the cardinality of $\cT(F)$ strongly depends on the orientation of the anisotropy of $F$.
%This shows, as illustrated on Figure \ref{fig:rotate}, that the uniform upper bound given in the pr
%For each asymmetric norm $F$ we denote by $F^\theta := F \circ R_\theta^\trans$ the asymmetric norm defined by the composition of $F$ with the transposed rotation $R_\theta$ of angle $\theta$ : for all $u\in \R^2$
For each $\theta\in \R$ we define the rotated asymmetric norm $F^\htheta$ by 
$$
F^\htheta(u) := F(R_\theta^\trans u),
$$
where $u\in \R^2$ and $R_\theta$ denotes the rotation matrix of angle $\theta$, see Figure \ref{fig:rotate}. 
%We denote $\ln^2 x := (\ln x)^2$, for all $x>0$.
%For all $x>0$, we denote $\ln^2 x := (\ln x)^2$.
\begin{figure}
\includegraphics[width=4cm]{\pathPic/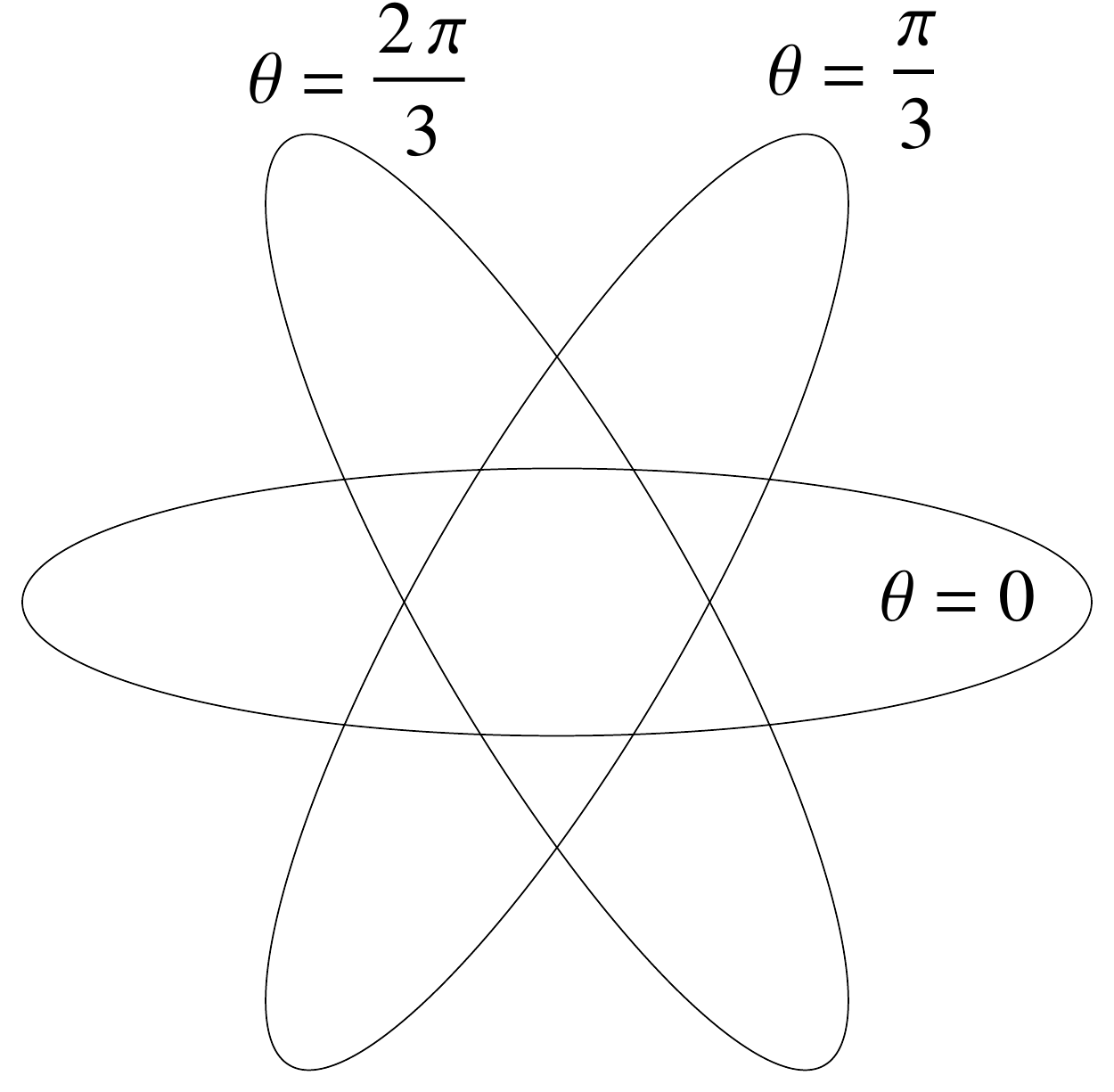}
\includegraphics[width=6cm]{\pathPic/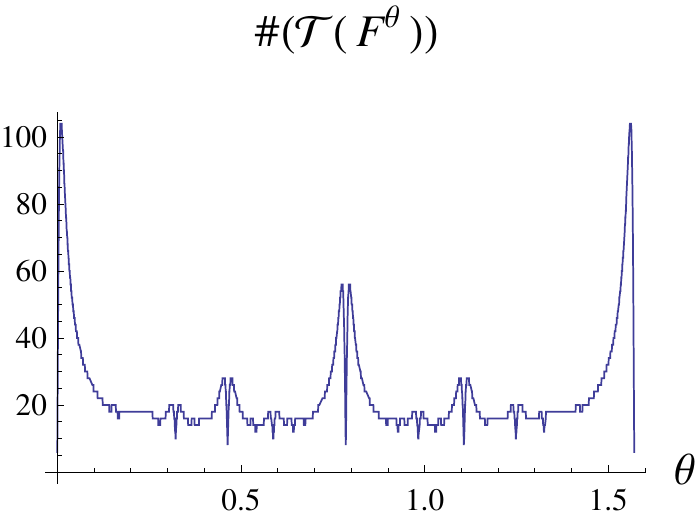}
\includegraphics[width=6cm]{\pathPic/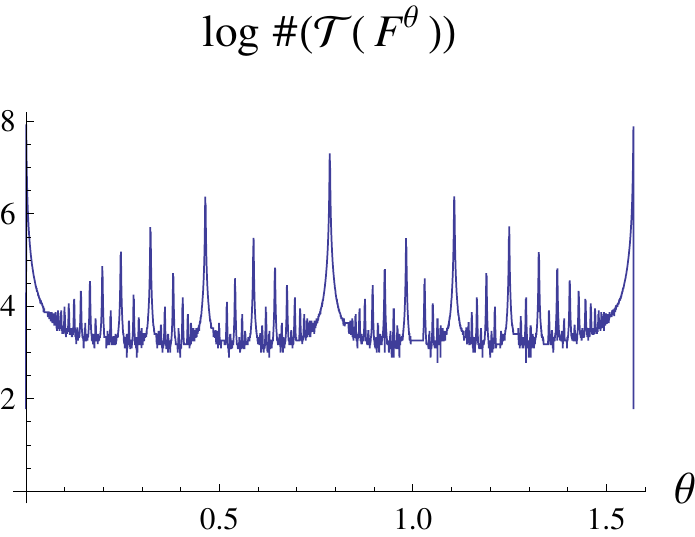}
\caption{
\label{fig:rotate}
The unit ball of $F^\htheta$ is the unit ball of $F$ rotated by the angle $\theta$ (left). Here the norm $F$, of anisotropic euclidean type, is given by the diagonal matrix of entries $(\kappa,1/\kappa)$, with $\kappa=4$ (left), $\kappa=100$ (center, linear plot) and $\kappa = e^8$ (right, log plot). 
In this example, the cardinality of $\cT(F^\htheta)$ is highly dependent on the angle $\theta$, and seems to spike when $(\cos \theta, \sin \theta)$ is close to be proportional to a vector with small integer coordinates.
}
\end{figure}
%(*figure. Ball rotation, some neighborhoods, cardinality in terms of orientation.*)
\begin{theorem}
\label{th:TCard}
There exists a constant $C$, such that for any asymmetric norm $F$ on $\R^2$, one has:
\be
\label{avgCardAsym}
\int_0^{2 \pi} \#(\cT(F^\htheta)) \, d\theta \leq C (1+\ln^3 \kappa(F))%C(1+\ln \kappa(F))^3.
\ee
A slightly sharper estimate holds if $F$ is \emph{symmetric}:
\be
\label{avgCardSym}
\int_0^{2 \pi} \#(\cT(F^\htheta)) \, d\theta \leq C(1+\ln^2 \kappa(F)).
\ee
\end{theorem}

We next use Proposition \ref{prop:WorstCase} and Theorem \ref{th:TCard} to obtain worst case and average case complexity estimates for the FM-ASR. For that purpose, we fix the bounded smooth domain $\Omega \subset \R^2$, and the scale parameter $h>0$.
For each angle $\theta \in \R$, and offset $\offset \in \R^2$, we introduce the grid 
$$
\cZ_{\theta, \offset} := h R_\theta^\trans (\offset + \Z^2).
$$
In the rest of this introduction, the subscript $*$, used above to denote discrete entities, is replaced with the grid parameters $(\theta,u)$.
The discrete domain is thus denoted by $\Omega_{\theta,u} := \Omega \cap \cZ_{\theta, u}$, the discrete boundary by $\partial \Omega_{\theta,u}$, and the stencils by $V_{\theta,u}(z)$, $z\in \Omega_{\theta,u}$.
Like other Dijkstra-inspired solvers of the Escape Time problem, see Remark \ref{rem:Complexity}, the complexity of the FM-ASR is given by 
\begin{equation}
\label{DijkstraComplexity}
\cO(N\ln N + N'_{\theta,u}),
\end{equation}
where $N$ denotes the total number of discrete points, and $N'_{\theta,u}$ the sum of the stencil cardinalities.
%only depends on the total number $N$ of discrete points and on the sum $N'_{\theta,u}$ of the stencil cardinalities:
\begin{equation*}
N := \#(\Omega_{\theta,u} \cup \partial \Omega_{\theta,u}), \qquad 
N'_{\theta,u} := \sum_{z\in \Omega_{\theta,u}} \#(V_{\theta,u}(z)) = \sum_{z \in \Omega_{\theta,u}} \#(\cT(\cF_z^\htheta)).
\end{equation*}
The discrete domain cardinality $N$ is mostly independent of the grid orientation parameters $\theta, u$ (this is why we write $N$ and not $N_{\theta,u}$): if the scale parameter $h$ is sufficiently small, then 
\begin{equation}
\label{NArea}
N \simeq |\Omega| h^{-2}.
\end{equation}
Proposition \ref{prop:WorstCase} implies a worst case upper bound for $N'_{\theta,u}$:
%, and thus for the FM-ASR complexity \eqref{DijkstraComplexity}: %denoting $\alpha := 0$ in the case of a symmetric Finsler metric, and $\alpha:=1$ otherwise
\begin{equation}
\label{NtuEstim}
N'_{\theta,u} \leq \#(\Omega_{\theta,u}) \,\max_{z \in \Omega} \#(\cT(\cF_z^\theta)) \leq N \kappa(\cF)(1+ \ln \kappa(\cF)).
\end{equation}
%This first estimate of $N'_{\theta,u}$ is only relevant 
Let $N'$ be the average value of $N'_{\theta,u}$, over the collection of grid orientation parameters $(\theta, u)\in [0,2\pi] \times [0,1]^2$. This average value is, as expected, much smaller than from the above uniform upper bound:
%In average over the grid parameters $(\theta, u)\in [0,2\pi] \times [0,1]^2$, the sum of stencil cardinalities is far from this upper bound:
\begin{align*}
2 \pi N' &:= \int_0^{2\pi} \int_{[0,1[^2} N'_{\theta,u}\, du \, d \theta \\
&= \int_0^{2\pi} \int_{[0,1[^2} \sum_{z \in \Omega_{\theta,u}} \#(\cT(\cF_z^\htheta)) \, du \, d \theta \\
&= h^{-2}\int_0^{2 \pi} \int_{z \in \Omega} \#(\cT(\cF_z^\theta)) \, dz\, d \theta \\
&\leq C h^{-2} \int_{z \in \Omega} (1+\ln^3 \kappa(\cF_z)) dz,\\
&\leq C \, |\Omega| h^{-2}\,  (1+\ln^3 \kappa(\cF)).
\end{align*}
Thus, using \eqref{NArea} 
\begin{equation}
\label{NpEstim}
N' \lesssim C N (1+\ln^3 \kappa(\cF)). 
\end{equation}
Combining \eqref{DijkstraComplexity} with \eqref{NtuEstim} and \eqref{NpEstim}, we obtain that the worst case complexity of the FM-ASR is $\cO(N \kappa(\cF) \ln \kappa(\cF) + N \ln N)$, while the average case complexity\footnote{%
In the case of a symmetric Finsler metric, the worst case and average case complexities drop respectively to  $\cO(N \kappa(\cF) + N \ln N)$ and $\cO(N \ln^2 \kappa(\cF) + N \ln N)$.%
},
over randomized grid orientation parameters $(\theta,u)$, is $\cO(N \ln^3 \kappa(\cF) + N \ln N)$. The worst case complexity corresponds to untypical cases where e.g.\ the Finsler metric has a preferred anisotropy direction over a large portion of the domain, and the discretization grid is almost aligned with this direction; the average case complexity is more likely to reflect application performance.

The FM-ASR average complexity $\cO(N \ln^3 \kappa(\cF) + N \ln N)$ is significantly below that of Bellman-Ford inspired algorithms, such as the AGSI \cite{BR06} of complexity $\cO(\lambda(\cF) N^{3/2})$, thanks to the quasi-linear complexity in $N$. Alternative Dijkstra inspired solvers include the Ordered Upwind Method (OUM) \cite{SV03} (which uses dynamic stencils, constructed on the fly during the execution), and the Monotone Acceptance OUM (MAOUM) \cite{AltonMitchell12}. They use stencils larger than those of the FM-ASR, of cardinality between $\kappa(\cF)$ and $\kappa(\cF)^2$, which results in a complexity linear if not polynomial in the anisotropy ratio: $\cO(\kappa(\cF)^\beta N \ln N)$ \cite{SV03,AltonMitchell12}, for some\footnote{
Strictly speaking, $\beta=2$. Yet the \emph{asymptotic} complexity of the OUM, as $N \to \infty$, drops to $\cO(\kappa(\cF) N \ln N)$. 
If the MAOUM is executed on a periodic mesh, then the stencil of $z$ only depends on a single parameter: the anisotropy ratio $\kappa(\cF_z)$. It costs $\cO(\kappa(\cF_z)^2)$ to construct, but the MAOUM execution only involves its boundary, which contains $\cO(\kappa(\cF_z))$ elements. In our numerical experiments \S \ref{sec:num} these stencils are precomputed, stored in a look-up table, and the complexity of the MAOUM drops to $\cO(\kappa(\cF) N \ln N)$. 
} 
$\beta \in [1,2]$. In defense of the AGSI, OUM and MAOUM, let us mention that these alternative algorithms are not limited to grid discretizations, contrary to the FM-ASR, see Remark \ref{rem:Specialization}.
Fast Marching using Lattice Basis Reduction (FM-LBR), introduced in \cite{M12} by the author, has like the FM-ASR a complexity $\cO(N \ln \kappa(\cF)+N \ln N)$ logarithmic in the metric anisotropy and quasi linear in the number of unknowns. Yet the application range of the FM-LBR is different: it extends to dimension 3, and 4 \cite{M12b}, but only applies to metrics of Riemannian type. A numerical comparison of the FM-ASR with the AGSI, the MAOUM, the FM-8 (a fast but not always convergent alternative) and the FM-LBR (when applicable), is presented in \S\ref{sec:num}.
%also have a larger complexity due to such as the Orthogonal Upwind method (OUM) 
\\

%\begin{remark}
%For any anisotropic \emph{euclidean} norm $F$, the construction of an alternative $F$-acute mesh $\cT'(F)$ is described in Proposition 1.9 \cite{M12}. This construction has the advantage (i) of a uniformly bounded cardinality, $\#(\cT'(F))\leq 6$, and (ii) of a generalization to anisotropic euclidean norms on $\R^3$ and $\R^4$, respectively in Proposition 1.10 of \cite{M12} and Proposition 3.2 of \cite{M12b}. On the other hand the construction proposed here applies to arbitrary Finsler metrics, and sometimes yields slightly more accurate numerical results \cite{M12b}.
%\end{remark}

We discuss in \S\ref{sec:construction} the construction of the $F$-acute mesh $\cT(F)$, for any asymmetric norm $F$. Section \S3 is devoted to the proof of our main result Theorem \ref{th:TCard}, achieved in Corollaries \ref{corol:CardSym} and \ref{corol:CardAsym}. The proof of the worst case analysis, Proposition \ref{prop:WorstCase}, is achieved in Corollaries \ref{corol:WorstCase} and \ref{corol:CardSym}. We present some numerical results in \S \ref{sec:num}. 

%These are finally achieved in Corollaries 
%, and we prove Proposition \ref{prop:WorstCase} in \S \ref{sec:worstCase}. Our main result Theorem \ref{th:TCard} is proved in \S\ref{sec:cardTF}. Some numerical experiments are presented in \S\ref{sec:num}.

\begin{remark}[Performance comes at the price of specialization]
\label{rem:Specialization}
The FM-ASR, introduced in this paper, is an efficient method to solve strongly anisotropic and/or asymmetric Escape Time problems when the discrete domain $\Omega_*$ is a subset of $\Z^2$, or of another orthogonal grid.
Extending this algorithm to a broader class of discrete domains requires to generalize the construction of the stencils $V(z)$,  $z\in \Omega_*$. In particular one must find an analog of the rule ``if $u,v\in \Z^2$ do not form an $\cF_z$-acute angle, then consider their sum $u+v \in \Z^2$'' which appears implicitly in the construction of the mesh $\cT(\cF_z)$, algorithm page \pageref{algo:TF}, and thus of $V(z)$ \iref{defVz}. This is non-trivial.
%is an open question. 
\begin{itemize}
\item \textbf{If the discrete domain $\Omega_*$, two dimensional, is not a grid subset}. The points $u,v$ do not belong to a lattice, but are differences $u=x-z$, $v=y-z$, between the point $z$ where the stencil is constructed, and close-by discrete points $x,y\in \Omega_*$. 
The new inserted stencil vertex $z'$ cannot be obtained as the sum $z+u+v = x+y-z$, which may not belong to $\Omega_*$. Instead, $z'$ should be chosen as the point closest to $z$ in the open cone $z+\R_+^* u+\R_+^*v$, where $\R_+^*$ denotes positive reals. However, it is not clear wether data structures exist, for the discrete domain $\Omega_*$, which allow to perform this closest point search without strongly increasing the complexity of the FM-ASR.
\item \textbf{If the domain is three dimensional, and $\Omega_*$ is a subset of $\Z^3$}. There are now three points $u,v,w\in \Z^3$, vertices of a \emph{facet} of the stencil boundary $\partial V$. 
The extension of the Causality Property, Proposition \ref{prop:upwind}, to the case of an arbitrary asymmetric norm $F: \R^3 \to \R_+$, requires not only the pairs of vertices $(u,v)$, $(u,w)$, $(v,w)$ to form $F$-acute angles, but also all the pairs of a vertex and a point of the opposite edge: $(u,\, t v+(1-t)w)$, $t\in [0,1]$, and likewise exchanging the roles of $u,v,w$.
This can be be difficult to check numerically. In the case of a Riemannian metric, checking $F$-acuteness for pairs of vertices is sufficient \cite{SV03,M12}, but there remains an ambiguity: should the new inserted vertex be $u+v$ or $u+w$, if none of the corresponding angles is $F$-acute? Our attempts to generalize the FM-ASR to this setting were unconvincing, both experimentally and theoretically, hence we recommend the FM-LBR \cite{M12} for such 3d, Riemannian, Escape Time problems.
\end{itemize}
\end{remark}

\begin{remark}[Detailled complexity analysis of the FM-ASR]
\label{rem:Complexity}
Preprocessing step, page \pageref{algo:prepross}. We omit the complexity of the construction of the discrete domain $\Omega_*$
as the intersection of the continuous domain with a grid, 
%$\Omega_{\theta,u}$, as the intersection of the grid $\cZ_{\theta,u}$ with the continuous domain $\Omega$, 
since this is either trivial or dependent on the chosen representation of $\Omega$. 
Consider an asymmetric norm $F$ such that answering the predicate ``$u,v$ form an $F$-acute angle'' has cost $\cO(1)$, for any $u,v\in \R^2$. That is the case is $F$ is differentiable, and if evaluating the gradient $\nabla F (u)$ has cost $\cO(1)$, using Lemma \ref{lem:acuteCriterion}. Then constructing the $F$-acute mesh $\cT(F)$ has cost $\cO(\#(\cT(F)))$, where $\#(\cT(F))$ denotes the number of triangles in the triangulation $\cT(F)$, which is also the number of its boundary vertices. %since it is star shaped
As a result, assembling the stencils of the FM-ASR has cost $\cO(N')$, where $N' = N'(\Omega_*, \cF)$ is the sum of the stencil cardinalities
$$
N' := \sum_{z\in \Omega_*} \#(V_*(z)).
$$
Assembling the reversed stencils $V^*(z)$, $z\in \Omega_*$, is done by reversing a directed graph having $N'$ edges, and thus also has cost $\cO(N')$. Note that $N'$ is also the sum of the cardinalities of the reversed stencils. Storing these stencils leads to a $\cO(N')$ memory footprint for the FM-ASR, which is not required by e.g.\ the AGSI \cite{BR06}, see Remark 2.5 in \cite{M12} for a discussion of this point. The total complexity of the FM-ASR Preprocessing is thus $\cO(N')$. 

Execution step, page \pageref{algo:execution}. Let $N := \#(\Omega_* \cup \partial \Omega_*)$ be the cardinality of the discrete domain, which is also the number of unknowns in \iref{eikonal_disc}. The execution requires to maintain a list of the points in $\Omega_* \cup \partial \Omega_*$ such that $b(z)=\trial$, sorted by increasing values of $\dist$. We assume in this complexity analysis that the data structure used for this purpose is a Fibonacci Heap, in such way that the ``Remove\_Key'' and the ``Decrease\_Key'' operations on this list have respective amortized complexity $\cO(\ln N)$ and $\cO(1)$. 
The ``Remove\_Key'' routine is called $N$ times, once a each command $b(y)::=\accepted$, and the ``Decrease\_Key'' routine at most $N'$ times\footnote{
Fibonacci Heaps are a data structure specifically tailored for Dijkstra-Like algorithms on densely connected graphs: $N' \gg N$. In the numerical experiments presented on \S\ref{sec:num}, one always have $N' \leq 20 N$ for the FM-ASR, and using a classical binary heap proved to be more efficient. We used Boost's implementation of Fibonacci heaps, and the Standard Template Library  for binary heaps.
}, once at each command $\dist(x) ::= \min\{\dist(x),\, \Lambda(\dist,x; \, b,y)\}$. 
Evaluating the modified Hopf-Lax update operator $\Lambda(\dist,x; \, b,y)$ requires to solve at most two convex minimization problems of the form \iref{mint}: one for each boundary edge $[y,z]$ of $V_*(x)$ containing $y$ and a vertex $z$ such that $b(z)=\accepted$ \iref{defLambdaby}. The complexity of their resolution is regarded as elementary; in many interesting cases, they have an explicit solution involving $\cO(1)$ elementary operations ($+,-,\times,/$ and $\sqrt\cdot$) among reals, see Proposition \ref{prop:OffsettedNorms}. Like other Dijkstra inspired algorithms, the total cost of the FM-ASR execution is thus $\cO(N'+N\ln N)$.
\end{remark}

\section{Construction of the stencils}%Construction of the local meshes}
\label{sec:construction}

We discuss in this section the construction of the $F$-reduced mesh $\cT(F)$, defined for each asymmetric norm $F$, and used to define the stencils of the FM-ASR \iref{defVz}. 
The construction presented in the introduction is reformulated as an in-order transversal of four binary trees. 
We establish a worst case upper bound on $\#(\cT(F))$ in Corollary \ref{corol:WorstCase}, and we introduce a number of tools that will be used in \S\ref{sec:cardTF} to estimate the average cardinality of $\cT(F^\htheta)$, $\theta\in [0,2\pi]$. 

\subsection{Mesh generation by recursive refinement}

%Our first definition introduces elementary triangles, all the triangles considered in the rest of this paper are of this type

All the triangles considered in the rest of this paper share some properties of geometric nature (or arithmetic nature, depending on the point of view), which are introduced in the next definition.
\begin{definition}
%A triangle $T$ is called elementary if it satisfies the following properties: 
An elementary triangle $T$, is a triangle satisfying the following properties:
\begin{itemize}
\item One of the vertices of $T$ is the origin $(0,0)$, and the the other two belong to $\Z^2$.
\item Denoting by $u,v$ the non-zero vertices of $T$, one has 
\be
\label{DetScal}
|\det(u,v)|=1
\ \text{ and } \ 
s(T):=\<u,v\> \geq 0.
\ee
\end{itemize}
\end{definition}

%The second point of this definition is equivalent to the geometrical statement:  
%%\begin{center}
%$T$ has area $1/2$, and an acute angle at the origin.
%%\end{center}
The second point of this definition can be rephrased as a geometrical statement:  
%\begin{center}
$T$ has area $1/2$, and has an acute angle at the origin. %, as in Definition \ref{def:acuteMesh} of $F$-acute meshes.
%\end{center}
Let us recall that for any two vectors $u,v\in \R^2$ one has the identity
\be
\label{ScalDetNorm}
\<u,v\>^2+\det(u,v)^2 = \|u\|^2 \|v\|^2.
\ee
If $u,v$ are non-zero, and if $u$ and $-v$ are not positively collinear, we denote by $\varangle(u,v) \in (-\pi,\pi)$ their oriented angle:
\be
\label{def:angle}
\cos(\varangle(u,v)) = \frac{\<u,v\>}{\|u\| \|v\|} \stext{ and } \sin(\varangle(u,v)) = \frac{\det(u,v)}{\|u\| \|v\|}.
\ee
The scalar product $s(T)$ associated to an elementary triangle $T$ reflects its thinness, indeed if $u,v$ are its non-zero vertices then combining \iref{DetScal}, \iref{ScalDetNorm} and \iref{def:angle} we obtain
\be
\label{sinScal}
\sin |\varangle(u,v)| = \frac 1 {\|u\| \|v\|} = (s(T)^2+1)^{-\frac 1 2}.
\ee

\begin{figure}
\begin{center}
%\includegraphics[width=5cm]{\pathPic/FastMarchingIllus/FinslerNeigh/SingleRefinement.pdf}
%\hspace{0.5cm}
%\includegraphics[width=5cm]{\pathPic/FastMarchingIllus/FinslerNeigh/Tree.pdf}
%\hspace{0.5cm}
%\includegraphics[width=4cm]{\pathPic/FastMarchingIllus/FinslerNeigh/T0.pdf}
\includegraphics[width=4cm]{\pathPic/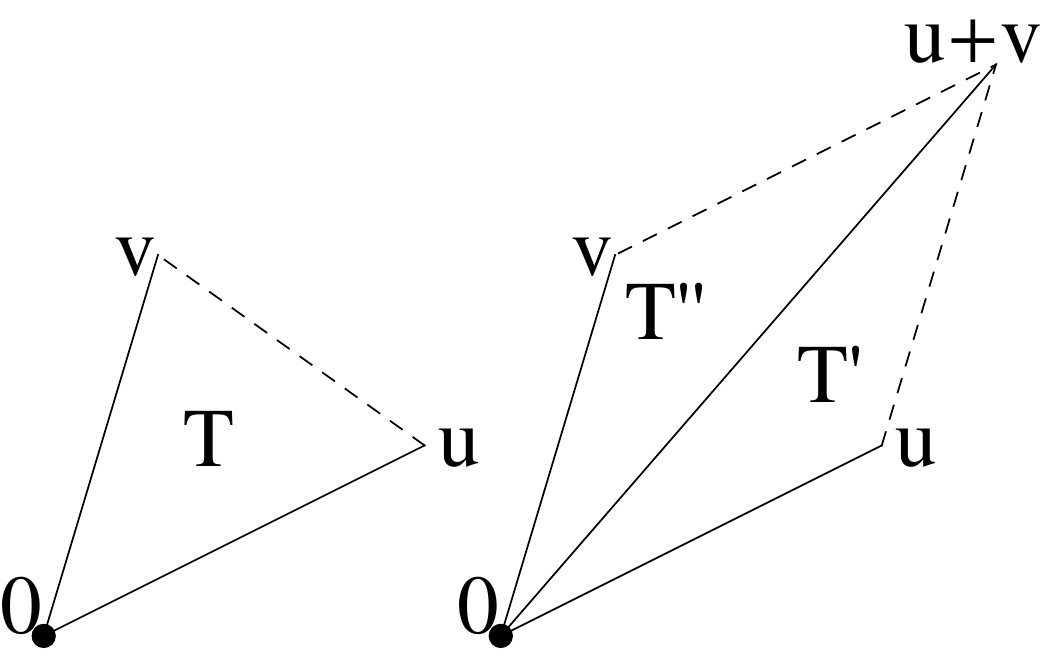}
\hspace{0.2cm}
\includegraphics[width=3cm]{\pathPic/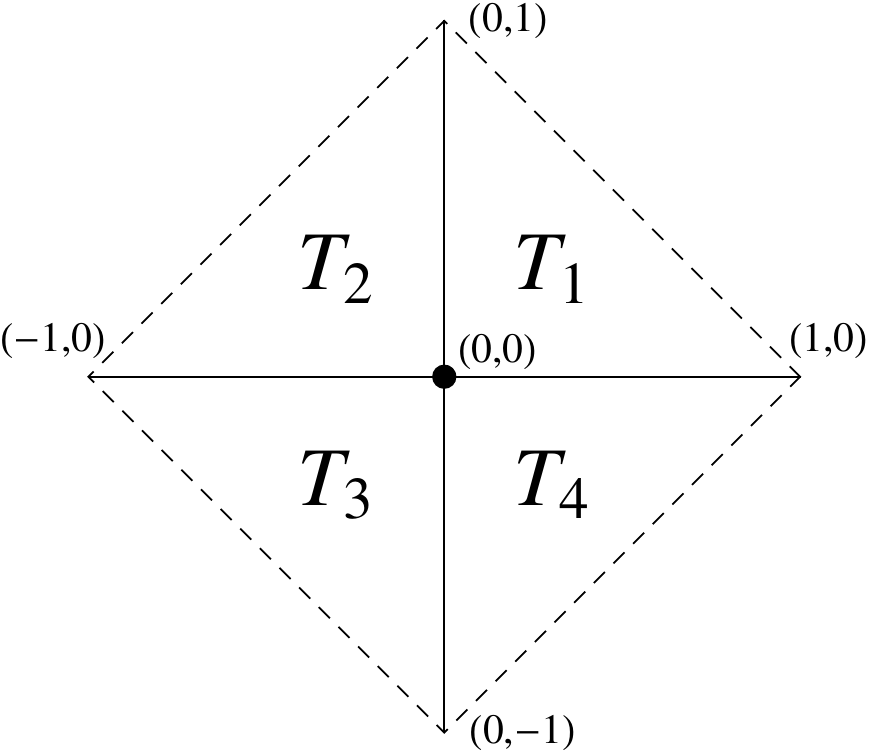}
\hspace{0.2cm}
\includegraphics[width=3cm]{\pathPic/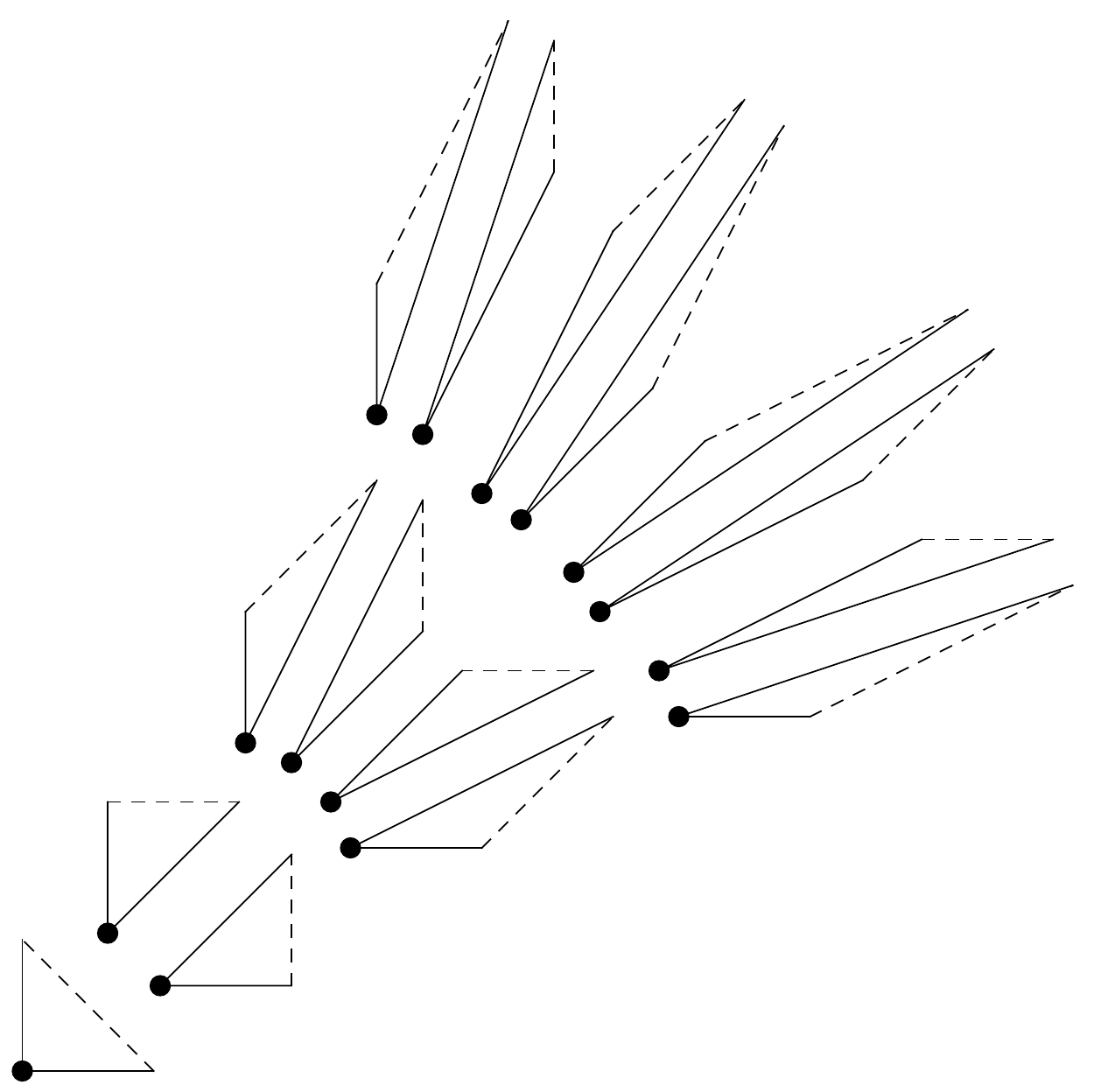}
\hspace{0.2cm}
\includegraphics[width=3cm]{\pathPic/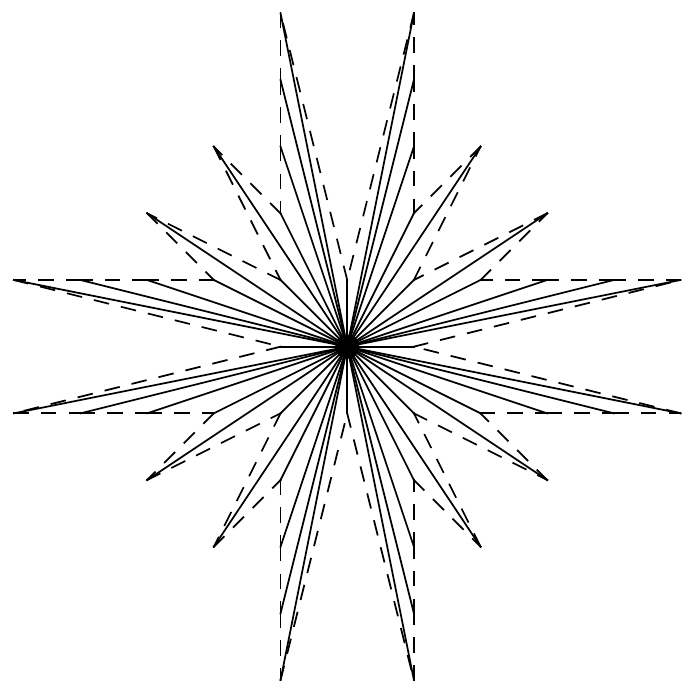}
\end{center}
\caption{
\label{fig:refinement}
Refinement of a triangle (left). Mesh $\cT_0$ (center left). First levels of the binary tree (center right) defined by the recursive refinements of $T_1$. Mesh defined by the ASC: ``$s(T)\geq 5$'' (right).}
\end{figure}

We introduce in the next definition the refinement of an elementary triangle $T$, which is illustrated on Figure \ref{fig:refinement} (left). 
Note that $T$ is  strictly covered by the union of its children.
%A elementary triangle is strictly covered by the union of its children.
%The union of the children of a triangle strictly cover their parent.
\begin{definition}
\label{def:refinement}
The \emph{refinement} of an elementary triangle $T$ of non-zero vertices $u,v$ consists of the two elementary triangles $T'$ and $T''$ of non-zero vertices $(u, u+v)$, and $(u+v, v)$, respectively, which are referred to as its children. 
\end{definition}

The scalar product $s(\cdot)$ grows with refinement:
\be
\label{sGrows}
s(T') = \<u,u+v\> = S(T) + \|u\|^2 \geq S(T)+1.
\ee
This property, combined with \iref{sinScal} reflects the fact that the recursive children of an elementary triangle become thinner an thinner, as can be observed on Figure \ref{fig:refinement} (center right).

We denote by $\cT_0$ the mesh, illustrated on Figure \ref{fig:refinement}, containing the four elementary triangles of non-zero vertices $(\pm 1,0)$ and $(0,\pm 1)$. The next lemma establishes that any elementary triangle can be generated by recursive bisections from an element of $\cT_0$.

\begin{lemma}
\label{lem:parents}
\begin{enumerate}
\item
Let $T$ be an elementary triangle, of non-zero vertices $u$ and $v$. 
The following are equivalent : (i) $T \in \cT_0$, (ii) $\|u\|=\|v\|$, (iii) $\<u,v\> < \min \{\|u\|^2,\|v\|^2\}$.
%If $T \notin \cT_0$ then
%\be
%\label{scalBig}
%\|u\| \neq \|v\| \stext{ and } \<u,v\>\geq \min \{\|u\|^2, \|v\|^2\}.
%\ee
%(*equivalent*)
%If $\|u\|=\|v\|$ then $T\in \cT_0$.
%Then we either have $T\in \cT_0$, or $\<u,v\>\geq \min \{\|u\|^2, \|v\|^2\}$.

\item
%For any elementary triangle $T$ there exists an integer $0 \leq n \leq s(T)$ and elementary triangles $T_0, \cdots, T_n$ such that $T_0 \in \cT_0$, $T_n=T$, and $T_{i+1}$ is a child of $T_i$ for each $0 \leq i < n$. Furthermore, $T$ has 

The collection of elementary triangles, equipped with parent-children relationship, is a forest of four infinite binary trees, which roots are the elements of $\cT_0$.
%
%(*forest of four trees, with roots
\end{enumerate}
\end{lemma}
\begin{proof}
Point 1. We clearly have $(i) \Ra (ii)$, by inspection of the four elements of $\cT_0$, and $(ii) \Ra (iii)$, by observing that $u$ and $v$ are not collinear and thus that $\<u,v\> < \|u\| \|v\|$. We next assume $(iii)$ and establish $(i)$. 
We have
%The following inequalities only involve integers
%The following series of inequalities involves only non-negative integers
$$
% s(T)^2 = 
\<u,v\>^2 <
 \min \{\|u\|^2, \|v\|^2\}^2 \leq \|u\|^2 \|v\|^2 = \<u,v\>^2+ \det(u,v)^2 =  \<u,v\>^2+1.
$$
Comparing the left and right and side, and observing that the members of these inequalities are all integers, we obtain that the non-strict inequality above is an equality. Hence $\|u\| = \|v\|$ and $(\|u\|^2)^2 = \<u,v\>^2+1$.
%(*update*)
%Let $T$ be an elementary triangle of non-zero vertices $u,v$ satisfying $\|u\|=\|v\|$. Then 
%$$
%(\|u\|^2)^2 = \|u\|^2 \|v\|^2 = \<u,v\>^2+\det(u,v)^2 = \<u,v\>^2+1,
%$$
Therefore $\<u,v\>^2$ and $(\|u\|^2)^2$ are consecutive integers which are both perfect squares. Only the integers $0$ and $1$ satisfy this property, hence $1=\|u\|=\|v\|$, which implies that these vectors are of the form $(\pm 1,0)$ or $(0, \pm 1)$. Since $|\det(u,v)|=1$, these vectors are not collinear, and we obtain that $T \in \cT_0$, which concludes the proof of the first point of this lemma.

Point 2. %We next turn to the proof of the second point of this lemma. 
%It follows from \iref{sGrows} that there exists no infinite sequence $(T_n)_{n \leq 0}$ of elementary triangles, such that $T_{n-1}$ is a parent of $T_n$ for each $n \leq 0$.
It follows from \iref{sGrows} that a triangle $T \in \cT_0$ cannot have a parent $R$, since it would satisfy $s(R) < s(T) = 0$. More generally, and for the same reason, an elementary triangle $T$ has at most $s(T)$ ancestors.

%Likewise, there exists no 
%%It follows from \iref{sGrows} that there exists no 
%infinite sequence $(T_n)_{n \geq 0}$ of elementary triangles, such that $T_{n+1}$ is a parent of $T_n$ for each $n \geq 0$.
%Inequality \iref{sGrows} also implies that any triangle $T \in \cT_0$ has no parent.

In order to conclude the proof, we need to show that any elementary triangle $T$, which is not in $\cT_0$, has exactly one parent $R$. Let $u,v$ be the non-zero vertices of $T$, ordered in such way that $\|u\|\leq\|v\|$. The non-zero vertices of a parent $R$ are either $(u-v,v)$ or $(u,v-u)$, but the first case can be excluded since $\<u-v,v\> = \<u,v\> - \|v\|^2 < \|u\| \|v\| - \|v\|^2 \leq 0$. Conversely, the triangle $R$ which vertices are the origin, $u$ and $v-u$, is an elementary triangle since $\det(u,v-u) = \det(u,v) = \pm 1$, and $\<u,v-u\> = \<u,v\> - \|u\|^2 \geq 0$.
\end{proof}

If a mesh $\cT$ contains only elementary triangles, then it automatically satisfies assumption (ii) of Definition \ref{def:acuteMesh}, and so does any mesh $\cT'$ obtained by refining, possibly recursively, some elements of $\cT$. If the mesh $\cT$ satisfies assumption (i) of Definition \ref{def:acuteMesh}, namely that the union of its elements is a neighborhood of the origin, then so does $\cT'$. The mesh constructions proposed in this paper consist in recursively refining the elements of the mesh $\cT_0$, defined above and fixed in the rest of this paper, until all of them satisfy a prescribed \emph{stopping criterion}, see Figures \ref{fig:ref1} and \ref{fig:ref2}.
%(Figure \ref{fig:refinement}, center left)

\begin{definition}
\label{def:Predicate}
An \emph{Admissible Stopping Criterion} (ASC) is a predicate $p$ which associates to each elementary triangle $T$ a boolean value $p(T)$, and which satisfies the following properties: 

\begin{itemize}
\item (Heredity)
Let $T',T''$ be the children of an elementary triangle $T$. If $p(T)$ holds, then $p(T')$ and $p(T'')$ also hold.
\item (Finiteness)
There exists a constant $s_p\geq 0$ such that $p(T)$ holds for any elementary triangle $T$ satisfying $s(T) \geq s_p$.
\end{itemize}
\end{definition}

%We denote by $\cT_0$ the mesh, illustrated on Figure \ref{fig:refinement}, containing the four elementary triangles $(T_i)_{1 \leq i \leq 4}$ (in trigonometric order) of non-zero vertices $(\pm 1,0)$ and $(0,\pm 1)$. 

The conjunction $p\wedge p'$ and the disjunction $p \vee p'$ of two ASCs $p,p'$ are clearly also ASCs. We write $p\Ra p'$ if $p(T) \Ra p'(T)$ for any elementary triangle $T$.

\begin{definition}
\label{def:TE}
Let $p$ be an ASC. We denote by $\cT(p)$ the collection of triangles obtained by recursively refining (i.e.\ replacing with their children) the elements of $\cT_0$, until each satisfies the predicate $p$.
We denote by $\cE(p)$ the collection of all elementary triangles which do not satisfy $p$.
\end{definition}

Definition \ref{def:Predicate} of an ASC $p$ is tailored so that the set $\cE(p)$ can be identified with four finite binary trees, which are subtrees of the four infinite binary trees of elementary triangles introduced in Point 2 of Lemma \ref{lem:parents}.  
The triangulation $\cT(p)$ consists of the outer leaves of these trees: $\cT(p)\cap \cE(p) = \emptyset$, but each triangle $T \in \cT(p)$ is the child of an element $R \in \cE(p)$. 
This is the main ingredient in the proof of following proposition.

%This interpretation of $\cE(p)$ as a binary tree is also the main ingredient in the proof of following proposition.

\begin{prop}
\label{prop:pCard}
%Consider the procedure of recursively refining the elements of $\cT_0$ until each of them satisfies a given ASC $p$.
\begin{enumerate}
\item
The recursive procedure described in Definition \ref{def:TE} 
%This procedure 
ends after a finite number of steps, and yields a mesh $\cT(p)$ which satisfies assumptions (i) and (ii) of Definition \ref{def:acuteMesh}. Furthermore 
\be
\label{upperTP}
\#(\cT(p)) = 4+\#(\cE(p)).
%\#(\cT(p)) = 4 + \sum_{1 \leq i \leq 4} \#(\cP_i(p)).
\ee
\item
If two ASCs $p,p'$ are such that $p\Ra p'$, then $\#(\cT(p)) \geq \#(\cT(p'))$.
\item 
For any two ASCs $p,p'$, one has $\#(\cT(p \wedge p')) \leq \#(\cT(p)) + \#(\cT(p'))$.
\end{enumerate}
\end{prop}

\begin{proof}
We denote by $(T_i)_{1 \leq i \leq 4}$ the four elements of the mesh $\cT_0$, and by $(\cP_i)_{1 \leq i \leq 4}$ the four infinite binary trees introduced in Point 2 of Lemma \ref{lem:parents}, see also Figure \ref{fig:refinement}.
The root of $\cP_i$ is the triangle $T_i$, for any $1 \leq i \leq 4$, and the children  of any $T\in \cP_i$ are those obtained by refining $T$.
%We denote by $(\cP_i)_{1 \leq i \leq 4}$  the four infininfinite binary trees which respective roots are $T_i$, and such that the children  of any $T\in \cP_i$ are those obtained by refining $T$. Such a tree structure is illustrated on Figure \ref{fig:refinement}. %See Figure \ref{fig:refinement} (center) for the first levels of such a tree.
For any ASC $p$ and any $1 \leq i \leq 4$ we denote 
\be
\label{def:cPi}
\cP_i(p) := \{T \in \cP_i; \, p(T) \text{ does not hold}\}.
\ee
The first point of Definition \ref{def:Predicate} implies that $\cP_i(p)$ is a (possibly empty) subtree of $\cP_i$: any $T' \in \cP_i(T)$ is either the root $T_i$, or the child of another $T\in \cP_i$.
The second point of the same definition, combined with \iref{sGrows}, implies that $\cP_i(p)$ is finite. 

The finiteness of the trees $\cP_i(p)$, implies that the refinement procedure ends after a finite number of steps. 
As already observed right after Definition \ref{def:refinement}, the collection $\cT(p)$ of triangles obtained at the end of this procedure, which is also the set of outer leaves of the finite binary trees $(\cP_i(p))_{1 \leq i \leq 4}$, is automatically a mesh satisfying Points (i) and (ii) of Definition \ref{def:acuteMesh}. 

%It follows from Lemma \ref{lem:parents}, and the fact observe above that an elementary triangles has at most one parent, that t
As observed in Point 2 of Lemma \ref{lem:parents}, the collection of all elementary triangles is the disjoint union of the trees $(\cP_i)_{1 \leq i \leq 4}$. Thus the subtrees $\cP_i(p)$ form a partition of the set $\cE(p)$.
Recalling that the number of leaves of a binary tree is one plus the number of its inner nodes, we obtain
\be
\label{cardTPi}
\#(\cT(p)) = \sum_{1\leq i \leq 4} \left(1+\#(\cP_i(p))\right) = 4+ \#(\cE(p)),
\ee
which concludes the proof of the first point.

The implication $p\Ra p'$ of two ASCs is equivalent to the reverse implication of the negations: $\neg p \Leftarrow \neg p'$, and thus to the inclusion $\cE(p) \supset \cE(p')$. If $p\Ra p'$ we thus obtain $\#(\cE(p)) \geq \#(\cE(p'))$, and therefore $\#(\cT(p)) \geq \#(\cT(p'))$, which establishes Point 2.

For any two ASCs $p,p'$, we have $\cE(p\wedge p') = \cE(p) \cup \cE(p')$. Hence $\#(\cE(p\wedge p')) \leq \#(\cE(p))+\#(\cE(p'))$, and therefore $\#(\cT(p \wedge p')) \leq \#(\cT(p)) + \#(\cT(p'))-4$, which concludes the proof of this proposition.
\end{proof}

We establish in the following lemma a first non-trivial estimate of the mesh cardinality $\#(\cT(p))$, in terms of the constant $s_p$ associated to the ASC $p$.

\begin{lemma}
\label{lemma:angleCard}
\begin{itemize}
\item
For any $s \in [1, \infty[$ one has 
\be
\label{sumLog}
\sum_{\substack{u \in \Z^2\\ 0 < \|u\| \leq s}} \frac 1 {\|u\|^2} \leq 8(1+ \ln s). 
\ee

\item
For each $u \in \Z^2$ denote
$$
\cE^+_u := \{v \in \Z^2; \, \|u\| < \|v\|, \, 0 \leq \<u,v\> < s_p, \ \det (u,v) = 1\}, 
$$
and define $\cE^-_u$ likewise, to the exception of the last constraint which is replaced with $\det(u,v) = -1$.
Then $\#(\cE^\ve_u) \leq s_p/\|u\|^2$, for $\ve \in \{+ , -\}$. 

\item
There exists a constant $C$ such that for any ASC $p$, with associated constant $s_p \geq 1$, one has 
\be
\label{cardTUpper}
\#(\cT(p)) \leq C s_p (1+\ln s_p).
\ee
\end{itemize}
\end{lemma}

\begin{proof}
We first establish \iref{sumLog}, and for that purpose we introduce the sup-norm $\|\cdot\|_\infty$ on $\R^2$ defined by $\|(x,y)\|_\infty := \max\{|x|,|y|\}$. Clearly $\|u\|_\infty \leq \|u\|$ for all $u\in \R^2$. For each $k \in \Z_+$ there exists precisely $(2k+1)^2$ elements $u \in \Z^2$ such that $\|u\|_\infty \leq k$. Hence for each integer $k \geq 1$ there exists precisely $(2k+1)^2-(2k-1)^2 = 8 k$ elements $u \in \Z^2$ such that $\|u\|_\infty = k$. Therefore
$$
\sum_{\substack{u \in \Z^2\\ 0 < \|u\| \leq s}} \frac 1 {\|u\|^2} \leq \sum_{\substack{u \in \Z^2\\ 0 < \|u\|_\infty \leq s}} \frac 1 {\|u\|_\infty^2} =\sum_{0<k\leq s} \frac{8 k}{k^2} \leq 8(1+ \ln s),
$$
which concludes the proof of \iref{sumLog}.

We next turn to the proof of the second point, and for that purpose we consider a fixed $u \in \Z^2$ such that $\cE^+_u$
is non-empty.
Let $v \in \cE^+_u$ be such that the scalar product $\<u,v\>$ is minimal. For any $v'\in E^+_u$ one has $\det (u,v'-v) = 1-1 = 0$, hence $v' = v+ \lambda u$ for some $\lambda \in \R$. Since $u,v,v' \in \Z^2$, and since $u$ has coprime coordinates (recall that $\det(u,v) = 1$), the scalar $\lambda$ must be an integer.
We thus have 
$$
\<u,v'\> = \<u,v\> + \lambda \|u\|^2 < s_p.
$$ 
Since $\<u,v\> \leq \<u,v'>$ we have $\lambda\geq 0$. Since $\<u,v\> \geq \|u\|^2$, using Point 1 of Lemma \ref{lem:parents}, we have $(1+\lambda) \|u\|^2 < s_p$. 
Hence $0 \leq \lambda < s_p/\|u\|^2-1$, and therefore $\#(\cE^+_u) \leq s_p/\|u\|^2$. Estimating  $\cE^-_u$ likewise, we conclude the proof of the second point.

Identifying an elementary triangle to its pair $(u,v)$ of non-zero vertices, ordered by increasing norm, we obtain 
\begin{equation}
\cE(p)\sm \cT_0 \subset \bigcup_{\substack{u \in \Z^2\\ \ve\in \{-,+\}}} \cE^\ve_u.
\end{equation}
%find that $\cE(p)\sm \cT_0$ injects in the union of the $\cE^\ve_u$, for $u \in \Z^2$ and $\ve \in \{+,-\}$. 
Furthermore the set $\cE^\ve_u$ is empty for $\|u\|>\sqrt{s_p}$, since using the second point of the proposition we find that its cardinal is strictly less than one. Hence %using the first point of this proposition
$$
\#(\cE(p)) \leq \#(\cT_0) + \sum_{u \in \Z^2} \left(\#(\cE_u^+)+ \#(\cE_u^-)\right) \leq 4+ \sum_{0<\|u\| \leq \sqrt{s_p}} \frac {2 s_p} {\|u\|^2} \leq 4+16 s_p (1+\ln s_p).
$$
Recalling that $\#(\cT(p)) = 4+\#(\cE(p))$, we conclude the proof of this proposition.
\end{proof}

\subsection{Mesh associated to an asymmetric norm}

We reformulate and study in this subsection, in Proposition \ref{prop:defMesh}, the algorithmic construction of the $F$-acute mesh $\cT(F)$ given in the introduction for each asymmetric norm $F$.
%We describe in Proposition \ref{prop:defMesh} of this section the construction of an $F$-acute mesh $\cT(F)$ for each asymmetric norm $F$. 
%We also introduce a way to smooth as
%While this construction does not require any assumption on the regularity of $F$, a significant part of the subsequent analysis is restricted to smooth asymmetric norms (outside of the origin), and extended to non-smooth ones by 
%In the subsequent analysis
%The unit circle is the set $S := \{u \in \R^2;\, |u| = 1\}$.

Our first lemma introduces a tool that will be frequently used in the rest of this paper: the approximation of an arbitrary asymmetric norm by smooth ones. % (outside of the origin).
For each $\theta \in \R$ we denote 
\be
\label{def:eTheta}
e_\theta := (\cos\theta,\sin \theta).
\ee
\begin{lemma}
\label{lem:cvNorm}
For any asymmetric norm $F$ on $\R^2$, there exists a sequence $(F_n)_{n \geq 1}$ of asymmetric norms such that
%\begin{itemize}
%\item
\begin{equation*}
%$F_n \to F$ locally uniformly on $\R^2$, as $n \to \infty$, %on the unit circle $S$ 
F_n \to F \text{ locally uniformly on } \R^2, \text{ as } n \to \infty, %on the unit circle $S$ 
\end{equation*}
%\end{itemize}
and for all $n \geq 1$:
\begin{itemize}
\item
$F_n \in C^\infty(\R^2\sm\{0\})$.
\item
$\kappa(F_n) \leq \kappa(F)$.
\item If $F$ is symmetric, then so is $F_n$.
\end{itemize}
\end{lemma}

\begin{proof}
%We construct the asymmetric norms $F_n$ by convolution
%We also denote  $e_\theta := (\cos\theta,\sin \theta)$.
We define the asymmetric norm $F_n$ through polar coordinates and by convolution: for each $r\geq 0$ and each $\vp\in \R$, 
%For each $n \geq 1$ we define for all $u\in \R^2$
%$$
%F_n(u) := \int_\sR F^{\htheta}( u ) \, \mu_n (\theta) \, d\theta.
%$$
%Clearly $F_n(u)$ 
%hence $\kappa(F_n) \leq \kappa(F)$.
$$
F_n(r e_\vp) := r \int_\sR F(e_\theta) \mu_n(\vp-\theta) d \theta.
$$
We denoted by $\mu_n$ the mollifier $\mu_n(\theta) := n \mu(n\theta)$, for all $n \geq 1$, where $\mu(\theta) := e^{-\theta^2}/\sqrt{\pi}$, for all $\theta \in \R$.
%We proceed by convolution, and for that purpose we introduce a mollifier: $\mu(\theta) := e^{-\theta^2}/\sqrt{\pi}$ and $\mu_n(\theta) := n \mu(n\theta)$ for all $\theta\in \R$ and all $n \geq 1$.
The four announced properties are immediate.
\end{proof}

\begin{figure}
\begin{center}
\includegraphics[width=15cm]{\pathPic/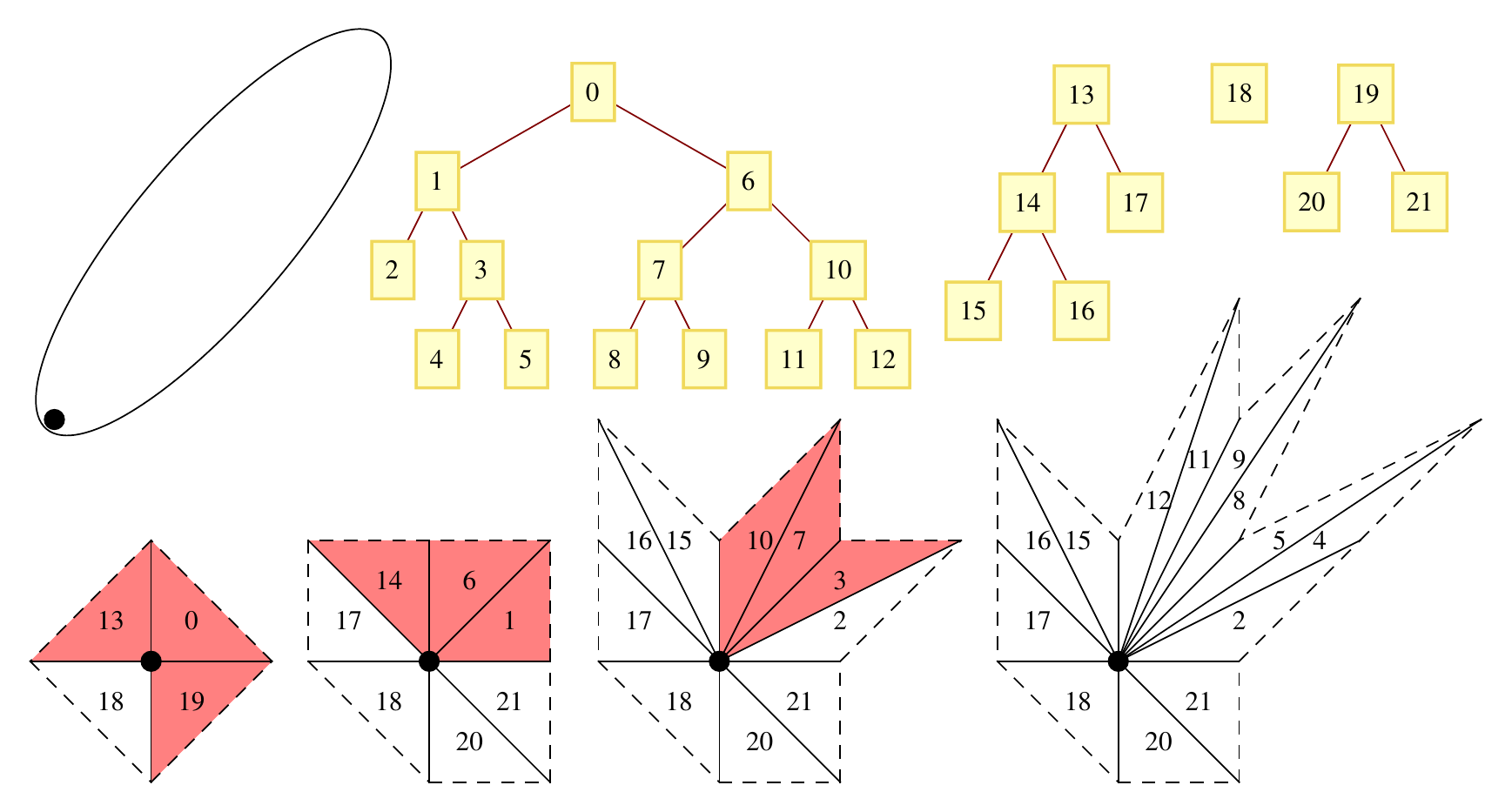}
\end{center}
\caption{
\label{fig:ref1}
Unit ball of an asymmetric norm $F$ of anisotropy ratio $\kappa(F)=20$ (top left). Generation of $\cT(p_F)$ by recursive bisection (bottom, left to right). The non-zero vertices of colored triangles do not form an $F$-acute angle, hence these triangles are refined. Colored triangles constitute the set $\cE(p_F)$, see definition \ref{def:TE}; they form the inner nodes of four finite binary trees of triangles, while the elements of $\cT(p_F)$ are the leaves. 
%The mesh $\cT(p_F)$ corresponds to the leaves of four finite binary trees of triangles, see Lemma \ref{lem:parents}.
}
\end{figure}

We presented in the introduction of this paper the construction of a mesh $\cT(F)$, associated to each asymmetric norm $F$. This definition is tied to the mesh generation method by recursive refinement presented in the previous subsection, since we claim that 
$$
\cT(F) = \cT(p_F),
$$
where for an elementary triangle $T$ of non-zero vertices $u,v$, the predicate value $p_F(T)$ stands for the test ``$u,v$ form an $F$-acute angle'', see the next proposition. %is an admissible stopping criterion (ASC, see Definition \ref{def:Predicate}) introduced in the next proposition. 
Indeed, denote by $(T_k)_{k=0}^K$ the elementary triangles defined by the consecutive pairs $(u,v)$ of vectors subject to the test ``If $u,v$ form an $F$-acute angle'', in the construction of $\cT(F)$. These triangles, and their order of appearance, are shown on Figure \ref{fig:ref1}. The sequence $(T_k)_{k=0}^K$ constitutes an in-order transversal of the four binary trees in $\cE(p_F) \cup \cT(p_F)$, see again Figure \ref{fig:ref1} and the proof of Proposition \ref{prop:pCard}.

Let us observe that Point (i) and (ii) of Definition \ref{def:acuteMesh}, of $F$-acute meshes, hold by construction for any mesh of the form $\cT(p)$, where $p$ is an ASC, as observed right above Lemma \ref{lem:parents}. %Definition \ref{def:Predicate}. 
On the other hand the predicate $p_F$ is designed so as to enforce Point (iii) of this definition.

\begin{prop}
\label{prop:defMesh}
Let $F$ be an asymmetric norm on $\R^2$.
\begin{itemize}
\item 
%Let $T$ be an elementary triangle of vertices $0,u,v$. The triangle $T$ is $F$-acute if 
Two vectors $u,v\in \R^2\sm\{0\}$ form an $F$-acute angle if  $\<u,v\> \geq 0$ and 
\be
\label{kappaFAngle}
\kappa(F) \sin |\varangle(u,v)| \leq 1.
\ee
\item
The predicate $p_F$ defined for any elementary triangle $T$ by 
\be
\label{def:pF}
p_F(T) \text{ holds if and only if the non-zero vertices of $T$ form an $F$-acute angle},
\ee
is an ASC, with associated constant 
$
s_{p_F} \leq \sqrt{\kappa(F)^2-1}.
$
\item Any vertex $u$ of $\cT(F)$ satisfies $\|u\| \leq 2 \kappa(F)$.
\end{itemize}
\end{prop}

\begin{proof}
First Point.
In order to establish \iref{kappaFAngle}, we restrict in a first time our attention to asymmetric norms $F$ which are smooth: $F\in C^1(\R^2 \sm \{0\})$.
In that case Proposition \ref{prop:phi} (below, but proved independently) shows in \iref{PhiBounded} that 
%$
%\kappa(F) \cos \left[ \varangle (e_\theta, \nabla F(e_\theta))\right] \leq 1.
%$
$
\kappa(F) \cos \varangle (e_\theta, \nabla F(e_\theta)) \leq 1,
$
for any $\theta \in \R$.
Since $F$ is homogeneous, one has $\nabla F(u) = \nabla F(\lambda u)$ for any $\lambda>0$ and any $u \in \R^2 \sm \{0\}$, and therefore 
$$
\kappa(F) \cos \varangle (u, \nabla F(u)) \leq 1.
$$
Assuming \iref{kappaFAngle} we thus obtain
%If $\vp \in \R$ is such that $\kappa(F) \sin |\theta - \vp| \leq 1$, then 
$$
|\varangle(v, \nabla F(u))| \leq |\varangle(v, u)| + |\varangle (u, \nabla F(u))| \leq \arcsin(1/\kappa(F)) + \arccos(1/\kappa(F)) = \pi/2,
$$
%$$
%|\varangle(e_\vp, \nabla F(e_\theta))| \leq |\varangle(e_\theta, e_\vp)| + |\varangle (e_\theta, \nabla F(e_\theta))| \leq \arcsin(1/\kappa(F)) + \arccos(1/\kappa(F)) \leq \pi/2,
%$$
and therefore $\<v,\nabla F(u)\> \geq 0$. Likewise $\<u,\nabla F(v)\> \geq 0$, hence using Point 1 of Lemma \ref{lem:acuteCriterion} we conclude that the vectors $u,v$ form an $F$-acute angle.

Now let us consider an arbitrary, possibly non-smooth, asymmetric norm $F$, two vectors $u,v$ satisfying \iref{kappaFAngle}, and a sequence $(F_n)_{n \geq 0}$ of asymmetric norms as described in Lemma \ref{lem:cvNorm}. It follows from the above argument that the vectors $u,v$ form an $F_n$-acute angle for each $n\geq 0$. Since Definition \ref{def:acute} of $F$-acuteness only involves non-strict inequalities, we obtain taking the limit that $u,v$ form an $F$-acute angle.

%Point \ref{PredicateRefinement} of Definition \ref{def:Predicate}
%First point. To Do.
Second Point. We need to check that $p_F$ satisfies the heredity and finiteness properties which characterize ASCs, see Definition \ref{def:Predicate}. 
Consider two vectors $u,v\in \R^2\sm \{0\}$ which form an $F$-acute angle. For each $\delta \geq 0$ we obtain 
\begin{eqnarray*}
F(u+v+\delta u) &=& F((1+\delta)(u+v) - \delta v)\\
& \geq & (1+\delta) F(u+v) - \delta F(v)\\
& = &F(u+v) + \delta( F(u+v)-F(v))\\
& \geq & F(u+v),
\end{eqnarray*}
where we used the triangular inequality in the second line, and the fact that $u$ and $v$ form an $F$-acute angle in the last. For the same reason $F(u+\delta(u+v)) = (1+\delta) F(u+\delta v/(1+\delta)) \geq (1+\delta) F(u) \geq F(u)$. Therefore $u$ and $u+v$ form an $F$-acute angle, and likewise $v$ and $u+v$ form an acute angle. As a result, if the predicate $p_F$ holds for an elementary triangle, then it also holds for its children. This establishes the heredity property in Definition \ref{def:Predicate}. 
For the finiteness property we consider an elementary triangle $T$, of non-zero vertices $u,v$, such that $s(T) \geq \sqrt{\kappa(F)^2-1}$. It follows from \iref{sinScal} that $\sin|\varangle(u,v)| \leq 1/\kappa(F)$, hence the first part of this proposition shows that $u,v$ form an $F$-acute angle. Thus $p_F(T)$ holds, which establishes the finiteness property of Definition \ref{def:Predicate}, and concludes the proof of the second point. % of this proposition.

Third Point. A triangle $T \in \cT(F)$ either belongs to $\cT_0$, or is the child of a triangle $T'$ in $\cE(p_F)$ which \emph{does not} satisfy the predicate $p_F$. In the first case there is nothing to prove, while in the second case the non-zero vertices $u,v$ of $T'$ satisfy 
$$
\max\{\|u\|^2, \|v\|^2\} \leq \|u\|^2\|v\|^2 = s(T)^2+1 \leq s_{p_F}^2+1 \leq \kappa(F)^2.
$$ 
We used the fact that $\min \{\|u\|,\|v\|\} \geq 1$, since these vectors have integer coordinates, and \iref{sinScal}.
Thus $\|u\|$ and $\|v\|$ are bounded by $\kappa(F)$, and therefore $\|u+v\| \leq 2 \kappa(F)$. This concludes the proof since the non-zero vertices of the triangle $T$ belong to $\{u,u+v,v\}$.
\end{proof}

At this point, we can establish the worst case analysis presented in Proposition \ref{prop:WorstCase}, except for \iref{CardSym} which is proved later in Corollary \ref{corol:CardSym}.
%At this point, we can establish the two properties announced in the end of Proposition \ref{prop:WorstCase}. The proof of \iref{CardAsym} and \iref{CardSym} will be given later.
\begin{corollary}
\label{corol:WorstCase}
\begin{itemize}
\item
Let $\kappa \geq 1$, let $p_\kappa$ be the predicate ``$s(T) \geq \kappa$'', and let $\cT_\kappa := \cT(p_\kappa)$. 
Let also $F$ be an asymmetric norm such that $\kappa(F) \leq \kappa$. Then $\cT_\kappa$ is $F$-acute and
\begin{equation*}
\#(\cT(F)) \leq \#(\cT_\kappa) \leq C \kappa (1+\ln \kappa).
\end{equation*}
%Then $p_\kappa \Ra p_F$ for each asymmetric norm $F$ such that $\kappa(F) \leq \kappa$, and the mesh $\cT_\kappa := \cT(p_\kappa)$ has cardinality $\#(\cT_\kappa) \leq C \kappa (1+\ln \kappa)$.
\item
For each $\tau \geq 1$, let $F_\tau$ be the anisotropic euclidean norm defined by the positive definite matrix 
$
M_\tau := \left(
\begin{array}{cc}
1 & \tau \\
\tau & 2\tau^2
\end{array}
\right)
$. Then $|\kappa(F_\tau) - 2\tau| \leq 1$ and $\#(\cT(F_\tau)) \geq 6+2 \lfloor \tau \rfloor$.
%\item
%For each $\kappa \geq 1$, let $\cT_\kappa$ be the mesh defined defined by the ASC ``$s(T) \geq \kappa$''. Then $\cT_\kappa$ is an $F$-acute mesh for any asymmetric norm $F$ such that $\kappa(F) \leq \kappa$, and $\#(\cT_\kappa) \leq C \kappa(1+\ln \kappa)$.
\end{itemize}
\end{corollary}

\begin{proof}
First Point.
For any asymmetric norm $F$ such that $\kappa(F) \leq \kappa$, we have $s_{p_F} \leq \sqrt{\kappa(F)^2-1} \leq \sqrt{\kappa^2-1} \leq \kappa$. Hence $p_\kappa \Ra p_F$, which implies simultaneously that $\#(\cT(F)) \leq \#(\cT_\kappa)$ (using Point 2 of Proposition \ref{prop:pCard}) and that $\cT_\kappa$ is $F$-acute (since $p_F$ holds for all the elements of $\cT_\kappa$). The upper bound on $\#(\cT_\kappa)$ was proved in Lemma \ref{lemma:angleCard}.

Second point. The $2 \times 2$ symmetric matrix $M_\tau$ is positive definite since its trace and determinant are both positive.
Denoting by $0<\lambda^2 \leq \mu^2$ the eigenvalues of $M_\tau$, where $\lambda$ and $\mu$ are positive, one has 
$$
\kappa(F_\tau)=\frac \mu \lambda, \quad \Tr(M_\tau) = \lambda^2+\mu^2 = 2\tau^2+1, \quad \det (M_\tau) = \lambda^2\mu^2 = 2 \tau^2 -\tau^2 = \tau^2.
$$
Hence denoting $\kappa := \kappa(F_\tau)$ 
$$
\frac 1 \kappa + \kappa = \frac \lambda \mu+ \frac \mu \lambda = \frac{\Tr(M_\tau)}{\sqrt{\det M_\tau}} = \frac{2 \tau^2+1} \tau = 2\tau + \frac 1 \tau.
$$
Therefore $|\kappa-2 \tau| = |\tau^{-1}-\kappa^{-1}| \leq 1$, since $\kappa \geq 1$ and $\tau \geq 1$.

We next observe that the elementary triangle of vertices $(1,0)$ and $(-r,1)$ is not $F_\tau$-acute for $0 \leq r < \tau$, since $\<(r,-1),M_\tau (1,0)\> = \<(r,-1),(1, \tau)\> = r-\tau <0$. Considering these triangles and the symmetric ones with respect to the origin, we obtain $2(1+\lfloor \tau \rfloor)$ non $F_\tau$-acute elementary triangles. Hence $\#(\cE(p_{F_\tau})) \geq 2 (1+\lfloor \tau \rfloor)$, and therefore $\#(\cT(F_\tau)) =4+\#(\cE(p_{F_\tau})) \geq 6+2 \lfloor \tau \rfloor$ using \iref{upperTP}, which concludes the proof.
\end{proof}

In the next section, we estimate the cardinality of $\cT(F)$
%(and the average cardinality of $\cT(F^\htheta)$, $\theta \in [0,2\pi]$)
for asymmetric norms $F$ which are smooth on $\R^2 \sm \{0\}$. These results are transferred to arbitrary asymmetric norms, using the approximation result Lemma \ref{lem:cvNorm} and the following lemma.

\begin{lemma}
\label{lem:limCard}
Let $F$ be an asymmetric norm, and let $(F_n)_{n \geq 0}$ be a sequence of asymmetric norms such that $F_n \to F$ locally uniformly as $n \to \infty$.
Then 
\begin{eqnarray}
\label{lim:cardT}
\#(\cT(F)) &\leq& \liminf_{n \to \infty} \#(\cT(F_n)),\\
\label{lim:intCardT}
\int_0^{2\pi} \#(\cT(F^\htheta))\, d\theta &\leq& \liminf_{n \to \infty} \int_0^{2\pi} \#(\cT(F_n^\htheta)) \, d\theta.
\end{eqnarray}
%\begin{itemize}
%
%\end{itemize}
\end{lemma}
\begin{proof}
To avoid notational clutter, we denote $\cE(F) := \cE(p_F)$, and $\cE(F_n) := \cE(p_{F_n})$, see Definition \ref{def:TE}.
If an elementary triangle $T$ belongs to $\cE(F)$, then it belongs to $\cE(F_n)$ for all $n$ sufficiently large, since $F$-acuteness is a closed condition, see Definition \ref{def:acute}. Hence 
$$
\cE(F) \subset \bigcup_{N \geq 0} \bigcap_{n \geq N} \cE(F_n).
$$
This immediately implies that $\#(\cE(F)) \leq \liminf_{n \to \infty} \#(\cE(F_n))$, by applying Fatou's lemma to the characteristic functions of $\cE(F)$ and $\cE(F_n)$. Inequality \iref{lim:cardT} then follows from the identity $\#(\cT(F)) = 4+\#(\cE(F))$, see \iref{upperTP}. %and likewise for $F_n$,
%We denote by $\cE$ the collection of all elementary triangles $T$, of non-zero vertices $u,v$, such that the set $\{n \geq 0; \, u,v \text{ form a } F_n\text{-acute angle}\}$ is empty or finite. By construction, there exists for each $T\in \cE$ an integer $n(T)$, such that $T \in \cE(F_n)$ for all $n \geq n(T)$. Hence $\liminf_{n \to \infty} \#(\cE(F_n)) \geq \#(\cE)$, and therefore using \iref{upperTP} the right hand side of \iref{lim:cardT} is larger or equal to $4+\#(\cE)$. 

%On the other hand, if an elementary triangle $T$ of non-zero vertices $u,v$ does not belong to $\cE$, then there exists a strictly increasing $\vp : \N \to \N$ such that $u,v$ form an $F_{\vp(n)}$-acute angle for each integer $n$. Since Definition \ref{def:acute} only involves non-strict inequalities, we find taking the limit that $u,v$ form an $F$-acute angle. Hence $\cE(F) \subset \cE$, and therefore $\#(\cT) \leq 4+ \#(\cE)$, which concludes the proof \iref{lim:cardT}.

The second estimate \iref{lim:intCardT} immediately follows from the first one \iref{lim:cardT}, by observing $F_n^\htheta \to F^\htheta$ locally uniformly as $n \to \infty$ for any $\theta\in [0,2 \pi]$, and applying Fatou's lemma on this interval.
\end{proof}

\section{Average Complexity}
\label{sec:cardTF}

%We estimate in this section the cardinality of the mesh $\#(\cT(F))$, constructed 
This section is devoted to the estimate of the cardinality $\#(\cT(F))$ of the stencils used in the FM-ASR, and of the average value of $\#(\cT(F^\htheta))$, as $\theta\in [0,2\pi]$.
Estimates are obtained for increasingly general types of (asymmetric) norms $F$: anisotropic euclidean norms in the first subsection, %which are given by a positive definite matrix, 
symmetric norms in the second, and finally asymmetric norms in the third. Each subsection builds on the estimate of the former one, hence they are not independent.

\subsection{Anisotropic euclidean norms}

An anisotropic euclidean norm $F$, is a norm given by a symmetric positive definite matrix $M$: for all $u \in \R^2$, $F(u) := \sqrt{u^\trans M u}$.
Our first lemma shows that the triangles refined during the construction of $\cT(F)$ %where $F$ is an anisotropic euclidean norm defined by a symmetric positive definite matrix $M$, 
are aligned with the eigenspace associated to the small eigenvalue of $M$, see also Figure \ref{fig:ref2}.

\begin{lemma}
\label{lem:eigenvec}
\begin{itemize}
\item
Let $F$ be an anisotropic euclidean norm, given by a matrix $M\in S_2^+$. If the non-zero vertices of an elementary triangle $T$ do not form an $F$-acute angle, then $T$ contains an eigenvector for the smallest eigenvalue of $M$ in its interior.
%\begin{itemize}
%\item (*with T?*)
%Let $M\in S_2^+$ and $u,v\in \R^2 \sm \{0\}$ be such that $\<u,v\> \geq 0$ and $u^\trans M v <0$. Then there exists $t\in ]0,1[$ such that $tu+(1-t)v$ is an eigenvector associated to the small eigenvalue of $M$.
\item
For any anisotropic euclidean norm $F$ one has $\#(\cT(F)) \leq 6+2 \kappa(F)$.%\sqrt{\kappa(F)^2-1}$ 
\end{itemize}
\end{lemma}

\begin{proof}
First Point. Let $0<\lambda \leq \mu$ the eigenvalues of $M$, and let $e$ a normalized eigenvector of $M$ associated to the eigenvalue $\lambda$. Let also $u,v$ be the non-zero vertices of $T$. Then
$$
\<u, M v\> = \lambda \<u,e\>\<v,e\> + \mu \det(u,e) \det(v,e) = \lambda \<u,v\> + (\mu-\lambda) \det(u,e) \det(v,e),
$$
where we used the identity $\<u,v\>= \<u,e\>\<v,e\> + \det(u,e) \det(v,e)$. 
Since $u,v$ do not form an $F$-acute angle, we have $\<u,M v\><0$ using Point 2 of Lemma \ref{lem:acuteCriterion}.
On the other hand $\<u,v\> \geq 0$. It follows that $\det(u,e) \det(v,e) < 0$, and therefore $\det(u,e)$ and $\det(v,e)$ are non-zero and have opposite signs. Hence by continuity (or linearity) there exists $t\in (0,1)$ such that $\det(t u+(1-t) v, e)=0$, which concludes the proof of this point.
% of the first point.
%

We next turn to the proof of second point, and for that purpose we adopt the notations of Proposition \ref{prop:pCard} and consider the four trees $(\cP_i(p_F))_{1 \leq i \leq 4}$. % (see Definition \ref{def:cPi}, for $\cP_i$ and Proposition \ref{prop:defMesh} for $p_F$).
It follows from the first point of this lemma that two of these trees are empty, and that the other two have a single branch, see also Figure \ref{fig:ref2}. 
%To see that these trees have a single branch, one may alternatively observe that $v^\trans M (u+v) \geq 0$ if $u^\trans M u \leq v^\trans M v$, for $u,v\in \R^2$ and $M \in S_2^+$, and use  the characterization of acuteness obtained in the second point of Lemma \ref{lem:acuteCriterion}.
The number of elements of these single branched trees is bounded by $1+s_{P_F} = 1+\sqrt{\kappa(F)^2-1} \leq 1+\kappa(F)$, using \iref{sGrows} and the second point of Proposition \ref{prop:defMesh}.
We finally obtain using \iref{cardTPi} that 
$
\#(\cT(F)) \leq 4+2 \times 0 + 2\left(1+\kappa(F)\right) = 6+2\kappa(F) %\sqrt{\kappa(F)^2-1},
$
which concludes the proof.
\end{proof}

\begin{figure}
\begin{center}
\includegraphics[width=3cm]{\pathPic/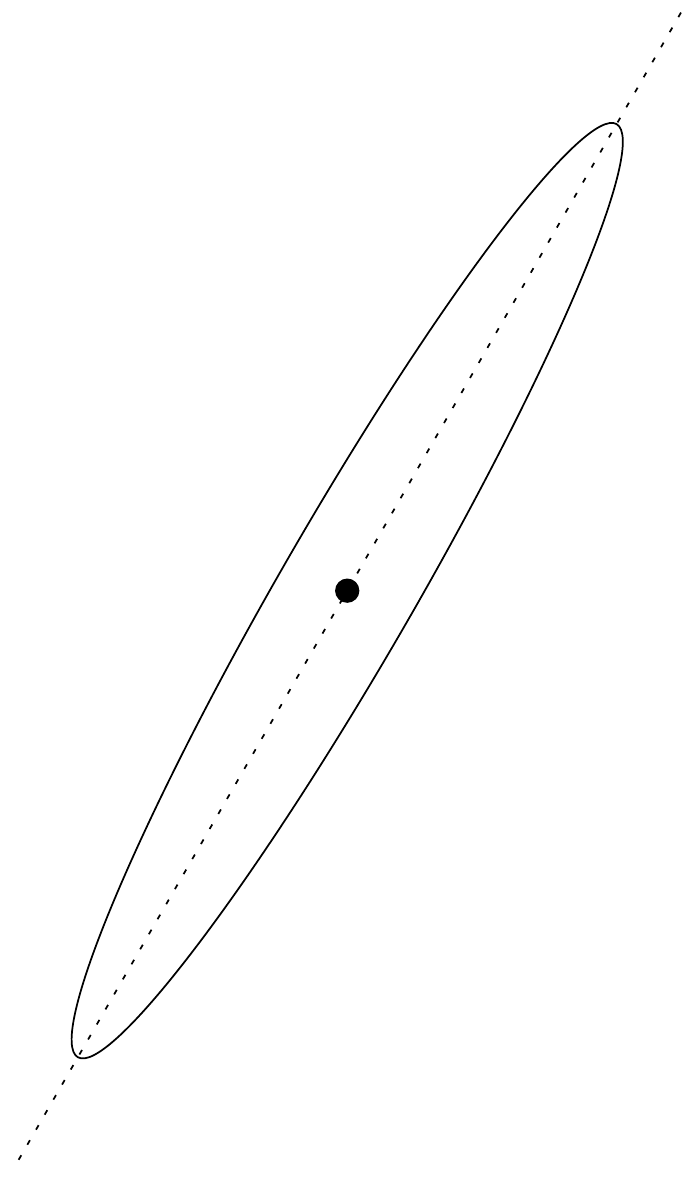}
\includegraphics[width=3cm]{\pathPic/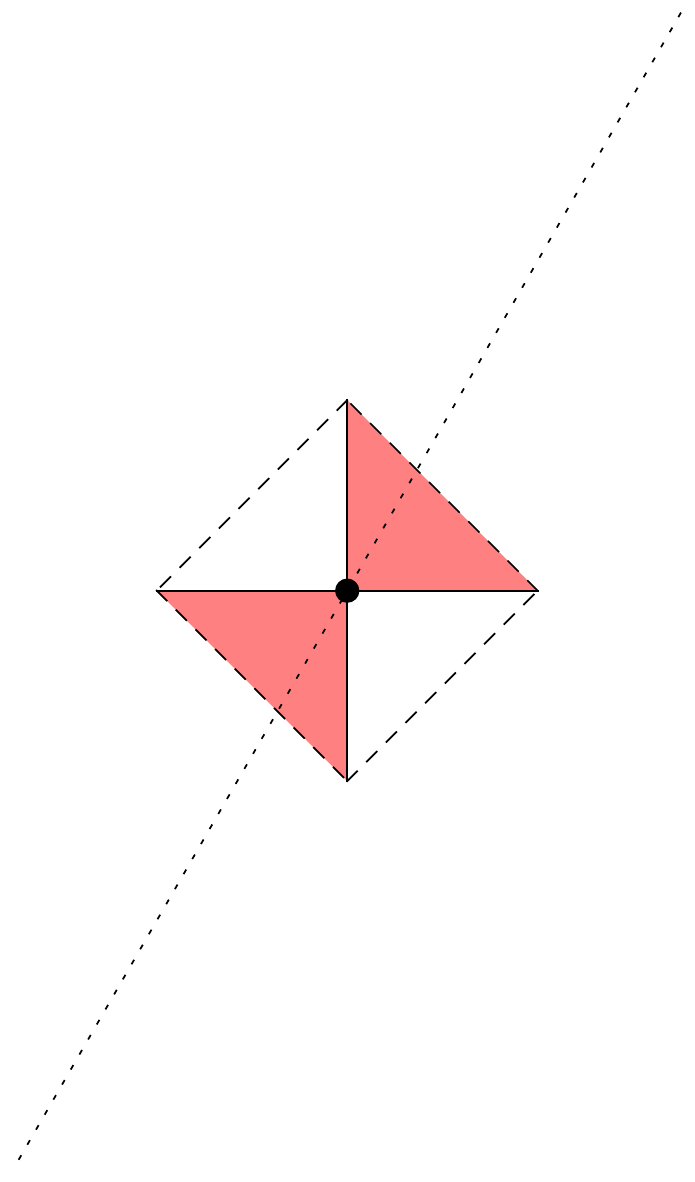}
\includegraphics[width=3cm]{\pathPic/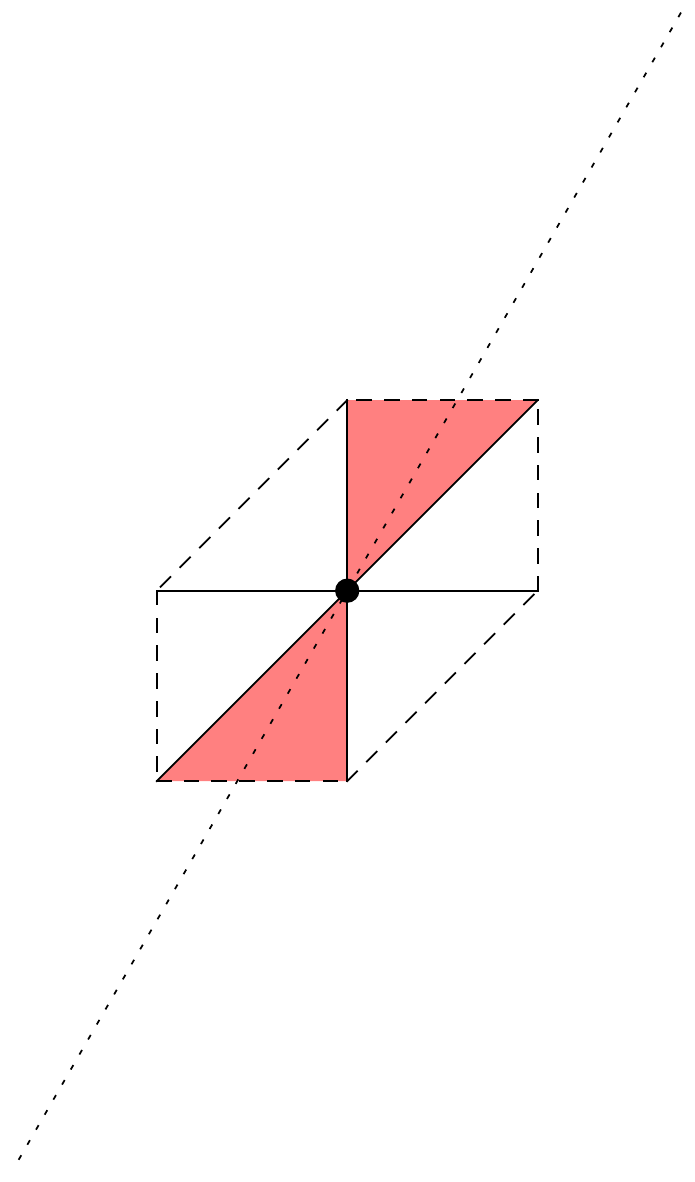}
\includegraphics[width=3cm]{\pathPic/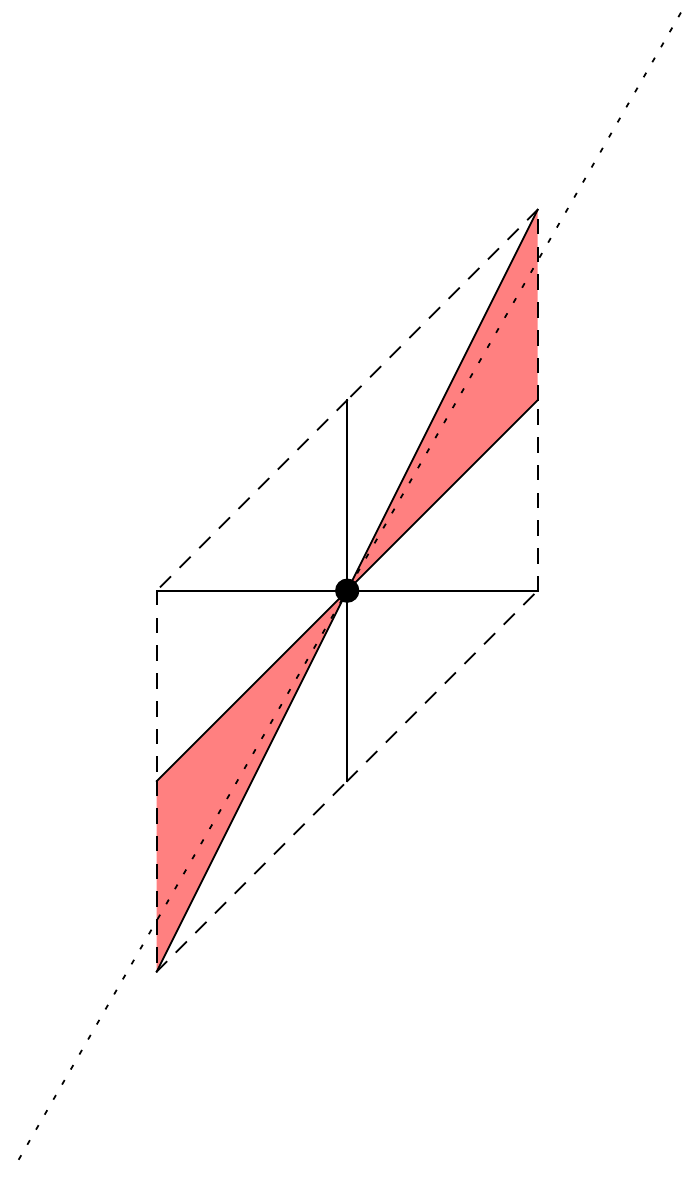}
\includegraphics[width=3cm]{\pathPic/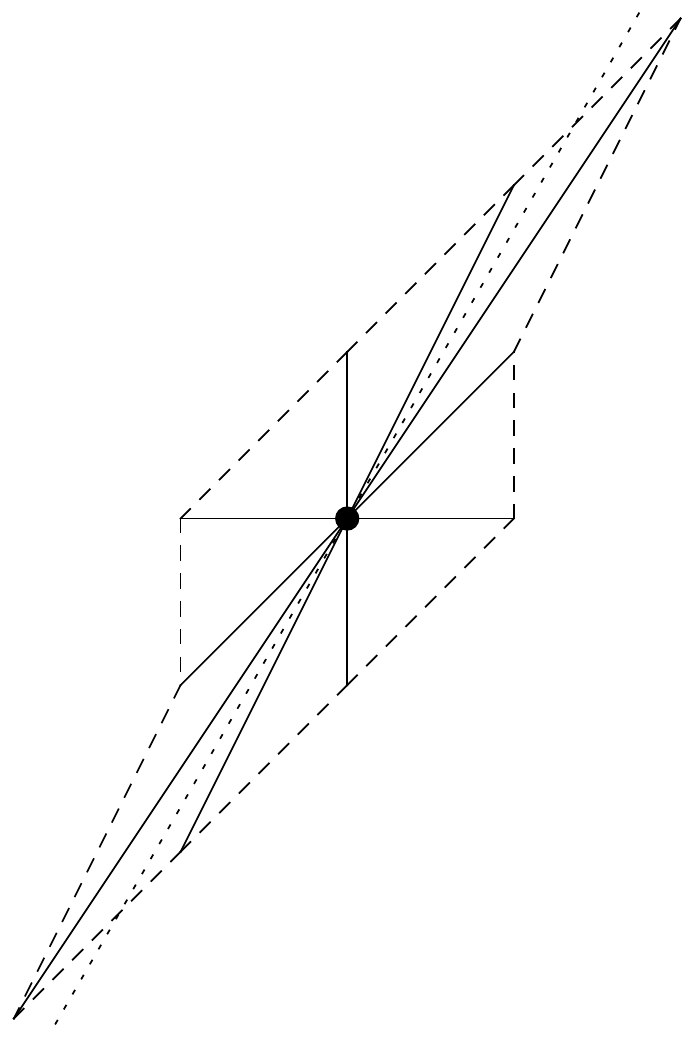}
\end{center}
\caption{
\label{fig:ref2}
Unit ball $\{u; \, F(u) \leq 1\}$ of a norm $F$ given by a positive definite matrix $M$, of anisotropy ratio $\kappa(F)=8$ (left). Eigenspace associated to the small eigenvalue of $M$ (dotted line). Generation of $\cT(F)$ by recursive bisection (second left to right). All refined triangles (colored) contain an eigenvector associated to the small eigenvalue of $M$ in their interior.
%Unit ball of an asymmetric norm $F$ of anisotropy ratio $\kappa(F)=20$ (left). Generation of $\cT(F)$ by recursive bisection (right). The non-zero vertices of colored triangles do not form an acute angle, hence these triangles are refined
}
\end{figure}

%We consider in this section a fixed constant $\kappa \geq 1$, and we denote by $M_\theta$, $\theta\in [0, 2 \pi]$, the symmetric positive definite matrices: % defined by 
%\be
%\label{def:M0}
%M_0 := 
%\left(
%\begin{array}{cc}
%1/\kappa & 0\\
%0 & \kappa
%\end{array}
%\right),
%\qquad
%M_\theta := R_\theta M_0 R_\theta^\trans,
%\ee
%where $R_\theta$ denotes the rotation matrix of angle $\theta$. 
%We denote by $F^0$ the norm defined by the positive definite matrix $M_0$, and we observe that $F^\theta$ is defined by the matrix $M_\theta$ for each $\theta \in \R$.

%For each $\theta\in \R/2\pi \Z$ we denote by $(e_\theta,e_\theta^\perp)$ the rotation of the canonical basis of $\R^2$ by the angle $\theta$ : $e_\theta := (\cos\theta,\sin\theta)$ and $e_\theta^\perp :=  (-\sin\theta,\cos\theta)$.

%For any elementary triangle $T$, of non-zero vertices $u,v$, any any $M\in S_2^+$, we recall that $T$ is $\|\cdot \|_M$-acute if and only if $\< u,M v\> \geq 0$, see Lemma \ref{lem:acuteT} .

%We introduce in the next definition some notations that 

\begin{definition}[The following definitions are restricted to this section]
\label{def:FEucl}
We consider a fixed constant $\kappa\geq 1$, and denote by 
$F$ the norm defined by the diagonal matrix $D$ of entries $(\kappa^{-1}, \kappa)$, in such way that $F(x,y) = \sqrt{\kappa^{-1}x^2+\kappa y^2}$.
%Let $\kappa \geq 1$ be a fixed constant, and let 
%$F$ the anisotropic euclidean norm defined by the diagonal matrix $D$ of entries $(\kappa^{-1}, \kappa)$, in such way that $\kappa(F) = \kappa$. 
\begin{itemize}
\item
For each $\theta \in [0,2\pi]$ the norm $F^\theta$ is of anisotropic euclidean type, defined by the matrix $M_\theta = R_\theta D R_\theta^\trans$, and satisfies $\kappa(F^\theta) = \kappa(F) = \kappa$.
\item
We denote by $\cE_\theta$, $\theta \in [0,2\pi]$, the collection of elementary triangles which non-zero vertices do not form an $F^\theta$-acute angle. In other words $\cE_\theta := \cE(p_{F^\theta})$.
%For each angle $\theta \in [0, 2 \pi]$, we denote by $\cE_\theta := \cE(p_{F^\theta})$, see Definition \ref{def:TE} and Proposition \ref{prop:defMesh}. In other words $\cE_\theta$ is the collection of all elementary triangles $T$, of non-zero vertices $u,v$ which do not form an $F^\theta$-acute angle, or equivalently such that $\<u,M_\theta v\> < 0$.
\item
For each elementary triangle $T$, we define $I_T := \{\theta \in [0,2\pi]; \, T \in \cE_\theta\}$. %of non-zero vertices $u,v$, we denote by $I_T$ the set of all $\theta \in [0,2 \pi]$ such that $\<u,M_\theta v\> < 0$.
\end{itemize}
\end{definition}
It follows from \iref{upperTP}, that for any $\theta \in [0, 2 \pi]$ one has 
\be
\label{ineqTE}
\#(\cT(F^\htheta)) = 4 + \#(\cE_\theta).
\ee
Furthermore, we have by construction 
\be
\label{sumElemT}
\int_0^{2 \pi} \#(\cE_\theta) \, d\theta = \sum_T |I_T|,
\ee
where $T$ ranges over all elementary triangles, and $|I_T|$ denotes the Lebesgue measure of $I_T$.% \subset [0,2 \pi]$.

\begin{prop}
\label{prop:SumAu}
For any fixed $u\in \Z^2 \sm \{0\}$, let $A_u$ the collection of all elementary triangles $T$, containing $u$ as a vertex, and such that the other non-zero vertex $v$ satisfies $\|u\| \leq \|v\|$. Then
%One has $A_u = \emptyset$ if $\|u\| \geq \sqrt{\kappa}$, 
%We have
\be
\label{sumIT}
\sum_{T\in A_u} |I_T| \leq \frac{C(1+\ln \kappa)}{\|u\|^2},
\ee
and furthermore this sum equals $0$ if $\|u\| \geq \sqrt \kappa$. %equals zero
%where the sum ranges over all elementary triangles containing $u$ as a vertex.
\end{prop}

\begin{proof}

%Let $T \in A_u$, of vertices $0,u,v$, and let $\theta \in I_T$. We write 
%$$
%u = u_\theta e_\theta+u_\theta^\perp e_\theta^\perp, \quad  v = v_\theta e_\theta+v_\theta^\perp e_\theta^\perp, 
%$$
%in such way that 
%$$
%0\leq \<u,v\> = u_\theta v_\theta + u_\theta^\perp v_\theta^\perp, \quad 0 > \<u,v\>_{M_\theta} = (1/\kappa) u_\theta v_\theta + \kappa u_\theta^\perp v_\theta^\perp.
%$$
%It follows that $u_\theta^\perp v_\theta^\perp <0$, and therefore that
Let $T$ be an elementary triangle such that $I_T\neq \emptyset$, and let $u,v$ be the non-zero vertices of $T$, with $\|u\| \leq \|v\|$. Since $\kappa(F^\htheta) = \kappa$ for any $\theta \in \R$, we obtain using  \iref{kappaFAngle} and \iref{def:angle}
%Let $T$ be an elementary triangle of non-zero vertices $u,v$, with $\|u\| \leq \|v\|$. Let $\theta \in I_T$ if this set is not empty, be such that $u,v$ do not form an $F^\theta$ acute angle. Then using \iref{kappaFAngle} and \iref{def:angle} we obtain 
\be
\label{kappau2}
\frac 1 \kappa %= \frac 1 {\kappa(F^\htheta)} 
< \sin |\varangle(u,v)| = \frac 1 {\|u\| \|v\|} \leq \frac 1 {\|u\|^2},
\ee
and therefore $\|u\| < \sqrt{\kappa}$. 
It follows as announced that \iref{sumIT} is zero if $\|u\| \geq \sqrt \kappa$.

It follows from Lemma \iref{lem:eigenvec} that $t u+(1-t) v$ is proportional to $e_\theta$, for some $t \in ]0,1[$. Hence 
\be
\label{ITUV}
|I_T| \leq 2 |\varangle(u,v)| = 2 \arcsin\left(\frac 1 {\|u\| \|v\|}\right) \leq \frac \pi {\|u\| \|v\|},
\ee
where we used the concavity estimate $\sin(\pi x/2) \geq  x$ for $x \in [0, 1]$.

We denote by $v_\ve$, for $\ve \in \{-1,1\}$, the element of $\Z^2$ for which the scalar product $\<u, v_\ve\>$ is non-negative and minimal, under the constraint that $\det(u,v_\ve) = \ve$ and $\|v_\ve \| \geq \|u\|$.

If $\det (u,v) = \ve$, then $v-v_\ve = \lambda u$ for some $\lambda\in \R$. Observing that $u,v,v_\ve$ have integer coordinates, and that $u$ has coprime coordinates, since $|\det(u,v)|=1$, we obtain that $\lambda$ is an integer. The scalar $\lambda$ is non-negative since $\<u,v_\ve\> \leq \<u,v\> = \<u,v_\ve\> + \lambda \|u\|^2$. Last we observe using \iref{kappau2} that $\kappa > \|u\| \|v\| \geq \<u,v\> = \<u,v_\ve+ \lambda u \> \geq \lambda \|u\|^2$, hence $\lambda \leq \kappa / \|u\|^2 \leq \kappa$.

%We next observe that $\{w ; \, \det (u,w) = 1\} = v+u\Z$. Denoting by $v_0\in \Z^2$ the element of minimal norm such that the triangle of vertices $0,u,v_0$ belongs to $A_u$, we obtain
We have $\|v_\ve\| \geq \|u\|$ by construction, and $\|v_\ve+ \lambda u\| \geq \lambda \|u\|$ for any $\lambda \geq 1$, since $\<u,v_\ve\> \geq 0$. Therefore, recalling \iref{ITUV},
$$
\sum_{T \in A_u} |I_T| \leq \sum_{\ve \in \{1,-1\}} \sum_{0 \leq \lambda\leq \kappa}  \frac \pi {\|u\| \| v_\ve + \lambda u\|} 
\leq  \sum_{0 \leq k \leq \kappa} \frac {2 \pi} {\|u\|^2 \max\{\lambda,1\}} \leq  \frac{2 \pi (2+\ln \kappa)}{\|u\|^2},
$$ 
which concludes the proof of this proposition.
\end{proof}

The following corollary implies the main result of this paper, Theorem \ref{th:TCard}, in the special case of anisotropic euclidean norms.

%We establish in the next Corollary that Theorem \ref{th:TCard} holds for the particular norm $F$ studied in this subsection. Any norm anisotropic euclidean norm is of this form, up to a rotation and a multiplication by a positive scalar, hence Theorem \ref{th:TCard} holds for all such norms.

\begin{corollary}
\label{corol:ThEucl}
There exists a constant $C$ such that for any anisotropic euclidean norm $G$ one has 
$$
\int_0^{2 \pi} \#(\cT(G^\htheta)) \, d\theta\leq C(1+\ln \kappa(G))^2. %4+C' (1+\ln \kappa)^2 \leq C'' (1+\ln \kappa(F))^2,
$$
%independent of $\kappa = \kappa(F)$, such that 
\end{corollary}

\begin{proof}
It is sufficient to prove this result for the specific norm $F$ introduced in Definition \ref{def:FEucl}, since any anisotropic euclidean norm has this form, up to a rotation and a multiplication by a positive scalar.
%Up to a rotation and a multiplication by a scalar, any anisotropic euclidean norm has the form

The sum \iref{sumElemT} of the interval lengths $|I_T|$ associated to all elementary triangles $T$, can be bounded as follows: %using Proposition \ref{prop:SumAu}
$$
\sum_T |I_T| \leq \sum_{u\in \Z^2} \sum_{T\in A_u} |I_T| \leq C(1+\ln \kappa) \sum_{\substack{u\in \Z^2\sm \{0\}\\ \|u\|\leq \sqrt \kappa}} \frac 1{\|u\|^2} \leq 8C (1+\ln \kappa)^2,
$$
where we used \eqref{sumIT} for the second inequality, and \iref{sumLog} for the last one. %, and $C':=8C$.
Combining this estimate with \iref{ineqTE}, we conclude the proof: %for the norm $F$ given by the matrix $M_0$ \iref{def:M0}, one has 
\[
\int_0^{2 \pi} \#(\cT(F^\htheta)) \, d \theta= 4+\int_0^{2 \pi} \#(\cE_\theta) \, d\theta = 4+\sum_T |I_T| \leq 4+8C(1+\ln \kappa)^2.% \leq C'' (1+\ln \kappa(F))^2,
\qedhere
\]
%since $\kappa(F) = \kappa$. 

\end{proof}

%the proof of Theorem \ref{th:TCard} in the case of a norm given by a symmetric positive definite matrix : 

\subsection{Symmetric norms}

In this section and the following one, we denote by $\gF$ the collection of asymmetric norms which are continuously differentiable outside of the origin. To each $F \in \gF$ we attach a $2\pi$-periodic map $\vp_F : \R \to ]-\pi/2,\pi/2[$, introduced in the following definition, which encodes the direction of its gradient.
See Figure \ref{fig:tangent} (left) for an illustration and (center) for two examples.

\begin{definition}
\label{def:phiF}
For any $F\in \gF$ and any $\theta\in \R$, let %we define the angle \iref{def:angle}
$
\vp_F(\theta) := \varangle(e_\theta, \nabla F(e_\theta)).
$
\end{definition}
Since the asymmetric norm $F$ is $1$-homogeneous, we have $\<e_\theta, \nabla F(e_\theta)\> = F(e_\theta)>0$ for all $\theta \in \R$, hence 
\be
\label{PhiFPi2}
\vp_F(\theta) \in ]-\pi/2, \pi/2[.
\ee
The composition of $F$ with a rotation, corresponds to the composition of $\vp_F$ with a translation:
%The rotation of the norm $F$ corresponds to a transl
\be
\label{vpFTrans}
\vp_{F^\theta} = \vp_F(\cdot-\theta).
\ee
Note that $\vp_F$ is $\pi$-periodic if $F$ is symmetric, and odd if $F(x,y) = F(x,-y)$ for all $x,y \in \R$.
In the special case of the euclidean norm, $F_0(u) := \|u\|$, we have $\nabla F_0(u) = u/\|u\|$, hence $\vp_{F_0} = 0$ identically on $\R$.

The next proposition establishes the two most noticeable properties of $\vp_F$ aside from its periodicity: its integrals are bounded \iref{intBounded} in terms of the anisotropy ratio $\kappa(F)$, and it obeys a semi-Lipschitz regularity property \iref{rightLipschitz}. %, which is a key ingredient of our analysis.
%The function $\vp_F$ obeys a semi-Lipschitz regularity property \iref{rightLipschitz}, as shown in the next proposition, which plays a key role in our analysis.
\begin{figure}
\includegraphics[width=5cm]{\pathPic/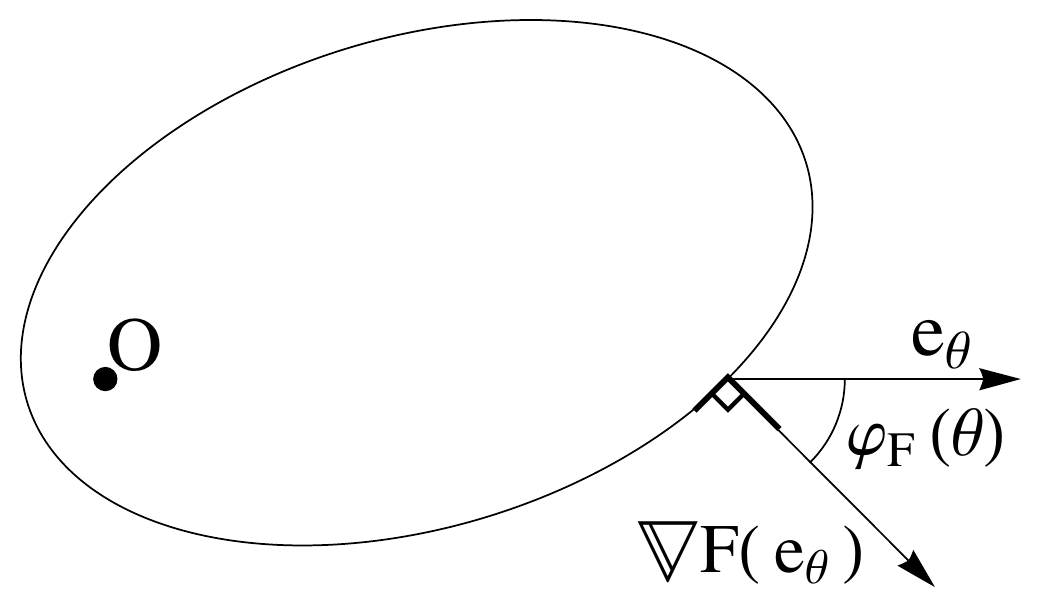}
\hspace{0.1cm}
\includegraphics[width=5cm]{\pathPic/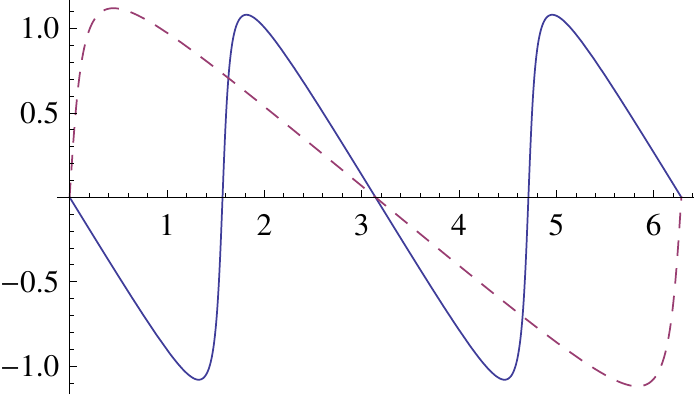}
\hspace{0.1cm}
\includegraphics[width=5.5cm]{\pathPic/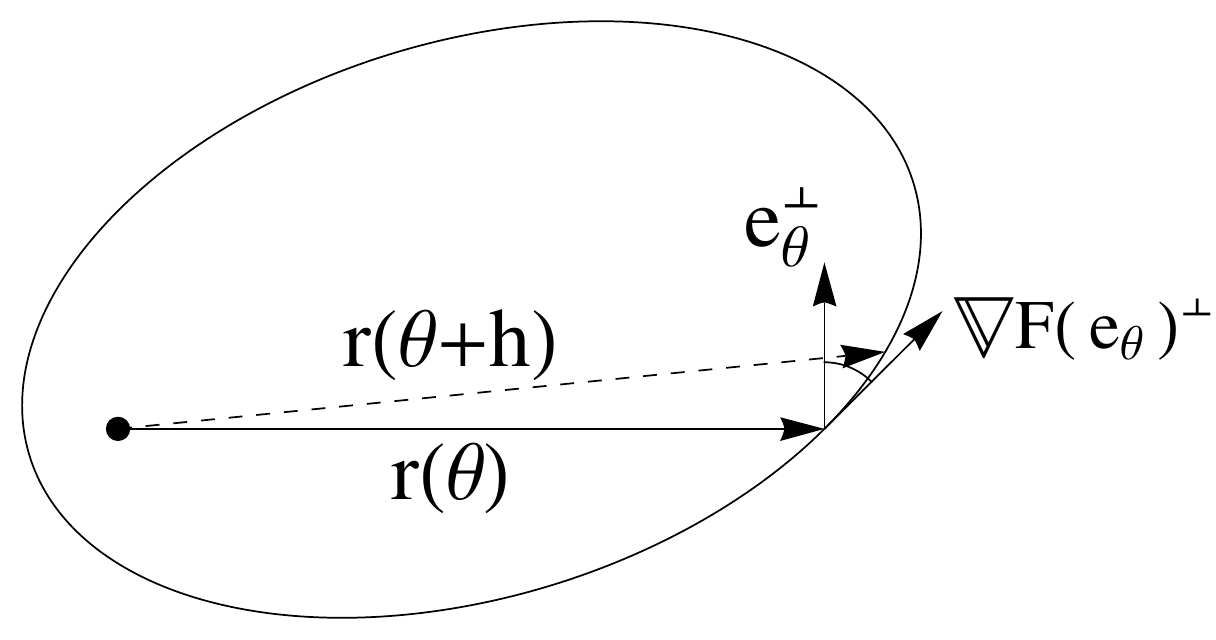}
\caption{
\label{fig:tangent}
Illustration of Definition \ref{def:phiF} (left), $\vp_F(\theta)<0$ in this example. Graph of $\vp_F(\theta)$ (center) for an anisotropic euclidean norm given by a diagonal matrix (plain%, anisotropy ratio $\kappa=4$
), and the asymmetric norm $\sqrt{x^2+y^2}-0.9x$ (dashed). %\kappa = 19
Notations of Proposition \ref{prop:phi}.}
\end{figure}

\begin{prop}
\label{prop:phi}
For any $F \in \gF$ and any $\theta\in \R$, one has 
%\be
%\label{DLnF}
$
\frac {d}{d\theta} \ln F(e_\theta) = \tan \vp_F(\theta).
$
%\ee
As a result for any $h>0$
%hence for any $\Theta,\Phi\in \R$,
\be
\label{intBounded}
\left| \int_\theta^{\theta+h} \tan \vp_F\right| \leq \ln \kappa(F).
\ee
Furthermore $\vp_F$ is right-Lipschitz, and $|\vp_F|$ is bounded strictly away from $\pi/2$: 
%Furthermore for any $h>0$
%\be
%\label{rightLipschitz}
%\vp_F(\theta+h)\geq \vp_F(\theta) - h,
%\ee
%and  $\kappa(F) \cos \vp_F(\theta) \geq 1$.
\begin{eqnarray}
\label{rightLipschitz}
\vp_F(\theta+h) &\geq& \vp_F(\theta) - h,\\
\label{PhiBounded}
\cos \vp_F(\theta) &\geq& 1/\kappa(F).
\end{eqnarray}
\end{prop}

\begin{proof}
Let $r\in C^1(\R,\R_+^*)$ be defined by $r(\theta) := 1/F(e_\theta)$, for all $\theta \in \R$. This quantity is illustrated on Figure \ref{fig:tangent} (right), as well as the vectors $e_\theta^\perp$ and $\nabla F(e_\theta)^\perp$, where $u^\perp$ denotes the rotation by $\pi/2$ of a vector $u\in \R^2$. % \emph{rotated} by $\pi/2$ (as indicated by the $\perp$ superscript).
One easily obtains the Taylor development
$$
r(\theta+h) = r(\theta) + h r(\theta) \tan (-\vp_F(\theta))+ o(h),
$$
for any fixed $\theta \in \R$, and for small $h$. In other words r$'(\theta) = - r(\theta) \tan \vp_F(\theta)$, and equivalently $\frac {d}{d\theta} \ln F(e_\theta) = \tan \vp_F(\theta)$. The left hand side of \iref{intBounded} therefore equals $|\ln F(e_\theta) - \ln F(e_{\theta+h})|$, which as announced is bounded by $\ln \kappa(F)$, by definition of the anisotropy ratio \iref{defKappa}.
%which is equivalent to the announced result \iref{DLnF}.

Let $\psi(\theta) := \theta+\vp_F(\theta)$, for all $\theta \in \R$. By construction, $\nabla F(e_\theta)$ is positively proportional to $e_{\psi(\theta)}$ for all $\theta \in \R$.
%Let $\psi : \R \to \R$ be such that $\nabla F(e_\theta)$ is positively proportional to $e_{\psi(\theta)}$ for all $\theta \in \R$, and $\psi(0) \in ]-\pi/2,\pi/2[$. 
The vectors $e_{\psi(\theta)}$ and $e_{\psi(\theta)}^\perp$ are respectively the unit normal and the unit tangent to the set $B_F := \{z\in \R^2; \, F(z) \leq 1\}$, in the direction $e_\theta$. Since $B_F$ is convex, the derivative of the tangent vector $\frac d {d\theta} e_\psi(\theta)^\perp = -\psi'(\theta)e_{\psi(\theta)}$ is negatively proportional to the normal $e_{\psi(\theta)}$. This shows that $\psi'(\theta) \geq 0$, for all $\theta\in \R$, hence that $\psi$ is non-decreasing.
%The vector $e_{\psi(\theta)}$ is the unit normal to the unit ball of $F$, which is convex, and $e_\psi(\theta)^\perp$ is the unit tangent.
%hence $\psi$ is an increasing function. % (*details ?*)
Recalling that $\vp_F(\theta) = \psi(\theta) - \theta$, we conclude that $\vp_F$ is the difference of a non-decreasing function and a $1$-Lipschitz function, which establishes \iref{rightLipschitz}.

%Let $\theta \in \R$ be such that $\vp := \vp_F(\theta)$. 
For the last inequality, we fix $\theta$ and first assume that $\vp := \vp_F(\theta) \geq 0$. We obtain using \iref{intBounded}
$$
\ln \kappa(F) \geq \int_\theta^{\theta+\vp} \tan \vp_F \geq \int_0^\vp \tan (\vp-u) du = -\ln (\cos \vp),
$$
hence $\cos(\vp) \geq 1/\kappa(F)$, as announced. If $\vp \leq 0$, a similar argument involving the integral on $[\theta + \vp, \theta]$ yields the same estimate, which concludes the proof of this proposition.
%An easy geometrical argument, illustrated on Figure \ref{fig:tangent}, shows the taylor 
\end{proof}

We rephrase in the next lemma a geometrical property, on the gradients of a family of asymmetric norms, into inequalities between the attached functions. % introduced in Definition \ref{def:phiF}.
%Some geometrical properties of the gradient of an asymmetric norm 

\begin{lemma}
\label{lem:ineqCone}
Let $F,F_1, \cdots, F_r \in \gF$, and let $u\in \R^2$. The following are equivalent:
%Let $u\in \R^2 \sm \{0\}$, and let $\theta \in \R$ be such that $u$ and $e_\theta$ are positively collinear. Let also $F, F_1, \cdots, F_r\in \gF$.
%MThe following are equivalent:
%\begin{enumerate}[(a)]
\begin{itemize}
\item
There exists 
$\alpha_1, \cdots, \alpha_r\in \R_+$ such that 
\be
\label{posVec}
\nabla F(u) = \sum_{1 \leq i \leq r} \alpha_i \nabla F_i(u).
\ee
\item 
Let $\theta\in \R$ be such that $u$ and $e_\theta$ are positively collinear. Then
\be
\label{PhiBoundPhi}
\min\{ \vp_{F_1}(\theta), \cdots, \vp_{F_r}(\theta)\} \leq \vp_F(\theta) \leq \max\{ \vp_{F_1}(\theta), \cdots, \vp_{F_r}(\theta)\}
\ee
\end{itemize}
%\end{enumerate}
\end{lemma}

\begin{proof}
Since $F$ is $1$-homogeneous, we have $\nabla F(u) = \nabla F(u/\|u\|) = \nabla F(e_\theta)$ (likewise $\nabla F_i(u) = \nabla F_i(e_\theta)$). 
Let $v := \lambda \nabla F(u)$ (resp.\ $v_i := \lambda_i \nabla F_i(u)$), where the positive scalar $\lambda$ (resp.\ $\lambda_i$) is chosen so that $\<u,v\> = 1$ (resp.\ $\<u,v_i\> = 1$).
We introduce the angles $\vp := \vp_F(\theta)$ (resp. $\vp_i:=\vp_{F_i}(\theta)$), which belong to $]-\pi/2,\pi/2[$, see \iref{PhiFPi2}, and we observe that 
$\tan \vp = \det(u,v)/\<u,v\>=\det(u,v)$ (resp.\ $\tan \vp_i = \det(u,v_i)$).

Proof that \eqref{posVec} $\Rightarrow$ \eqref{PhiBoundPhi}. Assuming \iref{posVec}, and denoting $\beta_i := (\lambda/\lambda_i) \alpha_i \geq 0$, we have $v = \sum_{i=1}^r \beta_i v_i$ and therefore
%$$
%\begin{array}{rcccccl}
%%1 = \< u,v\> = \sum_{1 \leq i \leq r} \beta_i \<u,v_i\> = \sum_{1 \leq i \leq r} \beta_i.
%1 &=& \< u,v\> &=& \sum_{1 \leq i \leq r} \beta_i \<u,v_i\> &=& \sum_{1 \leq i \leq r} \beta_i.
%\end{array}
%$$
\begin{gather*}
1 = \< u,v\> = \sum_{1 \leq i \leq r} \beta_i \<u,v_i\> = \sum_{1 \leq i \leq r} \beta_i,\\
\tan \vp = \det (u,v) = \sum_{1 \leq i \leq r} \beta_i \det(u,v_i) = \sum_{1 \leq i \leq r} \beta_i \tan \vp_i.
\end{gather*}
This shows that $\tan \vp$ is a weighted average of the reals $(\tan \vp_i)_{i=1}^r$, which implies \iref{PhiBoundPhi} since the function $\tan$ is increasing on $]-\pi/2, \pi/2[$.

Proof that \eqref{PhiBoundPhi} $\Rightarrow$ \eqref{posVec}. Conversely, if \iref{PhiBoundPhi} holds, we may assume without loss of generality that $\vp_1 \leq \vp \leq \vp_2$. We thus have $\tan \vp_1 \leq \tan \vp \leq \tan \vp_2$, hence there exists barycentric coefficients $\beta_1,\beta_2\in \R_+$, $\beta_1+\beta_2=1$, such that $\tan \vp = \beta_1\tan \vp_1+\beta_2\tan \vp_2$. Setting $\beta_3=\cdots=\beta_r=0$, we thus have $\tan \vp = \sum_{i=1}^r \beta_i \tan \vp_i$.
%then there exists non-negative weights $(\beta_i)_{i=1}^r$, $\beta_1+\cdots+\beta_r=1$, such that $\tan \vp = \sum_{i=1}^r \beta_i \tan \vp_i$. Setting 
Defining $V = \sum_{i=1}^r \beta_i v_i$, we obtain proceeding as above $\<u,v\>=1=\<u,V\>$ and $\det(u,v) = \tan \vp = \det(u,V)$, hence $v=V$. Denoting $\alpha_i := (\lambda_i/\lambda) \beta_i$ we obtain $\nabla F(u) = \sum_{i=1}^r \alpha_i \nabla F_u(u)$, which establishes \iref{posVec} and concludes the proof.
%On the other hand $\tan \vp = \det(u,v)/\<u,v\>=\det(u,v)$, and likewise $\tan \vp_i = \det(u,v_i)$, hence
%$$
%\tan \vp = \det (u,v) = \sum_{1 \leq i \leq r} \beta_i \det(u,v_i) = \sum_{1 \leq i \leq r} \beta_i \tan \vp_i,
%$$
%which 
\end{proof}

We emphasize the next proposition, which is a central component of our strategy. 
Consider a ``complex'' asymmetric norm $F$, and ``simpler'' norms $(F_i)_{i=1}^r$, say of anisotropic euclidean type.  Assume that $\vp_F$ is bounded in the sense of \iref{minFiF} by the $\vp_{F_i}$. The following result shows that $\#(\cT(F))$ can be estimated in terms of the $\#(\cT(F_i))$, for which efficient bounds were developped in the previous subsection. %we already have efficient bounds if the $F_i$ are of anisotropic euclidean type.

%This proposition shows that cardinality $\#(\cT(F))$ of the mesh associated to a ``complex'' asymmetric norm, is bounded by the sum of $\#(\cT(F_i))$, $1 \leq i \leq r$, where the $F_i$ are simpler norms. of ``simpler'' norms 
% (or the average cardinality of $\#(\cT(F^\htheta))$, $\theta\in 

%The next proposition, which is central in our analysis

\begin{prop}
\label{prop:IneqPhiIneqCard}
Let $F,F_1, \cdots, F_r \in \gF$ be such that everywhere on $\R$
\be
\label{minFiF}
\min\{ \vp_{F_1}, \cdots, \vp_{F_r}\} \leq \vp_F \leq \max\{ \vp_{F_1}, \cdots, \vp_{F_r}\}.
%\min\{ \vp_{F_1}^+(\theta), \cdots, \vp_{F_r}^+(\theta)\} \leq \vp_F^+(\theta) \stext{ and }  \vp_F^-(\theta) \leq \max\{ \vp_{F_1}^-(\theta), \cdots, \vp_{F_r}^-(\theta)\}.
%\min\{ \vp_{F_1}^-(\theta), \cdots, \vp_{F_r}^-(\theta)\} \leq \vp_F^-(\theta) \leq \vp_F^+(\theta) \leq \max\{ \vp_{F_1}^+(\theta), \cdots, \vp_{F_r}^+(\theta)\}.
\ee
Then 
\be
\label{CardTFTFi}
\#(\cT(F)) \leq \sum_{1 \leq i \leq r} \#(\cT(F_i)), \stext{ and }
\int_0^{2 \pi} \#(\cT(F^\htheta))\, d\theta \leq \sum_{1 \leq i \leq r} \int_0^{2\pi} \#(\cT(F_i^\htheta))\, d\theta.
\ee
%\be
%\label{PFiPF}
%p_{F_1} \wedge\cdots\wedge p_{F_r} \Ra p_F,
%\ee
%and therefore 
%\be
%\label{CardTFTFi}
%\#(\cT(F)) \leq \sum_{1 \leq i \leq r} \#(\cT(F_i)).
%\ee
\end{prop}

\begin{proof}
We begin with the proof of an intermediate result: if \iref{minFiF} holds, then we have the implication of ASCs 
\be
\label{PFiPF}
p_{F_1} \wedge\cdots\wedge p_{F_r} \Ra p_F.
\ee
Indeed let $T$ be an elementary triangle, of non-zero vertices $u,v$. Let $\theta_u, \theta_v \in \R$ be such that $e_{\theta_u}, e_{\theta_v}$ are respectively positively proportional to $u,v$.
Assume that $(p_{F_1} \wedge \cdots \wedge p_{F_r})(T)$ holds, which means that $u,v$ form an $F_i$-acute angle for all $1 \leq i \leq r$. Using Lemma \ref{lem:acuteCriterion} we obtain that $\<u,\nabla F_i(v)\>\geq 0$ and $\<v,\nabla F_i(u)\>\geq 0$, for all $1 \leq i \leq r$. Using \iref{minFiF} and Lemma \ref{lem:ineqCone} we find that $\nabla F(u)$ is a linear sum with non-negative coefficients of the vectors $\nabla F_i(u)$, $1 \leq i \leq r$, hence $\<v,\nabla F(u)\>\geq 0$. Likewise $\<u,\nabla F(v)\> \geq 0$. Using again Lemma \ref{lem:acuteCriterion} we obtain that $u,v$ form an $F$-acute angle, hence $p_F(T)$ holds. This concludes the proof of \iref{PFiPF}.

The left part of \iref{CardTFTFi} immediately follows from \iref{PFiPF} and Point 3 of Proposition \ref{prop:pCard}.
Due to the translation invariance \iref{vpFTrans}, inequality \iref{minFiF} is equivalent to $\min\{ \vp_{F_1^\htheta}, \cdots, \vp_{F_r^\htheta}\} \leq \vp_{F^\theta} \leq \max\{ \vp_{F_1^\theta}, \cdots, \vp_{F_r^\theta}\}$ for any $\theta \in \R$, and thus implies $\#(\cT(F^\theta)) \leq \sum_{1 \leq i \leq r} \#(\cT(F_i^\theta))$. Integrating over $[0,2\pi]$ we obtain the right part of \iref{CardTFTFi}, which concludes the proof.
\end{proof}

The next technical lemma describes the periodic function attached to an anisotropic euclidean norm. This is a prerequisite if one wants to construct a well chosen family $(F_i)_{i=1}^r$ of such norms which satisfies \iref{minFiF}, given an asymmetric norm $F$ of interest.

\begin{lemma}
%\label{prop:MinkowskiCompact}
\label{lem:angleMatrix}
\begin{itemize}
\item
The anisotropic euclidean norm $F$ defined by the diagonal matrix of entries $(\kappa^{-1}, \kappa)$, where $\kappa \geq 1$, satisfies $\kappa(F) = \kappa$.
The function $\vp_F$ attains its maximum at $\theta_\kappa := \arctan(\kappa^{-1})$, which is 
$$ 
\vp_F(\theta_\kappa) = \arctan\left[ \left(\kappa - \kappa^{-1}\right)/2 \right].
$$
\item
There exists a finite number anisotropic euclidean norms $F_1, \cdots, F_r$, %defined by positive definite matrices $M_1, \cdots, M_r$, 
such that on $\R$
$$
%\min\{ \vp_1,\cdots,\vp_r\} \leq -\pi/4 \stext{ and } \pi/4 \leq \max\{ \vp_1, \cdots, \vp_r\}
\min\{ \vp_{F_1},\cdots,\vp_{F_r}\} \leq -\pi/4 \stext{ and } \pi/4 \leq \max\{ \vp_{F_1}, \cdots, \vp_{F_r}\}.
$$
%where we denoted $\vp_i := \vp_{F_i}$.
\end{itemize}
\end{lemma}

\begin{proof}
First point.
Let $D$ be the diagonal matrix of entries $(1/\kappa, \kappa)$. We have $F(u)^2 = \<u, D u\>$ for any $u \in \R^2$, hence $F(u) \nabla F(u) = Du$ for any $u\neq 0$, by differentiation.
It follows that $F(e_\theta) \nabla F(e_\theta) =  \left(\kappa^{-1} \cos \theta , \kappa \sin \theta\right)$, for each $\theta \in \R$, and therefore
%\be
%\label{
$$
\tan \vp_F(\theta) = \frac{\det(e_\theta,\nabla F(e_\theta))}{\<e_\theta, \nabla F(e_\theta)\>} = \frac { (\kappa - 1/\kappa) \cos\theta \sin \theta}{\kappa^{-1}\cos^2 \theta + \kappa \sin^2 \theta} = \frac { \kappa-1/\kappa} {(\kappa \tan \theta)^{-1} + \kappa \tan \theta},
$$
%\ee
where the right hand side equals $0$ by convention if $\theta$ is a multiple of $\pi/2$. The maximum value of this right hand side is attained when its denominator is positive and minimal, that is when $\kappa \tan \theta = 1$. %, hence when $\theta = \theta_\kappa$.  
Thus the maximum value of $\tan \vp_F$ is $(\kappa-1/\kappa)/2$, attained at $\theta_\kappa$, as announced. %which concludes the proof of the first point of this proposition.

Second point. Let $\kappa$ be sufficiently large ($\kappa=13$ is fine) in such way that 
\begin{equation*}
\arctan\left[ \left(\kappa - \kappa^{-1}\right)/2 \right] \geq \pi/4+ \pi/5. %This is possible since $\pi/4+\pi/5 < \pi/2 = \lim_{\tau \to \infty} \arctan (\tau)$.
\end{equation*}
Let $F$ be the norm associated to the diagonal matrix of entries $(1/\kappa,\kappa)$.
Since $F$ is symmetric, the function $\vp_F$ is $\pi$-periodic. Since $F$ is defined by a diagonal matrix we have $F(x,y) = F(x,-y)$ for all $x,y\in \R$, and therefore $F$ is odd. Furthermore for all $\theta \in [\theta_\kappa, \theta_\kappa+\pi/5]$ we have $\vp_F(\theta) \geq \vp_F(\theta_\kappa) - \pi/5 \geq \pi/4$, using \iref{rightLipschitz}. Finally for all $n \in \Z$
$$
\vp_F \geq \pi/4 \text{ on }  [n \pi + \theta_\kappa, \, n \pi + \theta_\kappa+\pi/5], \stext{ and } \vp_F \leq -\pi/4 \text{ on }  [n \pi - \theta_\kappa-\pi/5,\, n \pi - \theta_\kappa].
$$
We choose $r:=5$, and introduce the anisotropic euclidean norms $F_k := F^{k\pi/5}$, for $1 \leq k \leq 5$. Using the translation invariance \iref{vpFTrans} we see that the sets on which $\vp_{F_i} \geq \pi/4$ (resp.\ $\vp_{F_i} \leq -\pi/4$), $1 \leq i \leq r$, contain intervals which cover the whole line $\R$. This concludes the proof. %union 
\end{proof}

%(*put somewhere else ? Maybe with $F^\htheta$*)
For any asymmetric norm $F$ and any $A \in \GL_2$, we denote by $F \circ A$ the asymmetric norm defined by 
$$
F\circ A(u) := F(Au).
$$
Clearly $F \circ A$ is symmetric (resp. is an element of $\cF$, resp. is of anisotropic euclidean type) if and only is that is the case for $F$. We denote $\kappa(A) := \|A\| \|A^{-1}\|$, and point out that $\|Au\|/\|Av\|\leq \kappa(A)$ for any $u,v\in \R^2$ such that $\|u\|=\|v\|=1$. It easily follows that %one easily checks that 
\be
\label{ineqKappaFA}
\max\left\{\frac{\kappa(F)}{\kappa(A)}, \frac{\kappa(A)}{\kappa(F)}\right\} \leq \kappa(F\circ A) \leq \kappa(F) \kappa(A).
\ee
Choosing $A=R_\theta^\trans$, for some $\theta \in \R$, we recover in particular that $\kappa(F^\htheta) = \kappa(F)$, for all $\theta\in \R$. 

We establish in the next corollary the average case and the worst case estimates for $\#(\cT(F))$, where $F$ is an arbitrary \emph{symmetric} norm, which were announced in Theorem \ref{th:TCard} and Proposition \ref{prop:WorstCase} respectively. 

\begin{corollary}
\label{corol:CardSym}
\begin{itemize}
\item
For each symmetric norm $G$ there exists $A \in \GL_2$ such that $F := G \circ A$ satisfies $\kappa(F) \leq \sqrt 2$. If $G \in \gF$, we therefore have on $\R$
\be
\label{ineqPhiFA}
-\pi/4 \leq \vp_F \leq \pi/4
\ee
\item
There exists a constant $C$ such that for any symmetric norm $G$ one has 
\be
\label{ThSym}
\#(\cT(G)) \leq C \kappa(G), \stext{ and } \int_0^{2\pi} \#(\cT(G^\htheta)) \, d\theta \leq C (1+\ln \kappa(G))^2.
\ee
\end{itemize}
\end{corollary}
\begin{proof}
%The minimum value of $\kappa(G\circ A)$, for $A\in \GL_2$, is attained and coincides with the multiplicative Banach-Mazur distance $d(G,\|\cdot\|)$ between the norm $G$ and the euclidean norm $\|\cdot\|$. 
First point. A classical theorem by John states that for any $d$-dimensional convex set $B$ which is \emph{symmetric} with respect to the origin, there exists an ellipsoid $E$ centered at the origin and such that $E \subset B \subset E \sqrt d$. 
Applying this result to $B := \{u\in \R^2; \, G(u) \leq 1\}$, we obtain an ellipsoid $E$ that can be written under the form $E := \{A u ;\,  \|u\|\leq 1\}$, for some $A \in \GL_2$. It easily follows by homogeneity that, for all $u \in \R^2$
$$
\|u\|/\sqrt 2 \leq G(A u) = F(u) \leq \|u\|.
$$
Hence $\kappa(F) \leq \sqrt 2$ as announced.
%It follows from John's theorem on the maximal ellipsoid that this distance is bounded by $\sqrt 2$ for $2$-dimensional norms ($\sqrt d$ for $d$-dimensional norms). 
%Hence we may choose $A \in \GL_2$ such that $F:=G \circ A$ satisfies $\kappa(F) \leq \sqrt 2$. 
If $G \in \gF$, then  $\vp_F$ is well defined and we have $\cos (\vp_{F}(\theta)) \leq 1/\kappa(F) \leq 1/\sqrt 2$, for all $\theta \in \R$, using Proposition \ref{prop:phi}. Finally $|\vp_F(\theta)| \leq \arccos(1/\sqrt 2) = \pi/4$. %, which concludes the proof of the first point. 
%Inequalities \iref{ineqPhiFA} follow from \iref{ineqPhi} and from the observation that $\arccos(1/\sqrt 2) = \pi/4$. % and therefore

Second point, keeping the same notations. % We keep the same notations, and turn to the proof of the second point. 
In view of Lemmas \ref{lem:cvNorm} and \ref{lem:limCard}, we can assume that $G\in \gF$.
Let $F_1, \cdots, F_r$ be as in the second point of Lemma \ref{lem:angleMatrix}, and let $G_i := F_i \circ A^{-1}$, for $1 \leq i \leq r$.
%Let the norms $G_i$, $1 \leq i \leq r$, be defined by $_i \circ A := G_i$.
We have by construction, identically on $\R$
$$
\min\{ \vp_{F_1}, \cdots, \vp_{F_r}\} \leq -\pi/4 \leq \vp_F \leq \pi/4 \leq \max\{ \vp_{F_1}, \cdots, \vp_{F_r}\}.
$$ 
Hence for each $u\in \R^2\sm\{0\}$ there exists, using Lemma \ref{lem:ineqCone}, non-negative coefficients $\alpha_1, \cdots, \alpha_r$ such that:
\be
\label{AGAu}
A^\trans \nabla G (A u) = \nabla F(u)
= \sum_{1 \leq i \leq r} \alpha_i \nabla F_i(u)
= A^\trans \left(\sum_{1 \leq i \leq r} \alpha_i  G_i(Au)\right).
\ee
It follows that $\nabla G(v)$ is a linear combination with non-negative coefficients of the $\nabla G_i(v)$, $1 \leq i \leq r$, for any $v \in \R^2 \sm \{0\}$ (choose $u=A^{-1} v$ in \iref{AGAu}). Using again Lemma \ref{lem:ineqCone} we conclude that $\min\{\vp_{G_1}, \cdots , \vp_{G_r}\} \leq \vp_G \leq \max \{\vp_{G_1},\cdots, \vp_{G_r}\}$ on $\R$. 

Hence 
\begin{gather}
\label{SumG}
\#(\cT(G)) \leq \sum_{1 \leq i \leq r} \#(\cT(G_i)) \leq C \sum_{1 \leq i \leq r} \kappa(G_i),\\
\label{intSumG}
\int_0^{2\pi} \#(\cT(G^\htheta)) \, d\theta \leq \sum_{1 \leq i \leq r} \int_0^{2\pi} \#(\cT(G^\htheta_i)) \leq C \sum_{1 \leq i \leq r} (1+\ln \kappa(G_i))^2,
\end{gather}
where we used Proposition \ref{prop:IneqPhiIneqCard} for the first inequality, of both lines, and for the second inequality Lemma \ref{lem:eigenvec} in the first line, and Corollary \ref{corol:ThEucl} in the second line. 

On the other hand we obtain using \iref{ineqKappaFA}, with $\kappa(A):= \|A\| \|A^{-1}\|$,
$$
\kappa(G_i) = \kappa(F_i \circ A) \leq \kappa(F_i) \kappa(A) 
\stext{ and } \kappa(G) \geq \frac{\kappa(A)}{\kappa(G\circ A)} = \frac{\kappa(A)}{\kappa(F)} \geq \frac{\kappa(A)}{\sqrt 2},
$$
thus $\kappa(G_i) \leq \sqrt 2 \,\kappa(F_i) \, \kappa(G)$. Combining this inequality with \iref{SumG} and \iref{intSumG}, we conclude the proof of \iref{ThSym}. 
\end{proof}

\subsection{Asymmetric norms}

We estimate in this section the average cardinality of $\#(\cT(F^\htheta))$, $\theta \in [0,2 \pi]$, for an arbitrary asymmetric norm $F$ on $\R^2$. 
Our strategy is similar to the case of symmetric norms, presented in the previous subsection: we construct a family $F_1, \cdots, F_r$ of anisotropic euclidean norms such that $\min\{ \vp_{F_1}, \cdots, \vp_{F_r}\} \leq \vp_F \leq \max\{ \vp_{F_1}, \cdots, \vp_{F_r}\}$, and we use Proposition \ref{prop:IneqPhiIneqCard}.

%, with the caveat that one cannot find in general a linear change of variables $A \in \GL_2$ such that $\kappa(F\circ A)$ is bounded by an absolute constant. up to the addition of a non-trivial approximation result presented in Proposition \ref{prop:AsymFi}. %The function $\vp_F$ is bounded above an below  
The construction the $(F_i)_{i=1}^r$ is however more subtle than in the previous section, and the integer $r$ grows logarithmically with $\kappa(F)$.
The following technical lemma, illustrated on Figure \ref{fig:broken}, is our first step.

\begin{lemma}
\label{lem:seq}
Let $0< \ve \leq \pi/6$, and let $\Phi\in C^0([0,2 \pi], ]0 ,\pi[)$ be such that 
(i) $\Phi \geq 2 \ve$, 
(ii) $\Phi-\Id$
is non-increasing, %on $[0, 2\pi]$, and for all
and (iii) for all $0 \leq t \leq t' \leq 2 \pi$ one has 
\be
\label{intcotan}
\int_t^{t'} \cotan \Phi(s) \, ds \leq |\ln \ve|.
\ee
Let $(t_n)_{0 \leq n \leq N}$ be the finite sequence of elements of $[0,2\pi]$ recursively defined as follows: $t_0 := 0$, and $t_{n+1}$ is the largest solution in $[0, 2 \pi]$ of 
\be
\label{deftn}
\Phi(t)= \ve +(t- t_n)
\ee
if one exists. Otherwise the sequence ends.

Then $N \leq C |\ln \ve|$, for some absolue constant $C$ (independent of $\ve$ and $\Phi$).
\end{lemma}

\begin{figure}
\includegraphics[width=7cm]{\pathPic/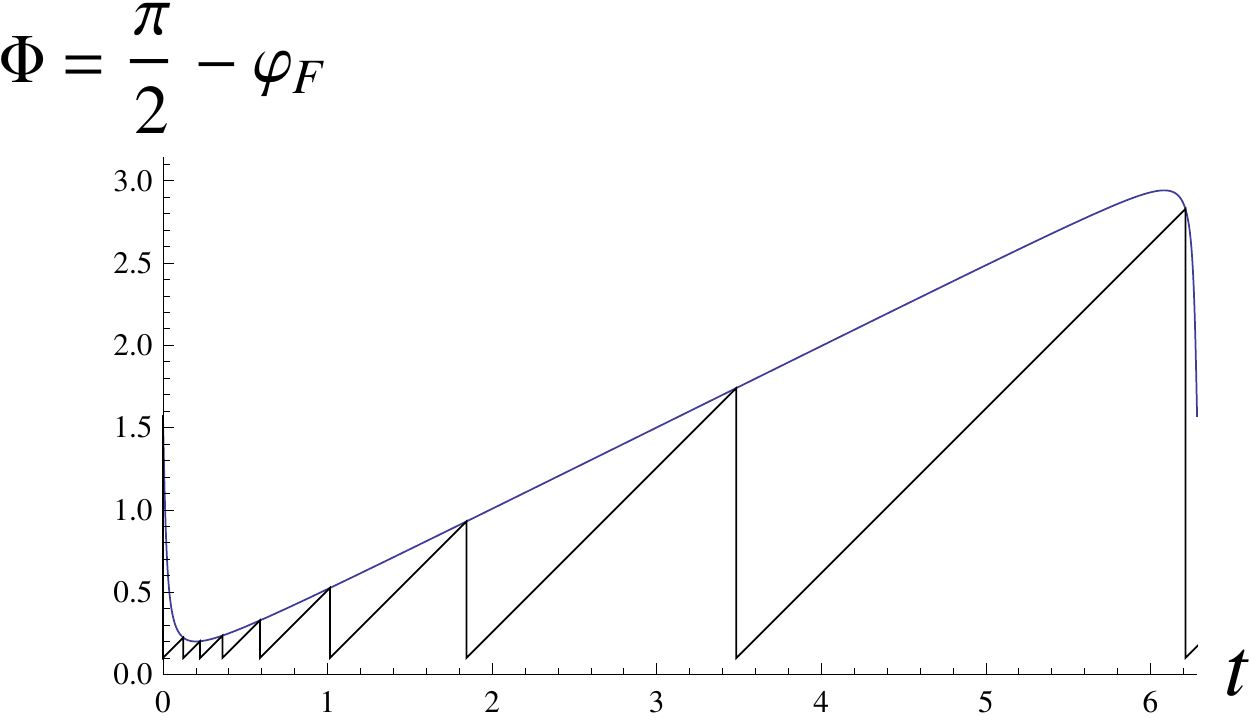}
\hspace{1cm}
\includegraphics[width=7cm]{\pathPic/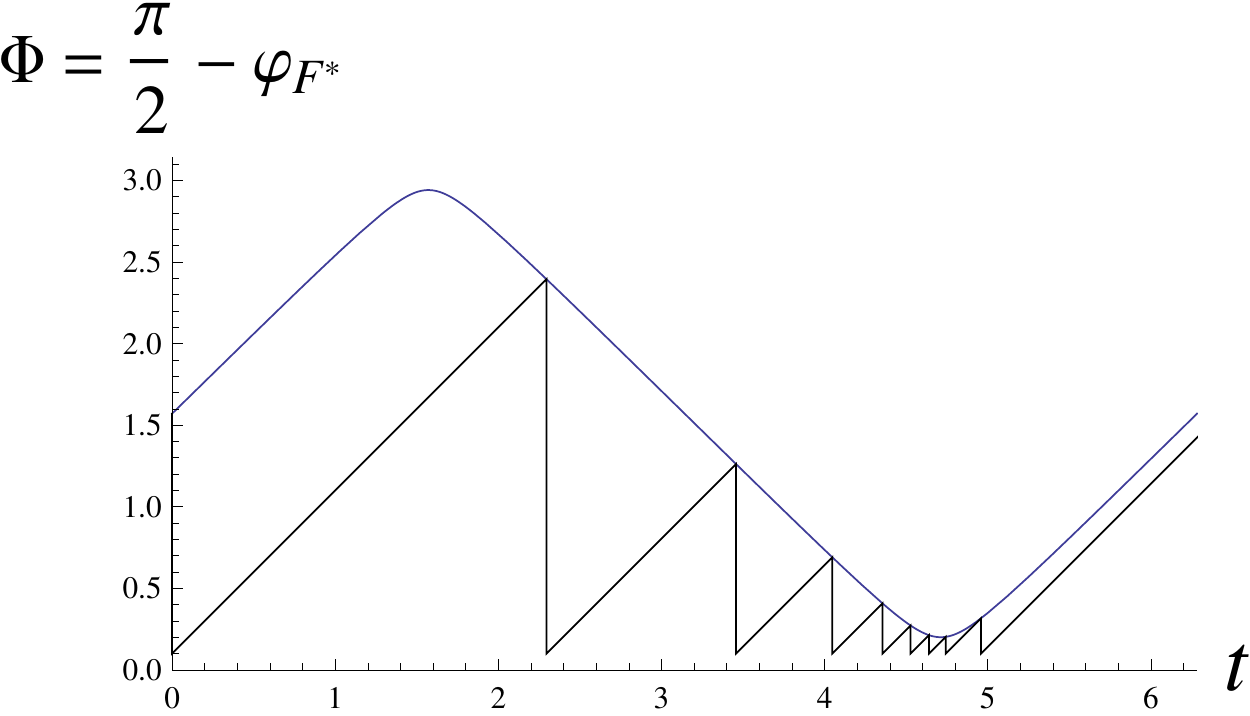}
\caption{
\label{fig:broken}
Illustration of Lemma \ref{lem:seq}. 
Curve : $\Phi$.
Broken line : construction of $t_0, \cdots, t_N$. The intersections of the broken line with the curve correspond to the equations \iref{deftn}. There is one vertical segment at each abscissa $t_n$, $0 \leq n \leq N$.
The chosen $\Phi$ is $\pi/2-\vp_F$, or $\pi/2-\vp_{F^*}$, consistently with Proposition \ref{prop:AsymFi}, with $F(x,y) = \sqrt{x^2+y^2}-0.98\,  x$. 
}
\end{figure}

\begin{proof}
Since $t \mapsto \Phi(t) - t$ is continuous and non-increasing, the solutions of \iref{deftn} are either the empty set, a singleton, or a closed interval.
We denote $\phi_n := \Phi(t_n)$ for all $0 \leq n \leq N$, and observe that $\phi_n \geq 2 \ve$ by hypothesis (ii). We also define for $0\leq n < N$
$$
\delta_n := t_{n+1}-t_n = \phi_{n+1}-\ve \geq 2 \ve -\ve = \ve.
$$
In particular $t_n \geq n\ve$ for $0\leq n \leq N$, and therefore $N \leq 2 \pi / \ve$: the sequence is finite as announced. 
We establish below the finer estimate $N \leq C |\ln \ve|$. 
%Its heuristic description is as follows: we distinguish for each segment $[t_n,t_{n+1}]$ three possible cases. The segment is either (i) long, if $t_{n+1}-t_n \geq \pi/6$, or (ii) such that $\phi_{n+1} \ll \phi_n$, or (iii) such that $\int_{t_n}^{t_{n+1}} \cotan \Phi$ is significant. 
%No more that $(2 \pi) / (\pi/6) = 12$ segments can be of type (i). We then show a uniform bound on the number of segments of type (ii) or (iii), between two segments of type (i).

Consider the two collections of integers
$$
E :=\{2, \cdots, N-1\}, \qquad E_+ := \{n \in E; \, \delta_{n-1} + \delta_n \leq \pi/6\}.
$$
Note that 
%\be
%\label{longSegments}
$$
(\pi/6) \#(E \sm E_+) \leq \sum_{n \in E \sm E_+} \left(\delta_{n-1}+\delta_n\right) \leq 2 \sum_{0\leq n <N} \delta_n = 2 (t_N-t_0) \leq 4 \pi,
$$
%\ee
and therefore $\#(E\sm E_+) \leq c_0 := 24$.
We next estimate the cardinality of $E_+$. Consider an arbitrary $n \in E_+$,
we obtain recalling that $\Phi-\Id$ is non-increasing 
$$
\int_{t_n}^{t_{n+1}} \cotan \Phi(t) \, dt \geq \int_0^{\delta_n} \cotan (\phi_n+t) \, dt = \ln \left(\frac{\sin (\phi_n+\delta_n)}{\sin \phi_n}\right).
$$
We next define and estimate a quantity $e_n$, attached to each $n \in E_+$
\be
\label{defEn}
e_n := \ln \left(\frac{\sin (\phi_n)}{\sin(\phi_{n+1})}\right) + 2\int_{t_n}^{t_{n+1}} \cotan \Phi \geq \ln \left(\frac{\sin(\phi_n+\delta_n)^2}{\sin(\phi_n) \sin(\phi_{n+1})}\right).
\ee %ve+\delta_n
We have $\sin (\phi_n) \leq \phi_n$, and $\sin (\phi_{n+1}) \leq \phi_{n+1} = \ve+\delta_n \leq 2 \delta_n$, since $\delta_n \geq \ve$.
On the other hand $\sin(x) \geq \rho x$ for all $x \in [0,\pi/3]$, by concavity of the sine function on this interval, with $\rho := \sin(\pi/3)/(\pi/3)$.
Observing that $\phi_n+\delta_n= \ve+\delta_{n-1}+\delta_n \leq \pi/6+\pi/6=\pi/3$, we thus obtain  $\sin(\phi_n+\delta_n) \geq \rho (\phi_n+\delta_n)$.
%On the other hand $\phi_n+\delta_n= \ve+\delta_{n-1}+\delta_n \leq \pi/6+\pi/6=\pi/3$, and therefore  $\sin(\phi_n+\delta_n) \geq \rho (\phi_n+\delta_n)$, by concavity of the sine function on $[0,\pi/3]$, with $\rho := \sin(\pi/3)/(\pi/3)$.
Injecting these inequalities in the right hand side of \iref{defEn} we obtain: %conclude that 
$$
\exp(e_n) \geq   \frac{ \rho^2(\phi_n+\delta_n)^2}{\phi_n(2 \delta_n)} = \frac{\rho^2} 2 \left(\frac{\phi_n}{\delta_n}+\frac{\delta_n}{\phi_n}\right)^2 \geq 2 \rho^2 =1.36\ldots >1,
$$
where for the second inequality we used that $x+x^{-1} \geq 2$ for all $x>0$.
%We have $\phi_{n+1} = 
%We next obtain, with $\rho := \sin(\pi/3)/(\pi/3) \simeq 0.827 > 1/\sqrt 2$
%$$
%\exp(e_n) \geq  \frac{\sin(\phi_n+\delta_n)^2}{\sin(\phi_n) \sin(2 \delta_n)} \geq \frac{ \rho^2(\phi_n+\delta_n)^2}{\phi_n(2 \delta_n)} = \frac{\rho^2} 2 \left(\frac{\phi_n}{\delta_n}+\frac{\delta_n}{\phi_n}\right)^2 \geq 2 \rho^2 >1,
%$$
%where we used for the first inequality we used $\phi_{n+1} = \ve + \delta_n \leq 2 \delta_n$, for the second one we used that $x \geq \sin(x) \geq \rho x$ for $x \in [0, \pi/3]$ (by concavity), and for the third one we observed that $x+x^{-1} \geq 2$ for all $x>0$.

Let $n, n+1, \cdots, n+k-1$ be consecutive integers in $E_+$. Then 
\begin{eqnarray*}
k \ln (2 \rho^2) &\leq& \sum_{i=n}^{n+k-1} e_i  \\%\sum_{n \leq i \leq n+k-1} e_{i}\\ 
&=& \ln \left(\frac{\sin \phi_n}{\sin \phi_{n+k}}\right) + 2 \int_{t_n}^{t_{n+k}} \cotan \Phi\\
&\leq& \ln \left(\frac{\sin(\pi/2)}{\sin (2\ve)}\right)+2 |\ln \ve| \\ %\left( \frac{F(e_{\psi_{n+k}})}{F(e_{\psi_n})}\right)\\ % 
&\leq& 3 |\ln \ve | %\kappa(F),
\end{eqnarray*}
where we used in the third line the hypotheses (i) and (iii). % : $\Phi \geq 2 \ve$ (thus $\sin$ , and the upper bound \iref{intcotan}.
Regarding the last line, we have by concavity $\sin (2 \ve) \geq (2/\pi) \ve \geq \ve$, thus $-\ln \sin (2 \ve) \leq |\ln \ve|$.
The maximal number $k$ of consecutive integers in $E_+$ is therefore bounded by $c_1 |\ln \ve|$, where $c_1 := 3/\ln (2 \rho^2)$.

The set $E_+$ can be arranged into at most $\#(E\sm E_+)+1$ series of consecutive elements, and  $\#(E \sm E_+) \leq c_0 := 24$, see above. By the previous argument, these series have length at most $c_1 |\ln \ve|$. 
Finally
$$
N = 2+\#(E\sm E_+) + \#(E_+) \leq 2+c_0+(c_0+1) c_1 |\ln \ve|, % \kappa(F).
$$
which concludes the proof of this lemma.
%
%
%(*old*)
%
%Denoting $\theta_i := 1/\kappa(F)^2 + \psi_i$, for $1 \leq i \leq N$, we obtain 
%that on $\R$ (*details*)
%$$
%\vp_F \leq \max\{\vp_1, \cdots, \vp_r\}.
%$$
%Using a similar argument, we can construct a second collection of angles of cardinality $\leq C(1+\ln \kappa(F))$ such that 
%$
%\vp_F \geq \min\{\vp_1, \cdots, \vp_r\}.
%$
%Joining these two collections we conclude the proof.
\end{proof}

\begin{proposition}
\label{prop:AsymFi}
There exists a constant $C$ such that the following holds.
Let $F\in \gF$, and let $G$ be the norm defined by the diagonal matrix of entries $(\kappa^{-1},\kappa)$, where $\kappa := 4 \kappa(F)+1$. %described in Lemma \ref{lem:G} with the constant $\kappa = \kappa(F)^2$.

Then there exists $r \leq C(1+\ln \kappa(F))$ and $\theta_1, \cdots, \theta_r \in \R$ such that denoting $\vp_i :=  \vp_G(\cdot-\theta_i)$ one has 
\be %\vp_{G^{\htheta_i}} =
\label{ineq:asym}
\min \{\vp_1, \cdots, \vp_r\} \leq \vp_F \leq \max\{\vp_1, \cdots, \vp_r\}.
\ee
\end{proposition}

\begin{proof}
We define $\Phi \in C^0([0,2\pi], [0,\pi])$ by $\Phi(t) := \pi/2-\vp_F(t)$. The difference $\Phi-\Id$ is non-increasing since $\vp_F+\Id$ is non-decreasing, see  Proposition \ref{prop:phi}.
For all $t \in [0,2\pi]$ we have $|\vp_F(t)| \leq \arccos(1/\kappa(F))$, using again Proposition \ref{prop:phi}, hence $\Phi(t) \geq \arcsin(1/\kappa(F)) \geq 1/\kappa(F)$. Last 
for $0 \leq t \leq t' \leq 2 \pi$, one obtains using \iref{intBounded} %and the definition of the anisotropy ratio $\kappa(F)$,
$$
\int_t^{t'} \cotan \Phi = \int_t^{t'} \tan \vp_F %= \ln \left(\frac{F(e_{t'})}{F(e_t)}\right) 
\leq \ln \kappa(F). %\ln (F(e_{t'}))-\ln (F(e_t))
$$

Therefore $\Phi$ satisfies the assumptions of Lemma \ref{lem:seq}, with $\ve  := \min\{\pi/6,\, 1/(2 \kappa(F))\}$.
Let $t_0, \cdots, t_N$ be the finite sequence given by this Lemma. We have $N \leq C |\ln \ve| \leq C' (1+\ln \kappa(F))$ for some absolute constant $C'$.
Using \iref{deftn} and the fact that $\Phi-\Id$ is non-decreasing, we obtain $\Phi(t) \geq \ve + (t-t_n)$ for all $0 \leq n \leq N$, and all $t_n \leq t \leq t_{n+1}$, and therefore 
$$
\vp_F(t) \leq \pi/2 -\ve - (t-t_n).
$$
We next observe, using Lemma \ref{lem:angleMatrix} and with $\Theta := \arctan(1/\kappa)$, that %Our next observation is that, with $\theta_\kappa$ defined in Lemma \ref{lem:angleMatrix}
$$
\vp_G(\Theta) = \arctan\left(\frac{\kappa-\kappa^{-1}} 2\right) \geq \frac \pi 2 - \frac 2 {\kappa-\kappa^{-1}} \geq \frac \pi 2 - \ve,
%\vp_G(\theta_\kappa) = \arctan[(\kappa-\kappa^{-1})/2] \geq \pi/2 - 2/(\kappa-\kappa^{-1})
$$
where we used successively that $\arctan(x) \geq \pi/2-1/x$ for all $x>0$, and that $\kappa-\kappa^{-1} \geq (4 \kappa(F)+1) - 1 \geq 4 \kappa(F) \geq 2/\ve$. Therefore $\vp_G(\Theta+t) \geq \pi/2-\ve-t$, for all $t \geq 0$, using \iref{rightLipschitz}.
Denoting $\theta_i := \Theta- t_i$ for all $0 \leq i \leq N$, and $\vp_i := \vp_G(\cdot-\theta_i)$, we conclude that 
$$
\vp_F \leq \max\{\vp_0, \cdots, \vp_N\},
$$
on $[0,2\pi]$, hence also on $\R$ by $2\pi$-periodicity.
We have obtained one side of the announced inequality \iref{ineq:asym}.

For the other side we define $\hat \Phi(t) := \pi/2+\vp_F(-t)$, obtain $\hat t_0, \cdots, \hat t_{\hat N}$ using Lemma \ref{lem:seq}, with again $\hat N \leq C' (1+ \ln \kappa(F))$. Setting $\hat \theta_i := -\Theta + \hat t_i$ and $\hat \vp_i := \vp_G(\cdot-\hat \theta_i)$, for $0 \leq i \leq \hat N$, we obtain likewise $\vp_F \geq \min \{\hat \vp_1, \cdots, \hat \vp_{\hat N}\}$. % ($\vp_G$ is odd since $G(x,y) = G(x,-y)$ for all $x,y\in \R$). 
This concludes the proof with $r=N+\hat N$.
\end{proof}

We conclude in the next corollary the proof of main result of this paper: the average estimate of $\#(\cT(F^\htheta))$, $\theta \in [0,2 \pi]$, for an arbitrary asymmetric norm $F$. %The author conjectures that the exponent $3$ in \iref{ln3} is overestimated.

\begin{corollary}
\label{corol:CardAsym}
There exists a constant $C$ such that for any asymmetric norm $F$ on $\R^2$, one has
\be
\label{ln3}
\int_0^{2 \pi} \#(\cT(F^\htheta)) \, d\theta \leq C (1+\ln \kappa(F))^3.
\ee
\end{corollary}

\begin{proof}
In view of Lemmas \ref{lem:cvNorm} and \ref{lem:limCard}, we can assume that $F\in \gF$.
%We may assume that $F \in \gF$,
Let $\kappa$, $G$, $r$ and $\theta_1, \cdots, \theta_r$ be as in Proposition \ref{prop:AsymFi}.
Applying successively Proposition \ref{prop:IneqPhiIneqCard} and Corollary \ref{corol:CardSym} we obtain 
\begin{eqnarray*}
\int_0^{2\pi} \#(\cT(F^\htheta)) \, d \theta &\leq& \sum_{1 \leq i \leq r} \int_0^{2 \pi} \#(\cT(G^{\htheta_i+\theta})) \, d\theta\\
&\leq &  \sum_{1 \leq i \leq r} C (1+\ln \kappa(G))^2\\
&=&  r \, C (1+2\ln \kappa)^2.
\end{eqnarray*}
Recalling that $r \leq C' (1+ \ln \kappa(F))$ and $\kappa = 4\kappa(F)+1$, see Proposition \ref{prop:AsymFi}, we obtain the announced result.
%Thus recalling that $r \leq C' (1+\ln \kappa(F))$ we obtain
%$$
%\int_0^{2\pi} \#(\cT(F^\htheta)) \, d \theta
%\leq  C C' (1+\ln \kappa(F)) (1+2 \ln \kappa(F))^2
%%= C'' (1+2 \ln \kappa(G))^3
%\leq C'' (1+\ln \kappa(F))^3,
%$$
%which concludes the proof.
\end{proof}

\section{Implementation and Numerical results}
\label{sec:num}

We compare in this section the algorithm introduced in this paper, FM-ASR, with two alternative solvers of the Escape Time problem, or Anisotropic Eikonal Equation, which enjoy a reputation of efficiency an simplicity in applications \cite{BC10}: the Adaptive Gauss Siedel Iteration (AGSI\footnote{With stopping criterion tolerance parameter $10^{-8}$, as  suggested in \cite{BR06}.%
}) of Bornemann and Rasch \cite{BR06}, and Fast Marching using the $8$ point stencil (FM-$8$). Stencils are illustrated on Figure \ref{fig:classical}.
Two more recent methods were also implemented: the Monotone Acceptance Ordered Upwind Method (MAOUM) of Alton and Mitchell \cite{AltonMitchell12} is tested when its memory usage allows it, and Fast Marching using Lattice Basis Reduction (FM-LBR) of the author \cite{M12} in the special case of Riemannian metrics.

Four test cases are considered: two involving (asymmetric) Finsler metrics, and two involving Riemannian metrics. 
Three of these tests are directly motivated by applications, including motion planning, seismic imaging, and image segmentation, while the fourth one has the advantage of having an analytic solution, avoiding the recourse to a reference solution. Depending on the test, the metric anisotropy $\kappa(\cF)$ ranges from $4$ to $400$. Each algorithm was executed on each test case, at $100$ different resolutions $n\times n$, where $n$ ranged from $61$ to $1201$ (odd values of $n$ are preferred for symmetry reasons). %, which covers typical uses in e.g.\ image processing. 
We compare the algorithm's efficiency by representing, on Figures \ref{fig:offsetted}, \ref{fig:spiraling}, \ref{fig:Vlad} and \ref{fig:Cohen}, their $L^\infty$ or averaged $L^1$ numerical error, with respect to an exact or reference solution\footnote{%
Reference solutions were computed on a $5001\times 5001$ grid, using the AGSI in the first and third test cases, and the FM-LBR in the last (the second test case has an analytic solution). They were extended to the continuous domain via bilinear interpolation.
}, as a function of CPU time. We also discuss accuracy \emph{resolution-wise} in the text: which resolution $n$ is required to meet a prescribed $L^\infty$ error bound ?

In practical applications \cite{BC10}, the choice of the FM-$8$ versus the AGSI is typically regarded as a compromise in favor, respectively, of CPU time or of numerical accuracy. Indeed the FM-$8$ is a single pass solver with a small stencil, which thus completes in short and predictable CPU time, almost independent of the problem solved. 
For instance the FM-$8$ completes our four benchmarks in CPU time\footnote{%
All timings in seconds, obtained on a $2.4$Ghz Core 2 Duo, using a single core. System memory: 8 GB.} 
$0.77s$, $0.79s$, $0.86s$ and $0.79s$ respectively,
on a $601 \times 601$ grid, while the AGSI takes\footnote{%
In the trivial case of a constant metric, equal to the euclidean norm, the AGSI takes $0.67s$, on the same grid.%, which is less than the FM-8, $0.72s$.
}
%$0.92s$, $8.05s$ and $1026s$, 
 $8.61s$, $285s$, $15.1s$ and $123s$. 
On the other hand the results produced by the AGSI are known to converge towards the viscosity solution of the continuous problem, as one refines the discretization grid, whereas convergence can only be guaranteed in limited cases for the FM-$8$. Indeed the acuteness condition (iii) in Definition \ref{def:acute} can be guaranteed for a Finsler metric $\cF$ such that $\kappa(\cF) \leq \sqrt 2$ (using \iref{kappaFAngle}), a Riemannian metric such that $\kappa(\cF) \leq \sqrt 2+1$ (using Proposition 1.2 in \cite{M12}), or axis aligned anisotropy \cite{AM08}. If this condition is violated, then there is no convergence guarantee.
Our objective is to bring together the best of both worlds: our algorithm is fast\footnote{CPU time for the FM-ASR, and the FM-LBR, includes stencil construction, which often accounts for 50\%.
%Stencil construction is counted in CPU time, and often account for 50\%, for the FM-ASR and FM-LBR.
}
($1.39s$, $3.09s$, $1.32s$ and $1.11s$
with the above settings) and universally convergent. 

The MAOUM \cite{AltonMitchell12} and the FM-LBR \cite{M12} are more recent algorithms than the AGSI or the FM-8, and are closer to FM-ASR from a theoretical point of view: they are universally convergent, inspired by Dijkstra's algorithm, and use static stencils assembled in a pre-processing step. However the MAOUM uses large isotropic stencils, see Figure \ref{fig:classical}, instead of smaller anisotropic stencils for the FM-ASR, resulting in a larger complexity and memory footprint.
The FM-LBR mainly distinguishes itself from the FM-ASR by its domain of application: it is restricted to Riemannian metrics, but extends to dimension $3$.
\\

\begin{figure}
%\begin{center}
\includegraphics[width=4cm]{\pathPic/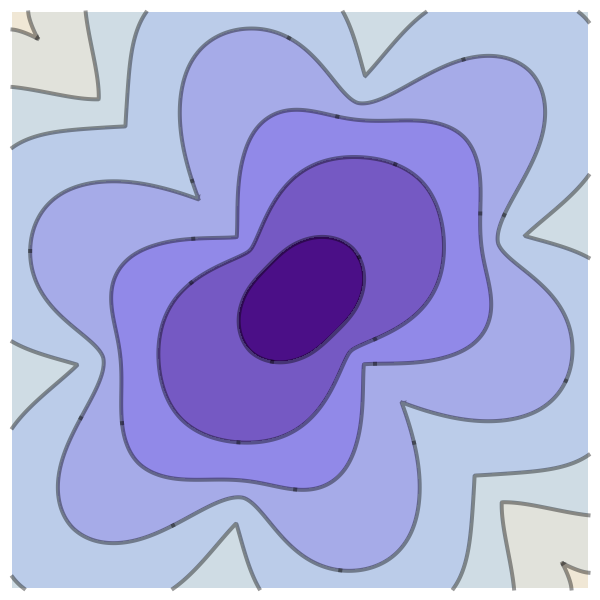}
\includegraphics[width=5.8cm]{\pathPic/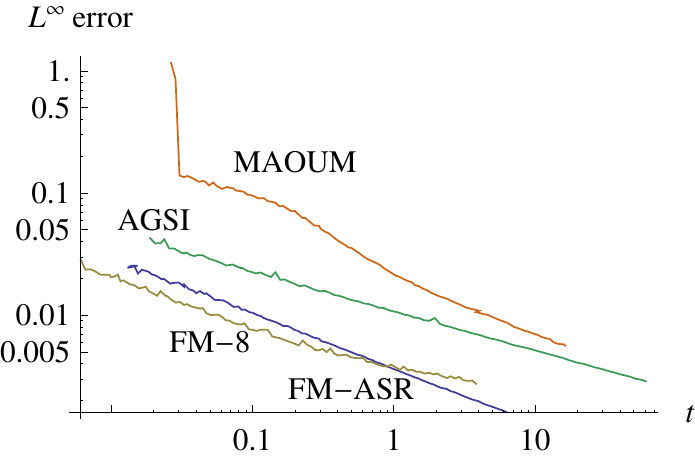}
\includegraphics[width=5.8cm]{\pathPic/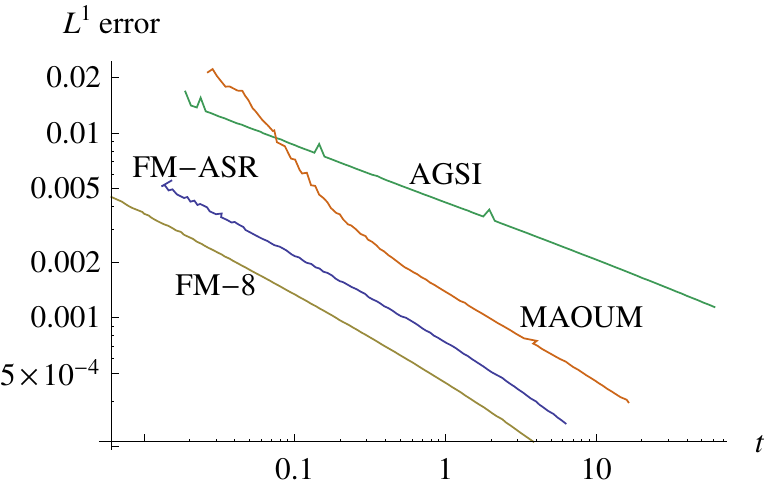}
%\end{center}
\caption{
\label{fig:offsetted}
Level lines of the first test case \cite{SV03}. $L^\infty$ error (center) and averaged $L^1$ error (right), with respect to a reference solution, plotted as a function of CPU time. Data obtained by running the algorithms at $100$ different resolutions $n \times n$, with $n$ ranging from $61$ to $1201$. Best viewed in color. FM-ASR: blue. FM-8: brown. AGSI: green. MAOUM: orange. See Figures \ref{fig:spiraling}, \ref{fig:Vlad} and \ref{fig:Cohen} for the other tests. 
}
\end{figure}

Our first two tests involve asymmetric norms defined as the sum of an anisotropic euclidean norm and of a linear form. %The unit ball of such an asymmetric norm is an ellipse containing the origin but generally not centered on it, hence they are referred to as ``Offsetted anisotropic euclidean norm''. 

\begin{prop} %[Offsetted anisotropic euclidean norms]
\label{prop:OffsettedNorms}
Let $M\in S_2^+$ and let $\omega\in \R^2$ be such that $\<\omega,M \omega\> <1$. The map $F : \R^2 \to \R$ defined by 
$$
F(u) := \sqrt{\<u,M u\>} - \<\omega,M u\>
$$
is an asymmetric norm, which unit ball $\{z; \, F(z) \leq 1\}$ is an ellipse, not centered at the origin if $\omega \neq 0$. Furthermore the dual asymmetric norm $F^*$ has the same form, with parameters $M_*, \omega_*$ defined by 
$$
\delta:=1-\<\omega,M \omega\>, \quad M_* := \frac{\omega \omega^\trans + \delta M^{-1}}{\delta^2}, \quad \omega_* := - M_*^{-1} \omega/\delta.
$$
The minimization problem \iref{mint}, appearing in the Hopf-Lax update operator, and the evaluation of the predicate ``$u,v$ form an $F$-acute angle'', cost numerically $\cO(1)$ operations $+,\, -,\, \times,\, /, \,\sqrt{\cdot}$ among reals.
\end{prop}
\begin{proof}
Up to a linear change of variables (by $M^\frac 1 2$), we may assume that $M=\Id$, and thus $F(u) = \|u\|-\<\omega,u\>$, with $\|\omega\| <1$. 
The $1$-Homogeneity and the Convexity of $F$ are obvious. %, and together they imply the triangle inequality. 
Furthermore $F(u) \geq \|u\| (1-\|\omega\|) \geq 0$, with equality if and only if $u=0$, which shows that $F$ is proper, hence is an asymmetric norm. %satisfies all the axioms of asymmetric norms. 
The boundary of the compact and convex set $\{z\in \R^2;\, F(z) \leq 1\}$ is characterized by the inhomogeneous quadratic equation $\|u\|^2 = (1+\<\omega,u\>)^2$, which first degree term $u \mapsto 2 \<\omega,u\>$ is non-zero if $\omega\neq 0$. Hence this set is an ellipse, non-centered if $\omega\neq 0$.

Consider an arbitrary $u\in \R^2 \sm \{0\}$. Observing that $1/F^*(u) = \min \{F(v); \, \<u,v\> = 1\}$, we find using the Khun-Tucker conditions for this constrained optimization problem, that the minimizer $v$ satisfies $v/\|v\| - \omega = \lambda u$ for some $\lambda \in \R$. Taking the scalar product of this equation with $v$ we obtain $\lambda=1/F^*(u)$. On the other hand observing that $\|\omega+\lambda u\| = 1$, we obtain a quadratic equation which positive root is $\lambda$. The announced expression of $F^*(u)$ follows.

Since the norm $F$ is differentiable, evaluating the predicate ``$u,v$ form an acute angle'' is straightforward.
It was observed in the very first works on fast marching methods \cite{T95} that in the special case $\omega=0$, the minimization problem \iref{mint} amounts to solving a quadratic equation. Choosing a non-zero $\omega$ is equivalent to subtracting $\<\omega,u\>$ to $d_u$ (resp. $\<\omega,v\>$ to $d_v$), thus the problem \iref{mint} has an equally simple solution. %the values of 
\end{proof}

%Our second test was taken from \cite{BR06}, where it was used to demonstrate the superiority of the AGSI over the OUM \cite{SV03}. 
Our first test is a motion planning control problem, also discussed in \cite{SV03,BR06}. 
The Finsler metric is given by its dual: $\cF^*_z(u) = \|u\|+\<\omega(z),u\>$, where $\omega(x,y) = -\gamma \sin(4 \pi x) \sin (4 \pi y)$ and $\gamma = 0.9$. The speed profile $\{u\in \R^2; \, \cF_z(u) \leq 1\}$, in the control theoretic interpretation, is the euclidean unit ball translated by $\omega(z)$, see Figure \ref{fig:MetricTypes} (right): this could model a boat, able to move at unit speed on still water, but caught in an ocean current of speed $\omega(z)$. 
We compute the distance, i.e.\ the minimal travel time, to the center of the square domain $[-0.5,0.5]^2$. See e.g.\ \cite{M12} for a discussion on shortest path extraction. The maximum anisotropy ratio, $\kappa(\cF) = (1+\gamma)/(1-\gamma)=19$, is not small, but anisotropy is pronounced only on a small region, where $|\sin(4 \pi x) \sin (4 \pi y)|$ is close to $1$. As a result, the FM-8 delivers excellent results in terms of averaged $L^1$ error, and best results in terms of $L^\infty$ error for CPU times $\leq 1s$, after what (presumed) non-convergence begins to show and the FM-ASR outperforms it.
The FM-ASR is the best method among those which benefit from a convergence guarantee: for the prescribed tolerance $5 \times 10^{-3}$ on the $L^\infty$ error, the AGSI takes $11.3s$, at resolution $n=661$, while the FM-ASR takes $0.51s$, at resolution $n=375$, thus reducing CPU time by a factor 22. 

\begin{figure}
\includegraphics[width=4cm]{\pathPic/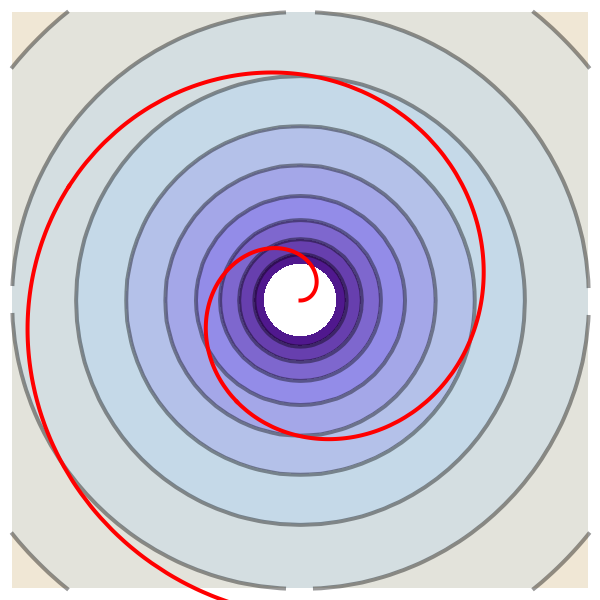}
\includegraphics[width=5.8cm]{\pathPic/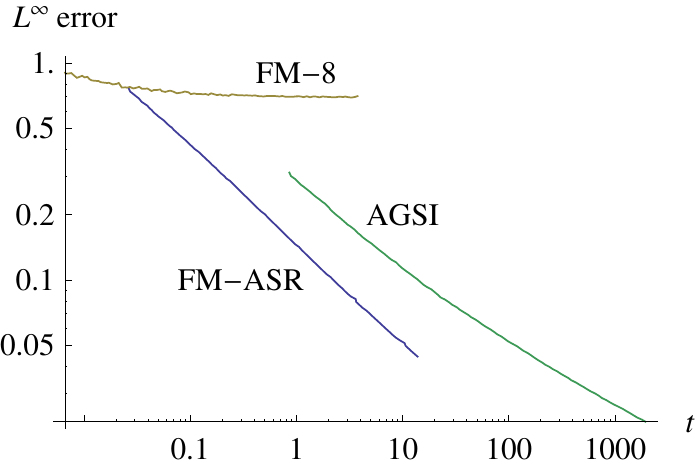}
\includegraphics[width=5.8cm]{\pathPic/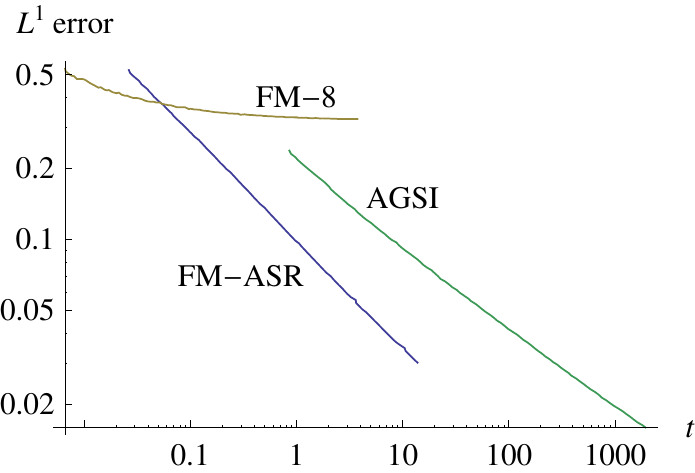}
\caption{
\label{fig:spiraling}
Level lines are circles in the the second test case. The shortest path joining $r e_\theta$ to the origin, where $e_\theta := (\cos \theta, \sin \theta)$, is the spiral $\vp \mapsto (r-\vp) e_{\theta-\vp}$, $\vp\in [0,r]$ (red curve).
}
\end{figure}

Our next test case is based on the following proposition.
\begin{prop}
\label{prop:exact}
Let $g \in C^0(\R_+,]-1,1[)$, and let $\cF$ be the Finsler metric on $\R^2\sm \{0\}$ defined by 
%Let $g\in C^0(\R_+, ]-1,1[)$ be such that $g(0) = 0$. Let $\cF$ be the Finsler metric defined by $\cF_0(u)=\|u\|$, $u\in \R^2$, and for all $z\in \R^2 \sm \{0\}$
\be
\label{defFg}
\cF_z(u) = \|u\| - \frac{g(\|z\|)}{\|z\|} \< z^\perp , u\>,
\ee
where $z^\perp$ denotes the rotation of $z$ by $\pi/2$.
Then the length $D(z)$ of the shortest path joining $z$ to the origin, solution of the eikonal equation \iref{eikonal} on $\Omega := \R^2 \sm \{0\}$, is given by %has the expression
\be
\label{intSqrtG}
D(z) = \int_0^{\|z\|} \sqrt{1-g(r)^2}\, dr.
\ee
\end{prop}

\begin{proof}
Let $z=r\omega$, where $r>0$ and $\|\omega\|=1$, and let $V(z) := g(r) \, \omega^\perp - \sqrt{1-g(r)^2} \,\omega$. We have \begin{equation}
\label{Fzg}
\cF_z(u) + \sqrt{1-g(r)^2} \<\omega,u\> = \|u\| - \<V(z), u\> \geq 0,
 \end{equation} 
 since $\|V(z)\|=1$, with equality if $u$ is positively proportional to $V(z)$. 

Let $\gamma \in C^1([0,1], \R^2)$ be such that $\gamma(0)=z$, $\gamma(1) = 0$. We may assume that $\gamma(t)\neq 0$ for all $t<1$, up to eliminating a loop starting and ending at the origin at the end of the path $\gamma$. For all $t\in [0,1[$, let $r(t) := \|\gamma(t)\|>0$ and let $\omega(t) := \gamma(t)/r(t)$. Note that $\<\omega(t),\gamma'(t)\> = r'(t)$. We obtain using \eqref{Fzg} %our lower estimate on $\cF_z(u)$ %the above estimate %Integrating %Integrating the above lower estimate fwe obtain 
$$
\length(\gamma) := \int_0^1 \cF_{\gamma(t)} (\gamma'(t)) \, dt 
\geq -\int_0^1 \sqrt{1-g(r(t))^2} \, \<\omega(t), \gamma'(t)\> \, dt 
= \int_1^0 \sqrt{1-g(r(t))^2} \, r'(t) \, dt,
$$
with equality if $\gamma'(t)$ is positively proportional to $V(\gamma(t))$ for all $t\in [0,1[$ (a path of minimal length can therefore be obtained by solving an ordinary differential equation). Observing that the right hand side of \iref{intSqrtG} and of the last equation coincide, we obtain the announced result.
%Clearly $\|V(z)\|=1$, and therefore $0 \leq \|u\| - \<V(z), u\> = \cF_z(u) - \sqrt{1-g(r)^2} \<\omega,u\>$, wi
\end{proof}

For our second test case, we chose the Finsler metric $\cF$ given by $g(r) := r/\sqrt{1+r^2}$ in \iref{defFg}, in such way that $D(z) = \arcsinh(\|z\|)$.
%Level lines of $D$ are circles centered on the origin, and paths of minimal length are spirals, see Figure \ref{fig:testCases} (right).
%One easily checks that the $\gamma(t) := (\|z\|-t) R_{-t} (z)$, $t\in [0,\|z\|]$, denoting by $R_t$ the rotation of angle $t$, is a path of minimal length joining an arbitrary point $z\in \R^2$ to the origin. Thus level lines of See Figure
The problem was discretized on a the square domain $Q := [-r_0,r_0]^2$, where $r_0=10$, which contains the disk $B := \{z\in \R^2; \, \|z\| \leq r_0\}$. Due to the allure of the paths of minimal length for the continuous problem, spirals, see Figure \ref{fig:spiraling}, the convergence of the discrete solution towards the continuous one can only be guaranteed within the disk $B$. Grid points that do not belong to $B$ are thus rejected when computing errors. 
The maximum anisotropy ratio on the disk $B$ is $\kappa(\cF_{|B}) = (r_0+\sqrt{1+r_0^2})^2 \simeq 402$, which is quite pronounced, and unsurprisingly the FM-$8$ non-convergence shows early. %was not convergent.
Due to the strong anisotropy, the MAOUM produced huge stencils, leading to a memory footprint incompatible with our equipment. Unlike other test cases, the AGSI produced here the most accurate results \emph{resolution-wise}: for the prescribed tolerance $5 \times 10^{-2}$ on the $L^\infty$ error, the AGSI takes $119s$, at resolution $n=435$, while the FM-ASR takes $10.7s$, at resolution $n=1069$.
Despite the higher resolution, the FM-ASR strongly reduces the CPU time needed to achieve a prescribed $L^\infty$ error bound, here by a factor 11.

Two more test cases, involving Riemannian metrics, are illustrated on Figures \ref{fig:Vlad} and \ref{fig:Cohen}. They are  inspired by seismic imaging and medical image segmentation respectively, and were originally proposed in \cite{SV03} and \cite{BC10}, see also \cite{M12}. The efficiency of the FM-ASR and of the FM-LBR \cite{M12} are comparable, and their superiority over alternative methods is here unquestionable. In the last test these two methods are in a class of their own, often reducing CPU time by four(!) orders of magnitude in comparison with their alternatives, for a target $L^\infty$ error bound.

\begin{figure}
\includegraphics[width=4cm]{\pathPic/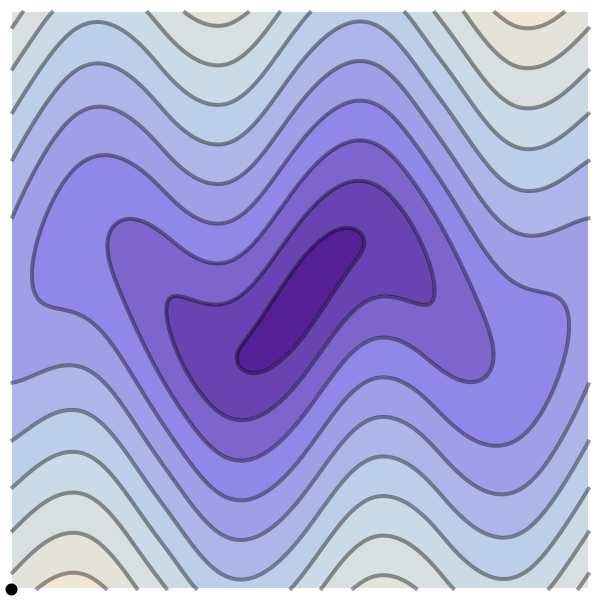}
\includegraphics[width=5.8cm]{\pathPic/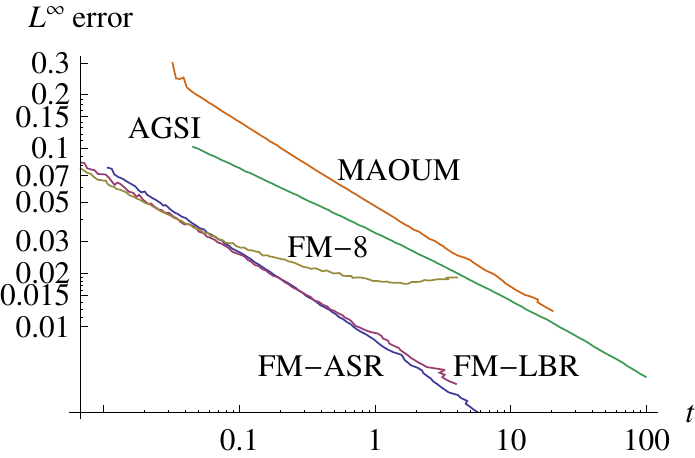}
\includegraphics[width=5.8cm]{\pathPic/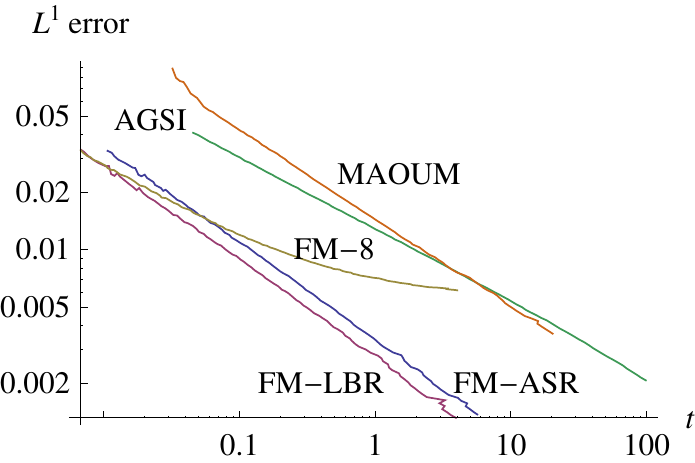}
\caption{
\label{fig:Vlad}
Test case inspired by seismic imaging (taken from \cite{SV03}, Figure 6, top left), with moderate anisotropy $\kappa(\cF)=4$. Riemannian metric $\cF_z(u) := \sqrt{\<u, \cM(z) u\>}$, where $\cM(z)$ has eigenvalues $0.8^{-2},0.2^{-2}$, the former associated to the eigenvector $(1, (\pi/2) \cos (4 \pi x))$. Domain $[-0.5,0.5]^2$.
Target $L^\infty$ error bound $2\times 10^{-2}$ is met by the AGSI in CPU time $4.3s$, at resolution $n=375$, and by the FM-ASR in CPU time $0.18s$ ($24$ times less), at resolution $n=239$.
%Unfolded 239, 0.177
%AGSI 375, 4.29 
%FM8 : 475, 0.4955
}
\end{figure}

\begin{figure}
\includegraphics[width=4cm]{\pathPic/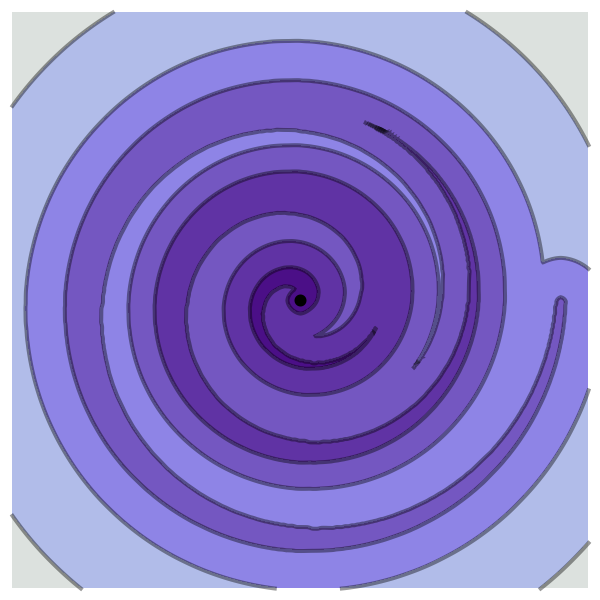}
\includegraphics[width=5.8cm]{\pathPic/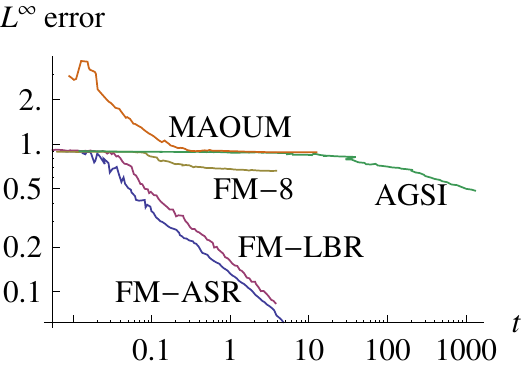}
\includegraphics[width=5.8cm]{\pathPic/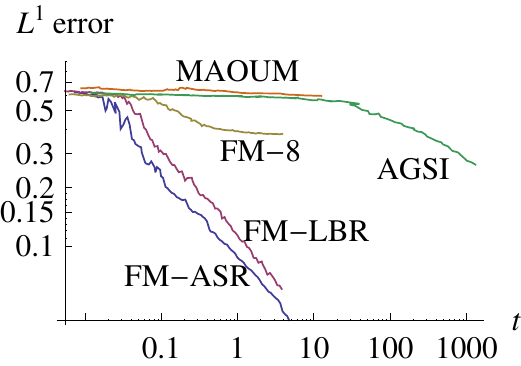}
\caption{
\label{fig:Cohen}
Test case inspired by tubular image segmentation, with strong anisotropy $\kappa(\cF)=100$, Figure 4 in \cite{BC10}. The Riemannian metric is equal to the euclidean norm, except on a band of width $1/100$ along a spiraling curve $\Gamma$ where it has eigenvalues $1/100^2, 1$, the former associated to the tangent vector to $\Gamma$. See \cite{M12} for details.
Target $L^\infty$ error bound $0.5$ is met by the AGSI in CPU time $1015s$, at resolution $n=1135$, and by the FM-ASR in CPU time $0.054s$ ($22700$ times less), at resolution $n=157$.
}
\end{figure}

%As often, performance comes at the price of specialization: in defense of the AGSI, we need to mention that it applies to arbitrary triangulated domains, while our approach requires a grid. Let us finally mention that in the case of \emph{Riemmanian} metrics, on a two \emph{or three} dimensional grid, an eikonal solver of even smaller complexity was presented in our earlier work \cite{M12} by the author: Fast Marching using Lattice Basis Reduction (FM-LBR). Our experiments have only shown subtle differences, hence they are not reproduced here. In essence, \cite{M12} is slightly faster, but the Accuracy/Complexity ratio is generally in favor of the FM-ASR.
%but paths of minimal length are slightly easier to extract with the present method. 
%(Under suitable regularity assumptions, these paths obey the differential equation $\gamma'(t) = \nabla \cF^*_{\gamma(t)}(-\nabla D(\gamma(t)))$. The descent direction, appea
%These paths consist of a gradient descent of the 
%Under suitable regularity assumptions the $\gamma'(t) = \nabla \cF^*_{\gamma(t)}(-\nabla D(\gamma(t)))$

%and the choice of method will depend on the emphasis either on computation time  
%The performance of our algorithm comes at the price of its specialization. 
%In the case of Riemannian metrics, an alternative 
%the alternative choi
 %alternative method w

\section*{Conclusion}

We introduced in this paper a variant of the fast marching algorithm, which applies to arbitrary Finsler metrics, on two dimensional domains discretized on a grid, and which is particularly efficient in the context of large anisotropies. Its complexity depends only (poly-)logarithmically on the anisotropy ratio $\kappa(\cF)$ of the given Finsler metric, in an average sense over grid orientations, whereas earlier methods had a linear or polynomial dependence in this parameter. 
%This complexity analysis is reflected by numerical experiments, which often 
Numerical experiments show a reduction by an order of magnitude, or more, of the CPU time required to meet a target error bound.
%Numerical experiments show, in several occasions, an improvement by an order of magnitude of the accuracy/complexity ratio in comparison with previous methods.
Future work will be devoted to the analysis of the accuracy of this algorithm, its extension to higher dimensions and to triangulated domains, and its application to image analysis.

\appendix

%\begin{center}
%APPENDIX
%\end{center}

\paragraph{Appendix : Proof of Proposition \ref{prop:upwind}.\\}
\noindent
The optimization problem of interest \iref{mint} is the minimization of a continuous convex function on a compact interval, hence there exists at least a minimizer. %it admits at least a solution. We cannot exclude however that a full interval of solutions exists.
Let $G$ be the asymmetric norm defined by $G(x,y) := F(x u+ y v)$, for all $(x,y)\in \R^2$. Let also $D := (d_u,d_v)$ and $\un := (1,1)$.
The problem \iref{mint} is equivalent to
\be
\label{minOmega}
\min \{ \<\omega,D\> + G(\omega);\, \omega \in \R_+^2, \, \<\omega,\un\>=1\}.
\ee
The assumption that $0$ and $1$ are not minimizers of the original problem \iref{mint}, implies that 
%Our assumption states that this
the minimum \iref{minOmega} is not attained when $\omega$ is equal to $e_x := (1,0)$ or $e_y := (0,1)$. %, hence both components of $\omega$ are positive. 
We denote by $\omega$ a minimizer of \iref{minOmega}, and observe that both components of $\omega$ are positive.
The Kuhn-Tucker relations, for this constrained optimization problem, state that there exists a scalar $\lambda\in \R$, the Lagrange multiplier, and an element $V \in \partial G(\omega)$ such that 
\be
\label{LagrangeMult}
D+V = \lambda \un.
\ee
We denoted by $\partial G(\omega)$ the sub-gradient of the convex function $G$ at the point $\omega$; if $G$ is differentiable at $\omega$, then $\partial G(\omega) = \{\nabla G(\omega)\}$. Since $G$ is $1$-homogeneous, we have $\<\omega,V\> = G(\omega)$, by Euler's homogeneous function theorem. %and $G^*(V) = 1$. 
Taking the scalar product of \iref{LagrangeMult} with $\omega$ we obtain
$$
\lambda = \lambda \<\omega, \un\> 
%= \<\omega,D+V\> 
= \<\omega,D\> + \<\omega,V\>
= \<\omega,D\> + G(\omega) = d_w.
$$
Injecting this relation in \iref{LagrangeMult}, we obtain $V = d_w \un - D$. %= (d_w - d_u, \, d_w-d_v)
In order to conclude the proof, we need to show that both $d_w - d_u = \<e_x,V\>$ and $d_w-d_v = \<e_y,V\>$ are positive.
Since the minimum \iref{minOmega} is not attained for $\omega = e_x$, we have
\be
\label{NotAtBounds}
d_u+G(e_x) > d_w 
=  \lambda \< e_x, \, \un\> 
= \<e_x,D\> + \<e_x,V\>
=d_u + \<e_x,V\>,
\ee
thus $G(e_x)>\<e_x,V\>$.
On the other hand, denoting by $(\alpha, \beta)$ the components of $\omega$, and recalling that they are positive, we obtain
\begin{equation}
\label{aevbev}
\alpha \<e_x,V\> +\beta \<e_y, V\> = \<\omega, V\> = G(\omega) = G(\alpha e_x+\beta e_y) \geq \alpha G(e_x),
\end{equation}
 %setting $\omega = \alpha e_x + \beta e_y$, and observing that $\alpha$ and $\beta$ are positive
where for the last inequality we used that $e_x,e_y$ form a $G$-acute angle, since $u,v$ form an $F$-acute angle.
Finally, combining \eqref{NotAtBounds} and \eqref{aevbev} we obtain
$$
G(e_x) > \<e_x,V\> \geq G(e_x) - (\beta/\alpha) \<e_y,V\>,
$$
hence $\<e_y,V\> >0$. Likewise $\<e_x,V\> > 0$, which concludes the proof.

\end{document}

%% file: StdCommands.tex
\def\vp{\varphi}
\def\<{\langle}
\def\>{\rangle}

\def\sm{\setminus}

\def\cE{{\cal E}}
\def\cF{{\cal F}}

\def\cM{{\cal M}}

\def\cO{{\cal O}}
\def\cP{{\cal P}}

\def\cT{{\cal T}}

\def\cZ{{\cal Z}}
\def\Chi{\raise .3ex
\hbox{\large $\chi$}} \def\vp{\varphi}
\def\lsima{\hbox{\kern -.6em\raisebox{-1ex}{$~\stackrel{\textstyle<}{\sim}~$}}\kern -.4em}
\def\lsim{\hbox{\kern -.2em\raisebox{-1ex}{$~\stackrel{\textstyle<}{\sim}~$}}\kern -.2em}
\def\({\Bigl (}
\def\){\Bigr )}
\def\({\Bigl (}
\def\){\Bigr )}

\newcommand{\be}{\begin{equation}}
\newcommand{\ee}{\end{equation}}
\newcommand{\bea}{$$ \begin{array}{lll}}
\newcommand{\eea}{\end{array} $$}
\newcommand{\bi}{\begin{itemize}}
\newcommand{\ei}{\end{itemize}}
\newcommand{\iref}[1]{(\ref{#1})}
\newtheorem{theorem}{Theorem}[section]
\newtheorem{remark}[theorem]{Remark}
\newtheorem{lemma}[theorem]{Lemma}
\newtheorem{proposition}[theorem]{Proposition}
\newtheorem{corollary}[theorem]{Corollary}
\newtheorem{definition}[theorem]{Definition}

\newtheorem{prop}[theorem]{Proposition}

\newtheorem*{remark*}{Remark}

\def\I{{\rm \hbox{I\kern-.2em\hbox{I}}}}
\def\P{{\rm \hbox{I\kern-.2em\hbox{P}}}}
\def\sP{{\rm \hbox{\scriptsize I\kern-.2em\hbox{\scriptsize P}}}}
\def\H{{\rm \hbox{I\kern-.2em\hbox{H}}}}
\def\R{{\rm \hbox{I\kern-.2em\hbox{R}}}}
\def\N{{\rm \hbox{I\kern-.2em\hbox{N}}}}
\def\Z{{\rm {{\rm Z}\kern-.28em{\rm Z}}}}
\def\C{{\rm \hbox{C\kern -.5em {\raise .32ex \hbox{$\scriptscriptstyle
|$}}\kern-.22em{\raise .6ex \hbox{$\scriptscriptstyle |$}}\kern .4em}}}
\def\sC{{\rm \hbox{\scriptsize \rm C\kern -.6em {\raise .4ex \hbox{$\scriptscriptstyle \scriptsize|$}}\kern-.22em{\raise .5ex \hbox{$\scriptscriptstyle \scriptsize |$}}\kern .4em}}}

\def\Ra{\Rightarrow}
\def\ve{\varepsilon}

\def\cO{\mathcal O}

\def\cE{\mathcal E}

\newcommand\trans{\mathrm T}

\def\cM{\mathcal M}

\def\bac{$$\left\{\begin{array}}
\def\eac{\end{array}\right.$$}
\def\beqa{\begin{eqnarray*}}
\def\eeqa{\end{eqnarray*}}

\def\gsim{\hbox{\kern -.2em\raisebox{-1ex}{$~\stackrel{\textstyle>}{\sim}~$}}\kern -.2em}

\DeclareMathOperator\Id{Id}

\DeclareMathOperator\Tr{Tr}

\DeclareMathOperator\interp{{\rm I}}

\DeclareMathOperator\GL{GL}

\newtheorem{question}{Question}
\newtheorem{programme}[question]{Programmation}

\def\bques{\begin{question}}
\def\eques{\end{question}}
\def\bprog{\begin{programme}}
\def\eprog{\end{programme}}

\usepackage{stmaryrd}

\def\II{\mathrm {I\kern-0.1exI}}

\newcommand\stext[1]{\ \text{ #1 } \ }

\def\sR{{\rm \hbox{\scriptsize I\kern-.2em\hbox{\scriptsize R}}}}

\def\sN{{\rm \hbox{\scriptsize I\kern-.2em\hbox{\scriptsize N}}}}
\def\sC{{\rm \hbox{\scriptsize \rm C\kern -.6em {\raise .4ex \hbox{$\scriptscriptstyle \scriptsize|$}}\kern-.22em{\raise .5ex \hbox{$\scriptscriptstyle \scriptsize |$}}\kern .4em}}}

%% file: FMFinsler.bbl
\begin{thebibliography}{99}

\bibitem{AM08} K. Alton, I. M. Mitchell, {\it Fast Marching Methods for Stationary Hamilton-Jacobi Equations
 with Axis-Aligned Anisotropy},
 SIAM Journal of Numerical Analysis, 47:1, pp. 363-385, 2008.

\bibitem{AltonMitchell12}
 K. Alton, I. M. Mitchell,
{\it An Ordered Upwind Method with Precomputed Stencil and Monotone Node Acceptance
for Solving Static Hamilton-Jacobi Equations},
Journal of Scientific Computing, 51:2, pp. 313-348, 2012.


\bibitem{BC10} F. Benmansour, L. D. Cohen, {\it Tubular Structure Segmentation Based on Minimal Path Method and Anisotropic Enhancement}, International Journal of Computer Vision, 92(2), 192-210, 2010.

%\bibitem{PC} Peyr\'e, Carlier, on diffuse geodesics and metric optimization.


\bibitem{BR06} F. Bornemann, C. Rasch, {\it Finite-element Discretization of Static Hamilton-Jacobi Equations based on a Local Variational Principle}, Computing and Visualization in Science, 9(2), 57-69, 2006.

%\bibitem{JBTDPIB08}  S. Jbabdi, P. Bellec, R. Toro, J. Daunizeau, M. P{\'e}l{\'e}grini-Issac, H. Benali, {\it Accurate Anisotropic Fast Marching for Diffusion-Based Geodesic Tractography}, International Journal of Biomedical Imaging, 2008.

\bibitem{KushnerDupuis92} H.J. Kushner, P.G. Dupuis,
{\it Numerical Methods for Stochastic Control Problems
in Continuous Time}, Academic Press, New York, 1992.

\bibitem{L03} Q. Lin, \emph{Enhancement, extraction, and visualization of 3D volume data}, Ph.D. Thesis, Linkopings Universitet, 2003.

\bibitem{L82} P.L. Lions, {\it Generalized solutions of Hamilton-Jacobi equations}, Pitman, Boston, 1982.

%\bibitem{M1896} H. Minkowski, {\it Geometrie der Zahlen}, Teubner, Leipzig-Berlin, 1896, Reprinted: Johnson, New York, 1968.

\bibitem{MPAT08} J. Melonakos, E. Pichon, S. Angenent, A. Tannenbaum, {\it Finsler Active Contours}, IEEE Transactions on Pattern Analysis and Maching Intelligence, 30(3), pp. 412-423, 2008

\bibitem{M12} J.-M. Mirebeau, {\it Anisotropic Fast Marching on Cartesian Grids, using Lattice Basis Reduction}, preprint, 2012.

\bibitem{M12b} J.-M. Mirebeau, {\it On the Accuracy of Anisotropic Fast Marching}, preprint.


\bibitem{OTV09} A. M. Oberman, R. Takei and A. Vladimirsky, {\it Homogenization of Metric Hamilton-Jacobi Equations}, Multiscale Modeling \& Simulation, 8(1), 269-295, 2009

\bibitem{PPKC10}  G. Peyr{\'e}, M. P{\'e}chaud, R. Keriven, L. D. Cohen, {\it Geodesic Methods in Computer Vision and Graphics}, Foundations and Trends in Computer Graphics and Vision, 5(3-4), 197-397, 2010.

\bibitem{RS09} C. Rasch, T. Satzger, {\it Remarks on the $\cO(N)$ Implementation of the Fast Marching Method}, IMA Journal of Numerical Analysis, 29, 806-813, 2009. 

\bibitem{SKD07} M. Sermesant, E. Konukoglu, H. Delingette, {\it An anisotropic multi-front fast marching method for real-time simulation of cardiac electrophysiology}, Proc of Functional Imaging and Modeling of the Heart, 2007.

\bibitem{SethBook96} J.A. Sethian,
{\it Level Set Methods and Fast Marching Methods: Evolving Interfaces in
Computational Geometry, Fluid Mechanics, Computer Vision and Materials
Sciences}, Cambridge University Press, 1996.

\bibitem{S99} J. A. Sethian, {\it Level Set Methods and Fast Marching Methods},
% : Evolving Interfaces in Computational Geometry},
%Fluid Mechanics, Computer Vision and Materials Science
J.A. Sethian, Cambridge University Press, 1999.
%(New edition) Sethian Book.

\bibitem{SV03} J. A. Sethian, A. Vladimirsky, {\it Ordered Upwind Methods for Static Hamilton-Jacobi Equations : Theory and Algorithms}, SIAM Journal of Numerical Analysis, 41(1), 325-363, 2003

\bibitem{TsaiChengOsherZhao03} 
Y.-H.R. Tsai, L.-T. Cheng , S. Osher, and H.-K. Zhao,
{\it Fast sweeping algorithms for a class of Hamilton-Jacobi equations},
SIAM Journal on Numerical Analysis, 41:2, pp.659-672, 2003.

\bibitem{T95} J. N. Tsitsiklis, {\it Efficient algorithms for globally optimal trajectories}, IEEE Transactions on Automatic Control, 40(9), 1528-1538, 1995.

\bibitem{Vlad08} A. Vladimirsky,
{\it Label-setting methods for Multimode Stochastic Shortest Path
problems on graphs}, Mathematics of Operations Research 33(4), pp. 821-838, 2008.

\bibitem{ZSN09} C. Zach, L. Shan, M. Niethammer, {\it Globally Optimal Finsler Active Contours}, vol 5748, pp 552-561, Springer Berlin Heidelberg
%\bibitem{H02} M. Henk, {\it Successive minima and lattice points}, 2002.

\bibitem{Zhao05} H. Zhao, {\it A Fast Sweeping Method for Eikonal Equations}, Mathematics of Computation, 74(250), 603-627, 2005

\end{thebibliography}
